\documentclass[12pt]{amsart}
\usepackage{cases}
\usepackage{amsthm}
\usepackage{amsmath}
\usepackage{amscd}
\usepackage{graphicx}
\usepackage[mathscr]{eucal}
\usepackage[colorlinks,linkcolor=blue,citecolor=blue, pdfstartview=FitH]{hyperref}

\input xy
\xyoption{all} \numberwithin{equation}{section}
\numberwithin{figure}{section}
\begin{document}

\title[Asymptotic expansions, relative Reshetikhin-Turaev, Turaev-Viro invariant]
{On the asymptotic expansion of quantum invariants related to surgeries of Whitehead link I: relative
Reshetikhin-Turaev invariants and the Turaev-Viro invariants at $e^{\frac{2\pi\sqrt{-1}}{N+\frac{1}{2}}}$}
\author[Qingtao Chen and Shengmao Zhu]{Qingtao Chen and
Shengmao Zhu}

\address{Department of Pure Mathematics \\
Xi'an Jiaotong-Liverpool University \\
Suzhou Jiangsu \\
China}
\email{Qingtao.Chen@xjtlu.edu.cn,chenqtao@hotmail.com}

\address{Department of Mathematics \\
Zhejiang Normal University  \\
Jinhua Zhejiang,  321004, China }
\email{szhu@zju.edu.cn}
\begin{abstract} 
In this article, we obtain an asymptotic expansion formula for the relative Reshetikhin-Turaev invariant in the case that the ambient 3-manifold is gained by doing rational surgery along one component of Whitehead link.  In addition, we obtain an asymptotic expansion formula for the  Turaev-Viro invariant of the cusped 3-manifold which is gained by doing rational surgery along one component of the Whitehead link. 
\end{abstract}

\maketitle

\theoremstyle{plain} \newtheorem{thm}{Theorem}[section] \newtheorem{theorem}[%
thm]{Theorem} \newtheorem{lemma}[thm]{Lemma} \newtheorem{corollary}[thm]{%
Corollary} \newtheorem{proposition}[thm]{Proposition} \newtheorem{conjecture}%
[thm]{Conjecture} \theoremstyle{definition}
\newtheorem{remark}[thm]{Remark}
\newtheorem{remarks}[thm]{Remarks} \newtheorem{definition}[thm]{Definition}
\newtheorem{example}[thm]{Example}





\tableofcontents
\newpage

\section{Introduction}
In \cite{CY18},  the first author and Yang proposed a version of volume conjecture for Reshetikhin-Turaev and the Turaev-Viro invariants of a hyperbolic 3-manifold at certain roots of unity.  In \cite{WongYang23}, Wong and Yang proposed the volume conjecture for the relative Reshetikhin-Turaev invariants of a closed oriented 3-manifold with a colored framed link inside it whose asymptotic behavior is related to the volume and Chern-Simons invariant of the hyperbolic cone metric on the manifold with singular locus the link and cone angles determined by the coloring (see Conjecture 1.1 in \cite{WongYang23}). 

This paper contains two parts. First,  we prove an asymptotic expansion formula for the relative Reshetikhin-Turaev invariant in the case that the ambient 3-manifold is obtained by doing rational surgery along one component of whitehead link. Next, we present an asymptotic expansion formula for the  Turaev-Viro invariant of the cusped 3-manifold obtained by doing rational surgery along one component of whitehead link.

In the following, we first focus on the following  special version of Wong-Yang's volume conjecture for relative Reshetikhin-Turaev invariant \cite{WongYang23}.

\begin{conjecture} \label{conjecture}
Let $M$ be a closed oriented 3-manifold and let $\mathcal{L}$ be a framed hyperbolic link with $n$-components in $M$. For an odd integer $r=2N+1$ with $N\geq 1$ and let $\mathbf{N}=(N,\cdots, N)$ be the $n$-tuple $N$'s, then we have
\begin{align}
  \lim_{r\rightarrow \infty}\frac{4\pi}{r}\log RT_{r}(M,\mathcal{L};\mathbf{N})=Vol(M\setminus \mathcal{\mathcal{L}})+\sqrt{-1}CS(M\setminus\mathcal{\mathcal{L}})  \mod \sqrt{-1}\pi^2 \mathbb{Z}.
\end{align}
\end{conjecture}
Note that when $M=S^3$, Conjecture \ref{conjecture} is reduced to the Question 1.7 in \cite{DKY18}.

We consider the Whitehead link $W$ given in Figure \ref{figureWL0} which has two components $L_1$ and $L_2$.
\begin{figure}[!htb] 
\begin{align}
\raisebox{-40pt}{
\includegraphics[width=150pt]{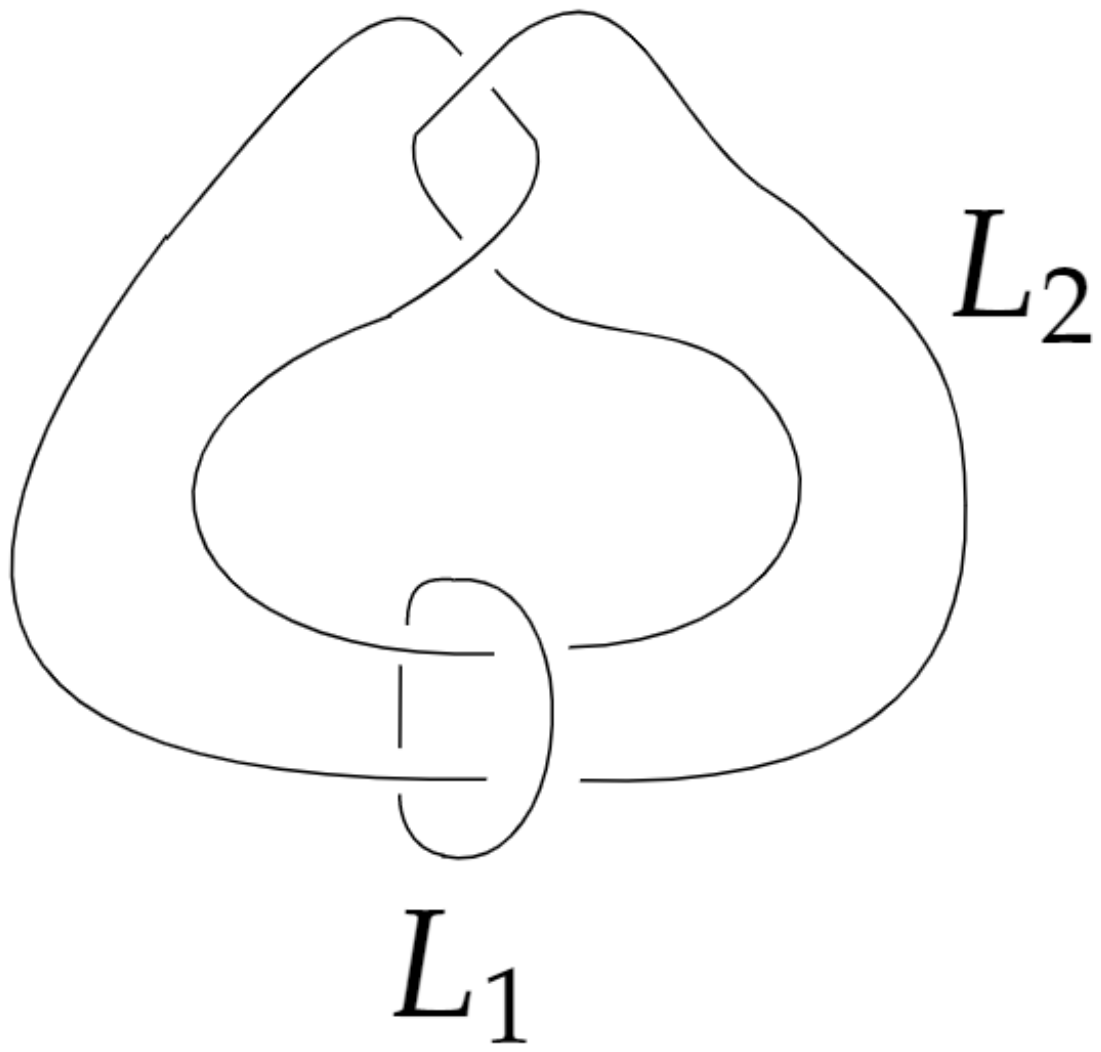}}
\end{align}
\caption{The Whitehead link $W$}
\label{figureWL0} 
\end{figure}
Let $W(p,q)$ be the cusped 3-manifold obtained by doing $\frac{p}{q}$-surgery along the component $L_1$ of the Whitehead link. The $\frac{p}{q}$-surgery along the unknot in $S^3$ gives the lens space $L(p,q)$, so  $W(p,q)$ is the complement of the knot $L_2$ in lens space $L(p,q)$.  For brevity, we let 
$\bar{J}_{N+1}(W(p,q);t)=RT_{r}(L(p,q),L_2;N)$    
be the $r$-th relative Reshetikhin-Turaev invariant of $L_2$ with color $N$ in $L(p,q)$, where $t=e^{\frac{4\pi\sqrt{-1}}{r}}$ and $r=2N+1$. Then  
the normalized relative Reshetikhin-Turaev invariant is defined as $J_{N}(W(p,q);t)=(-1)^{N-1}\frac{\{1\}}{\{N\}}\bar{J}_{N}(W(p,q);t)$. 

   Let $V^+ (p,q;\theta_1,\theta_2)$ be the potential function of the relative Reshetikhin-Turaev invariant given by formula (\ref{formula-V+-0}). 
Let $( \theta_1^0,\theta_2^0)$ be the critical point of $V^{+}(p,q;\theta_1,\theta_2)$. Set $z_i^0=e^{2\pi\sqrt{-1}\theta_i^0}$ for $i=1,2$,  put
\begin{align} \label{formula-zetapq}
  \zeta(p,q)&=
V(p,q;\theta_1^0,\theta_2^0)\\\nonumber
&=\pi\sqrt{-1}\left(\left(\frac{p}{2q}+1\right)(\theta_1^0)^2-\frac{\theta_1^0}{q}+(\theta_2^0)^2-\theta_2^0+\frac{5}{6}+\frac{p'}{2q}\right)\\\nonumber
&+\frac{1}{2\pi\sqrt{-1}}\left(\text{Li}_2(z_2^0z_1^0)+\text{Li}_2\left(\frac{z_2^0}{z_1^0}\right)+3\text{Li}_2(z_2^0)-\text{Li}_2\left((z_2^0)^2\right)\right)
\end{align}
and
\begin{align} \label{formula-omegapq0}
\omega(p,q)=\frac{\sin\left(\frac{\theta_1^0\pi}{q}-J(s^+)\pi\right)}{\sqrt{1-(z_2^0)^2}\sqrt{H(p,q;z_1^0,z_2^0)}},    
\end{align}
where
\begin{align}
&H(p,q;z_1^0,z_2^0)\\\nonumber
&=\left(\frac{p}{2q}+1\right)\left(1+\frac{3z_2^0}{1-z_2^0}-\frac{4(z_2^0)^2}{1-(z_2^0)^2}\right)+\left(\frac{p}{2q}+2+\frac{3z_2^0}{1-z_2^0}-\frac{4(z_2^0)^2}{1-(z_2^0)^2}\right)\\\nonumber
&\cdot\left(\frac{z_1^0z_2^0}{1-z_1^0z_2^0}+\frac{z_2^0}{z_1^0-z_2^0}\right)+\frac{4z_1^0(z_2^0)^2}{(1-z_2^0z_1^0)(z_1^0-z_2^0)}.    
\end{align}

Then, we have
\begin{theorem}  \label{theorem-main}
  For $(p,q)$ lies in the set $S$ which is given by formula (\ref{set-S}),  the asymptotic expansion of the (normalized) relative Reshetikhin-Turaev is given by 
    \begin{align} \label{formula-asymain}
        J_{N}(W(p,q);t)&=C_{N}(p,q)\frac{\sqrt{2N+1}}{\sin\frac{\pi}{2N+1}\sqrt{q}}\omega(p,q)e^{(N+\frac{1}{2})\zeta(p,q)}\left(1+O\left(\frac{1}{N+\frac{1}{2}}\right)\right),
    \end{align}
    where $C_N(p,q)$ is a constant of norm $1$ independent of the geometric structure, and $\zeta(p,q)$ and $\omega(p,q)$ are given by formulas (\ref{formula-zetapq}) and (\ref{formula-omegapq0}).
\end{theorem}
By Proposition \ref{prop-crit=volume}, we know that
\begin{align}
     2\pi \zeta(p,q)= Vol(W(p,q))+\sqrt{-1}CS(W(p,q)) \mod \pi^2\sqrt{-1}\mathbb{Z}.
\end{align}
Theorem \ref{theorem-main} implies that
\begin{corollary}
    For $(p,q)\in S$, we have
    \begin{align}
        \lim_{N\rightarrow \infty}\frac{2\pi}{N}\log J_{N}(W(p,q);t)=Vol(W(p,q))+\sqrt{-1}CS(W(p,q)) \mod \pi^2\sqrt{-1}\mathbb{Z}.
    \end{align}
\end{corollary}
It confirms Conjecture \ref{conjecture} for the relative Reshetikhin-Turaev invariants $RT_{r}(L_{p,q},L_2;N)$.  

We should mention that although we have used the notation $\bar{J}_{N}(W(p,q);t)$ to denote the relative Reshetikhin-Turaev invariant, it is not a topological invariant of the cusped manifold $W(p,q)$. Now we turn to the Turaev-Viro invariant of $W(p,q)$ which is a topological invariant. Let us recall the volume conjecture for Turaev-Viro invariants proposed in \cite{CY18}.  
\begin{conjecture}
    Let $M$ be a hyperbolic 3-manifold, either closed, with cusps, or compact
with totally geodesic boundary. Then as $r$ varies along the odd natural numbers, 
\begin{align}
 \lim_{r\rightarrow \infty} \frac{2\pi}{r} \log |TV_{r}(M,t)|=Vol(M)   
\end{align}
\end{conjecture}

It is well-known that there is a close relationship between Turaev-Viro and Reshetikhin-Turaev invariants \cite{Turaev94,Roberts95,BP96}. Hence the volume conjecture for Turaev-Viro invariant of a closed manifold is a direct consequence of  the volume conjecture for the corresponding  Reshetikhin-Turaev invariant.  Now we turn to the cusped hyperbolic 3-manifolds. When $\mathcal{L}$ is a link in $S^3$ with $n$-components, the following formula was presented by Detcherry-Kalfagianni-Yang \cite{DKY18}, 
\begin{align}
TV_{r}(S^3\setminus \mathcal{L},t)=2^{n-1}\mu_r^2\sum_{m=1}^{N}|\bar{J}_{m}(\mathcal{L},t)|^2.    
\end{align}
Using this formula, they proved the volume conjecture for the Turaev-Viro invariant of link components in $S^3$ when the link $\mathcal{L}$ is equal to the figure-8 knot $4_1$ and Borromean ring. The volume conjecture for the Turaev-Viro invarinat of the complements of any fundamental shadow link in $\#^{c+1}(S^1\times S^2)$ was proved by Belletti-Detcherry-Kalfagianni-Yang \cite{BDKY22}. On the other hand, Wong and Au \cite{WongAu17} obtained an asymptotic expansion for the Turaev-Viro invariants of the figure-8 knot complement, see Theorem 8 in \cite{WongAu17}. 

As to the Turaev-Viro invariant of the cusped manifold $W(p,q)$,  we have the following formula
\begin{align}
TV_{r}(W(p,q),t)=\mu_r^2\sum_{m=1}^{N}|\bar{J}_{m}(W(p,q);t)|^2.    
\end{align}

Using the saddle point method two times similar to \cite{WongAu17}, we obtain the following asymptotic expansion formula.

\begin{theorem} \label{theorem-main2}
When $p\geq 1000$ or $q\geq 1000$, we have
\begin{align} \label{formula-asymTV}
    &TV_{r}(W(p,q),t)\\\nonumber
    &\sim_{N\rightarrow \infty}\Big|\frac{\sin^2\left(\frac{\pi \theta_1^0}{q}-J(s^+)\pi\right)}{q(1-(z_2^0)^2)H(p,q;z_1^0,z_2^0)}\Big|\frac{(N+\frac{1}{2})^{\frac{1}{2}}}{\sqrt{\text{Im}\left(\frac{1}{1-z_2^0}\right)}}e^{\frac{(N+\frac{1}{2})}{\pi}\text{Vol}(W(p,q))}.
\end{align}
\end{theorem}

Theorem \ref{theorem-main2} includes the asymptotics of the Turaev-Viro invariants of twist knots $K_{s}$ with $|s|\geq 1000$.

\begin{remark}
  Although we assume condition $p\geq 1000$ or $q\geq 1000$ in Theorem \ref{theorem-main2} for technical reasons, we believe that the  asymptotic expansion formula (\ref{formula-asymTV}) holds for all $p,q\in \mathbb{Z}$ such that $W(p,q)$ is hyperbolic. For example, when $(p,q)=(1,1)$, we have $W(1,1)=S^3\setminus 4_1$, and the above asymptotic expansion (\ref{formula-asymTV}) recovers Theorem 8 of \cite{WongAu17}. See Example \ref{example-4_1} for details.   
\end{remark}

We remark that the surgery approach to 3-manifold obtained from Dehn filling of Whitehead link was first attempted by us from Jul. 2022 to Feb. 2023. Then we studied Masbaum's formula of colored Jones polynomial to prove the asymptotic expansion formula for twist knots where we developed certain technique such as the 1-dim Saddle Point Method to shrink the integration area to meet the condition of positive definiteness of Hessian etc. and finished the paper \cite{CZ23-1}.
We apply those mature techniques to finish this present paper. Paper \cite{CZ25-2} of asymptotics of relative Reshetikhin-Turaev invariants at roots of unity in the case that the ambient 3-manifold is obtained by doing rational surgery along one component of Whitehead link at root of unity $e^{\frac{2\pi\sqrt{-1}}{N+\frac{1}{2M}}}$ and colored Jones polynomial of twist knots at $e^{\frac{2\pi\sqrt{-1}}{N}}$ (most calculation has been done by late Feb. 2023) and paper \cite{CZ25-3} of asymptotics of the Reshetikhin-Turaev invariants of closed hyperbolic 3-manifolds obtained by doing two rational surgeries along two components of Whitehead link at root of unity $e^{\frac{2\pi\sqrt{-1}}{N+\frac{1}{2}}}$ will be finished later.

The rest of this article is organized as follows. In Section \ref{Section-2RT}, we fix the notation and
review the related materials that will be used in this paper. In Section \ref{Section-Potentialfunction}, we compute
the potential function for the relative Reshetikhin-Turaev invariant $J_{N}(W(p,q),t)$. We follow the approach developed by Wong-Yang \cite{WongYang20-1} which largely simplifies the computations for the relative Reshetikhin-Turaev invariant of the 3-manifold obtained by doing rational surgery. In Section \ref{Section-Poissonsummation}, we express the relative Reshetikhin-Turaev invariant as a summation of Fourier coefficients with the help of Poisson summation formula. The geometric interpretation of the critical point equations and critical value was presented in Section \ref{Section-Critical}.  
In Section \ref{Section-asympticexpansion}, 
we first show that infinite terms of the Fourier coefficients can be neglected. Then we estimate the remained term of Fourier coefficients by using the saddle point method, we obtain that only two main Fourier coefficients will contribute to the final form of the  asymptotic expansion.  Hence we finish the proof of Theorem  \ref{theorem-main}. In Section \ref{Section-TV}, we study the asympotic expansion for the Turaev-Viro invariants $TV_r(W(p,q),t)$ and give the outline of the proof of Theorem \ref{theorem-main2}.  
The final Appendix  Section \ref{Section-Appendices} is devoted to the proof of several results which will be used in previous sections.

\textbf{Acknowledgements.} 

The first author would like to thank Nicolai Reshetikhin, Kefeng Liu and Weiping Zhang for bringing him to this area and a lot of discussions during his career, thank Francis Bonahon,   Giovanni Felder and Shing-Tung Yau for their continuous encouragement, support and discussions, and thank Jun Murakami and Tomotada Ohtsuki for their helpful discussions and support. He also want to thank Jørgen Ellegaard Andersen, Sergei Gukov, Thang Le, Gregor Masbaum,  Rinat Kashaev, Vladimir Turaev, Tian Yang and Hiraku Nakajima for their support, discussions and interests, and thank Yunlong Yao who built him solid analysis foundation twenty years ago. The second author would like to thank Kefeng Liu and Hao Xu for bringing him to this area when he was a graduate student at CMS of Zhejiang University, and for their constant encouragement and helpful discussions since then. Both of the authors thank Ruifeng Qiu for his interests and supports.

\section{Preliminaries}  \label{Section-2RT}

\subsection{Definition of the relative Reshetikhin-Turaev invariants}
We use the skein theory approach  of relative Reshetikhin-Turaev invariants \cite{BHMV92,Lic93} 
following the concise illustration given in \cite{WongYang23}.  We focus on the $SO(3)$ TQFT theory and consider the root of unity $e^{\frac{4\pi\sqrt{-1}}{r}}$, where $r$ is an odd number, we write $r=2N+1$ with $N\geq 1$. 

\begin{definition}
    Let $F$ be an oriented surface, given $A=e^{\frac{\pi\sqrt{-1}}{r}}$.  The Kauffman bracket skein module of $F$, denoted by $K(F)$, is a $\mathbb{C}$-module generated by the isotopic classes of link diagrams in $F$ modulo the submodule generated by the following two  relations:
    
(1) Kauffman bracket skein relation:
\begin{figure}[!htb] 
\begin{align}
\langle\raisebox{-15pt}{
\includegraphics[width=50 pt]{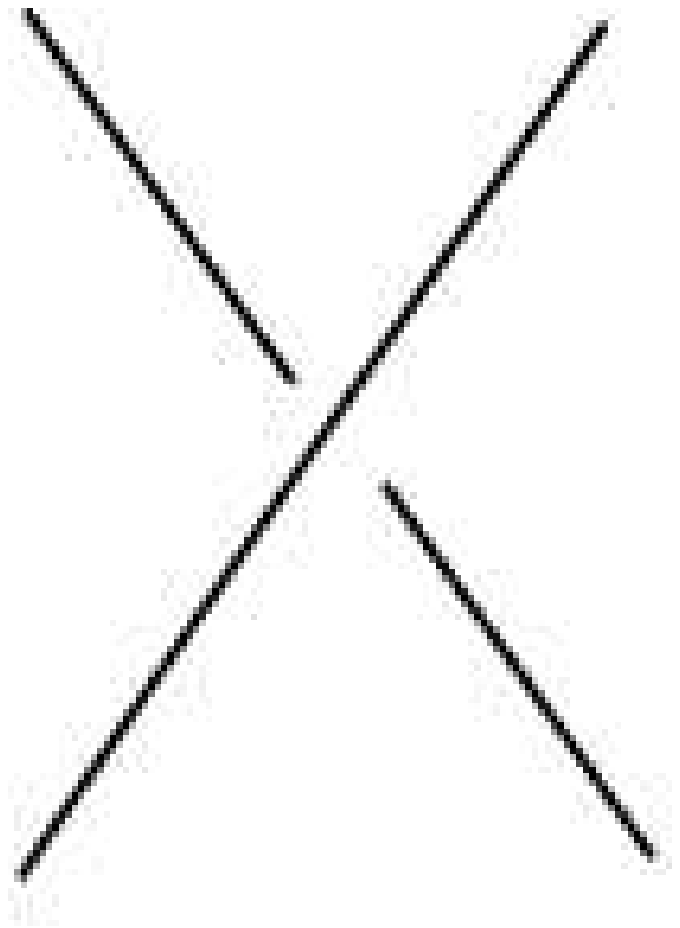}}\rangle=A\langle\raisebox{-15pt}{ \includegraphics[width=50 pt]{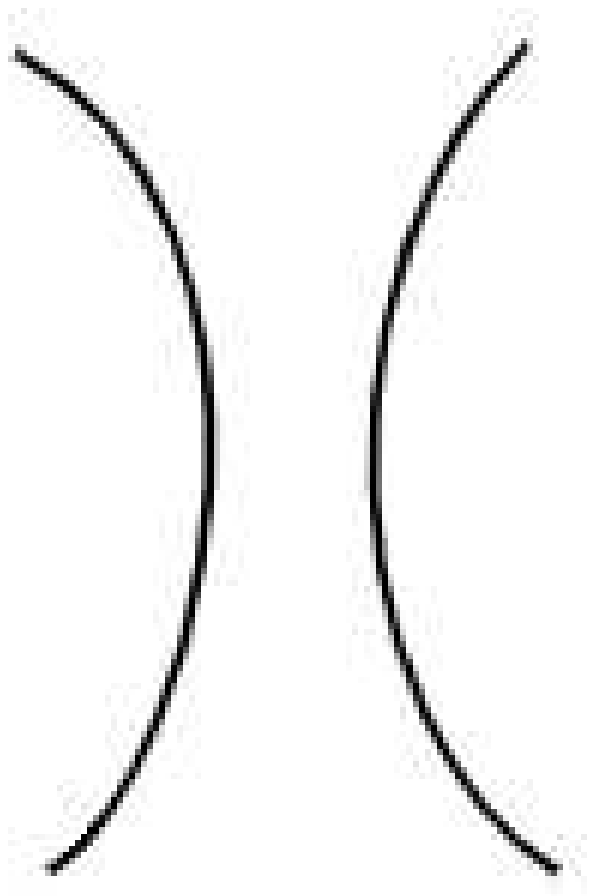}}\rangle+A^{-1}\langle\raisebox{-15pt}{ \includegraphics[width=50 pt]{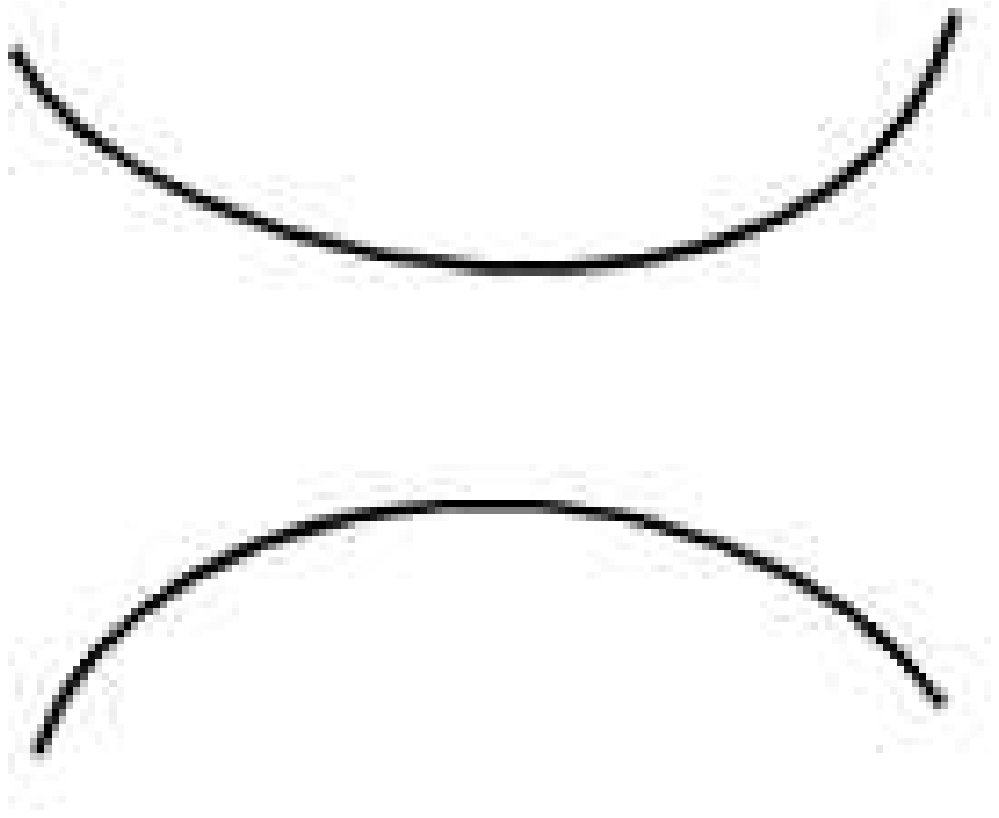}}\rangle  
\end{align}
\end{figure}

(2) Framing relation:
\begin{figure}[!htb] 
$
\langle D\cup \raisebox{-15pt}{
\includegraphics[width=50 pt]{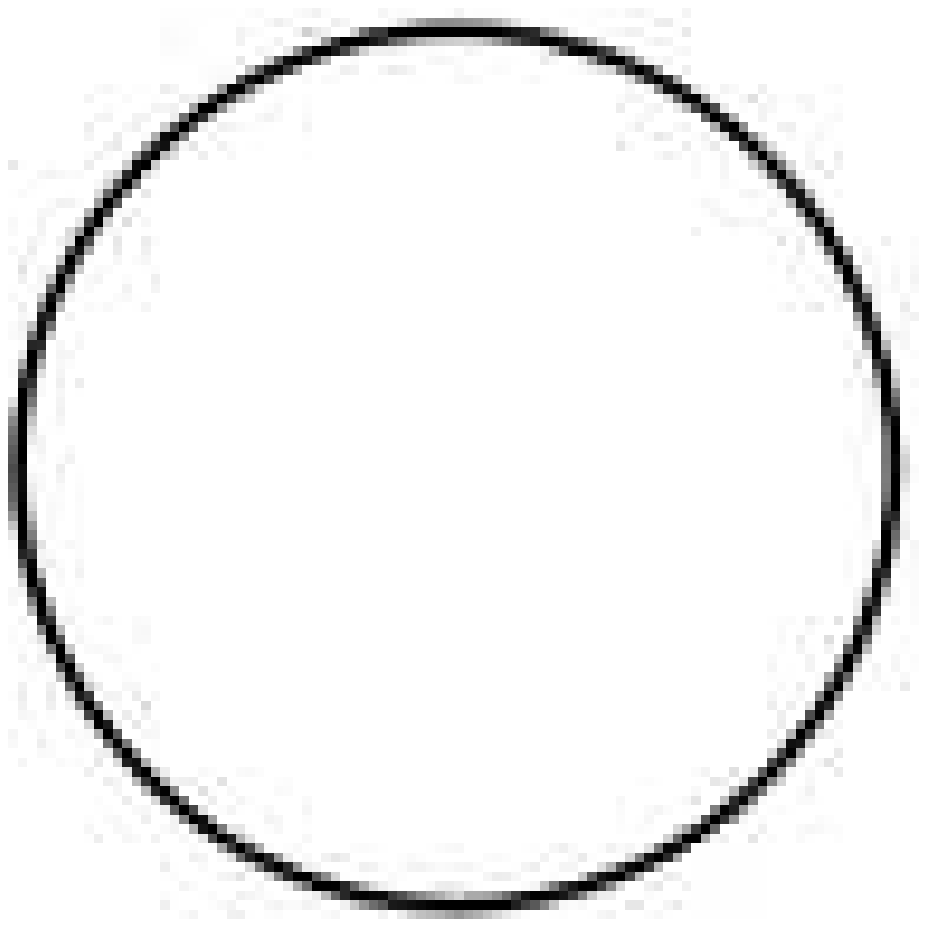}}\rangle=(-A^2-A^{-2})\langle D\rangle.
$
\end{figure}
\end{definition}

When
    $F=S^1\times [0,1]$ is an annulus, set $\mathcal{B}=K(S^1\times [0,1])$. Actually, $\mathcal{B}$ is an algebra, which is called the Kauffman bracket skein algebra of $S^1\times I$. The algebraic structure (i.e. product) of $\mathcal{B}$ is induced by the gluing of two annulus. For any link diagram $D$ in $\mathbb{R}^2$ with $k$ ordered components and $b_1,..,b_k\in \mathcal{B}$, let
    \begin{align}
        \langle b_1,..,b_k\rangle_D
    \end{align}
be the complex number obtained by cabling $b_1,...,b_k$ along the components of $D$ then taking the Kauffman bracket $\langle \ \rangle$.

Note that the skein algebra $\mathcal{B}$ is commutative and the empty diagram is the unit, hence denoted by $1$.
Let 
$z\in \mathcal{B}$ be the core curve of $S^1\times I$ as illustrated in  Figure \ref{figure-z}. 

 \begin{figure}[!htb] 
 \label{figure-z}
 \begin{align*} 
\raisebox{-15pt}{
\includegraphics[width=100 pt]{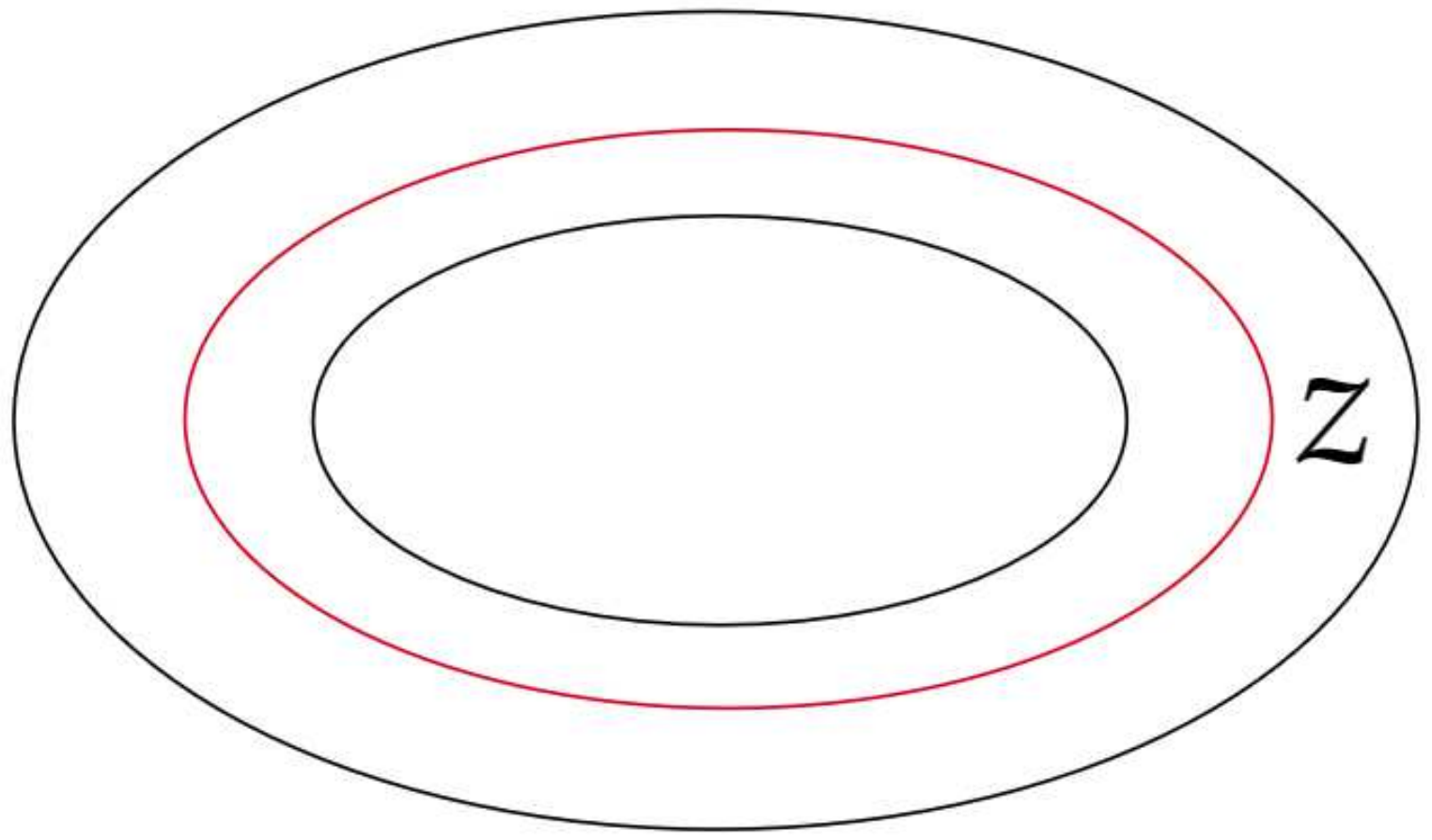}}
\end{align*}
\caption{The core curve $z$}
\end{figure}   
Then  $z^n$ means $n$-parallel copies of $z$.
Moreover, we have $\mathcal{B}=\mathbb{C}[z].$ We define the skein elements $e_n\in \mathcal{B}$ recursively by $e_0=1$, $e_1=z$ and 
$e_n=ze_{n-1}-e_{n-2}$
for $n\geq 2$. The Kirby color $\Omega_r\in \mathcal{B}$ is defined by 
\begin{align}
    \Omega_r=\mu_r\sum_{n=0}^{r-2}(-1)^n[n+1]e_n,
\end{align}
where $\mu_r=\frac{\sin\frac{2\pi}{r}}{\sqrt{r}}$, and $[n]$ is the quantum integer given by 
\begin{align}
   [n]=\frac{t^{\frac{n}{2}}-t^{-\frac{n}{2}}}{t^{\frac{1}{2}}-t^{-\frac{1}{2}}}. 
\end{align}
Note that we fix the convention $t=A^4=e^{\frac{4\pi\sqrt{-1}}{r}}$ throughout this paper.

Let $M$ be a closed oriented 3-manifold and let $\mathcal{L}$ be a framed link in $M$ with $n$ components. Suppose $M$ is obtained from $S^3$ by doing a surgery along a framed link $\mathcal{L}'$, let $D_\mathcal{L}'$ be the standard diagram of $\mathcal{L}'$ which implies that the blackboard framing of $D_{\mathcal{L}'}$ coincides with framing of $\mathcal{L}'$. 
The link $\mathcal{L}$ adds additional components to $D_{\mathcal{L}'}$ and forms a linking diagram $D_{\mathcal{L}\cup\mathcal{L}'}$ with $D_{\mathcal{L}}$ and $D_{\mathcal{L}'}$ linking in possibly a complicated way. Let $U_+$ be the diagram of the unknot with framing $1$, $\sigma(\mathcal{L}')$ be the signature of the linking matrix of $\mathcal{L}'$ and $\mathbf{m}=(m_1,..,m_n)$ be a multi-elements of $I_r=\{0,1,2..,\frac{r-3}{2}\}$. The $r$-th relative Reshetikhin-Turaev invariant of $M$ with $\mathcal{L}$ colored by $\mathbf{m}$ is defined as 
\begin{align} \label{formula-relativeRTformula}
RT_r(M,\mathcal{L};\mathbf{m})=\langle e_{m_1},\cdots,e_{m_n},\Omega_r,...,\Omega_r\rangle_{D_{\mathcal{L}\cup \mathcal{L}'}}\langle\Omega_r\rangle_{U_+}^{-\sigma(\mathcal{L}')}.      
\end{align}

Note that when $\mathcal{L}=\emptyset$ or $m_1=\cdots =m_n=0$, then $RT_r(M,\mathcal{L};\mathbf{0})=RT_r(M)$ the $r$-th Reshetikhin-Turaev invariant of $M$. When $M=S^3$, then $RT_{r}(S^3,\mathcal{L};\mathbf{m})=\bar{J}_{\mathbf{m+1}}(\mathcal{L},e^{\frac{4\pi\sqrt{-1}}{r}})$, the value of the $\mathbf{m+1}$-th unnormalized colored Jones polynomial of $\mathcal{L}$ at $t=e^{\frac{4\pi\sqrt{-1}}{r}}$. We should remark  that the relative Reshetikhin-Turaev invariant defined by (\ref{formula-relativeRTformula}) is different to \cite{WongYang23} with a factor $\mu_r$.

 The relationship between Turaev–Viro
and Reshetikhin–Turaev invariants was studied by Turaev-Walker \cite{Turaev94} and
Roberts \cite{Roberts95} for closed 3-manifolds, and by Benedetti and Petronio \cite{BP96} for 3-manifolds with boundary.  The $SO(3)$-version was treated in \cite{DKY18}.  Given $r=2N+1$ odd, for a link $\mathcal{L}$ in $S^3$ with $n$ components,  they derive a formula for Turaev-Viro invariant of the complement $S^3\setminus \mathcal{L}$ as follows
\begin{align}
TV_{r}(S^3\setminus \mathcal{L},t)=C\mu_r^2\sum_{ \mathbf{m}}|\overline{J}_{\mathbf{m+1}}(\mathcal{L},t)|^2=C\mu_r^2\sum_{ \mathbf{m}}|RT_{r}(S^3,\mathcal{L},\mathbf{m})|^2.
\end{align}
With a slightly modification, one can generalize the above formula to the case that a link $\mathcal{L}$ in a general closed oriented $M$. We have
\begin{align} \label{formula-TV=RT}
TV_{r}(M\setminus \mathcal{L},t)=C\mu_r^2\sum_{ \mathbf{m}}|RT_{r}(M,\mathcal{L},\mathbf{m})|^2,
\end{align}
where the sum is over all multi-elements $\mathbf{m}$ of $I_r$, and  $C$ is a constant \cite{BDKY22}. 
The formula (\ref{formula-TV=RT}) can be regarded as another definition of Turaev-Viro invariant, although we didn't give the original definition of Turaev-Viro invariant here.

\subsection{Dilogarithm and Lobachevsky functions}
Let $\log: \mathbb{C}\setminus (-\infty,0]\rightarrow \mathbb{C}$ be the standard logarithm function defined by 
\begin{align}
    \log z=\log |z|+\sqrt{-1}\arg z
\end{align}
with $-\pi <\arg z<\pi$. 

The dilogarithm function $\text{Li}_2: \mathbb{C}\setminus (1,\infty)\rightarrow \mathbb{C}$ is defined by 
\begin{align}
    \text{Li}_2(z)=-\int_0^{z}\frac{\log(1-x)}{x}dx
\end{align}
where the integral is along any path in $\mathbb{C}\setminus (1,\infty)$ connecting $0$ and $z$, which is holomorphic in $\mathbb{C}\setminus [1,\infty)$ and continuous in $\mathbb{C}\setminus (1,\infty)$. 

The dilogarithm function satisfies the following properties 
\begin{align}
    \text{Li}_2\left(\frac{1}{z}\right)=-\text{Li}_2(z)-\frac{\pi^2}{6}-\frac{1}{2}(\log(-z) )^2.
\end{align}
In the unit disk $\{z\in \mathbb{C}| |z|<1\}$,  $\text{Li}_2(z)=\sum_{n=1}^{\infty}\frac{z^n}{n^2}$, and on the unit circle 
\begin{align}
 \{z=e^{2\pi \sqrt{-1}\theta}|0 \leq \theta\leq 1\},    
\end{align}
we have
\begin{align}
    \text{Li}_2(e^{2\pi\sqrt{-1} \theta})=\frac{\pi^2}{6}+\pi^2\theta(\theta-1)+2\pi \sqrt{-1}\Lambda(\theta)
\end{align}
where 
\begin{align} \label{formula-Lambda(t)}
\Lambda(\theta)=\text{Re}\left(\frac{\text{Li}_2(e^{2\pi \sqrt{-1}\theta})}{2\pi \sqrt{-1}}\right)=-\int_{0}^{\theta}\log|2 \sin \pi \tau|d \theta 
\end{align}
for $\theta\in \mathbb{R}$. The function $\Lambda(\theta)$ is an odd function which has period $1$ and satisfies 
$
\Lambda(1)=\Lambda(\frac{1}{2})=0.
$

Furthermore, we have the follow estimation for the function

\begin{align}
    \text{Re}\left(\frac{1}{2\pi\sqrt{-1}}\text{Li}_2\left(e^{2\pi\sqrt{-1}(t+X\sqrt{-1})}\right)\right)
\end{align}
with $t,X\in \mathbb{R}$.   

\begin{lemma} (see Lemma 2.2 in \cite{OhtYok18}) \label{lemma-Li2}
    Let $\theta$ be a real number with $0<\theta<1$. Then there exists a constant $C>0$ such that 
\begin{align}
    \left\{ \begin{aligned}
         &0  &  \ (\text{if} \ X\geq 0) \\
         &2\pi \left(\theta-\frac{1}{2}\right)X & \ (\text{if} \ X<0)
                          \end{aligned} \right.-C<\text{Re}\left(\frac{1}{2\pi\sqrt{-1}}\text{Li}_2\left(e^{2\pi\sqrt{-1}(\theta+X\sqrt{-1})}\right)\right)
\end{align}
\begin{align*}
    <\left\{ \begin{aligned}
         &0  &  \ (\text{if} \ X\geq 0) \\
         &2\pi \left(\theta-\frac{1}{2}\right)X & \ (\text{if} \ X<0)
                          \end{aligned}\right.+C.
\end{align*}.
\end{lemma}
A key identity used in the proof is 
\begin{align} \label{formula-key}
    \text{Re}\left(\frac{1}{2\pi\sqrt{-1}}\left(\text{Li}_2\left(e^{2\pi\sqrt{-1}(\theta+X\sqrt{-1})}\right)+\text{Li}_2\left(e^{-2\pi\sqrt{-1}(\theta+X\sqrt{-1})}\right)\right)\right)=2\pi(\theta-\frac{1}{2})X.
\end{align}

\subsection{Quantum dilogarithm functions}
Given a positive integer $N$, we introduce the holomorphic function $\varphi_N(\theta)$ for $\{\theta\in
\mathbb{C}| 0< \text{Re}(\theta) < 1\}$ by the following integral
\begin{align}
\varphi_N(\theta)=\int_{-\infty}^{+\infty}\frac{e^{(2\theta-1)x}dx}{4x \sinh x
\sinh\frac{x}{N+\frac{1}{2}}}.
\end{align}
Noting that this integrand has poles at $n\pi \sqrt{-1} (n\in
\mathbb{Z})$, where, to avoid the poles at $0$, we choose the
following contour of the integral
\begin{align}
\gamma=(-\infty,-1]\cup \{z\in \mathbb{C}||z|=1, \text{Im} z\geq 0\}
\cup [1,\infty).
\end{align}

\begin{lemma}  \label{lemma-varphixi}
The function $\varphi_N(\theta)$ satisfies 
\begin{align}
    (t)_n&=\exp \left(\varphi_N\left(\frac{\frac{1}{2}}{N+\frac{1}{2}}\right)-\varphi_N\left(\frac{n+\frac{1}{2}}{N+\frac{1}{2}}\right)\right)   \  \   \left(0\leq n\leq N\right), \\
    (t)_n&=\exp \left(\varphi_N\left(\frac{\frac{1}{2}}{N+\frac{1}{2}}\right)-\varphi_N\left(\frac{n+\frac{1}{2}}{N+\frac{1}{2}}-1\right)+\log 2\right)   \  \   \left(N< n\leq 2N\right).
\end{align}
\end{lemma}

\begin{lemma} \label{lemma-varphixi2}
    We have the following identities:
\begin{align}
    \varphi_N(\theta)+\varphi_N(1-\theta)&=2\pi \sqrt{-1}\left(-\frac{N+\frac{1}{2}}{2}\left(\theta^2-\theta+\frac{1}{6}\right)+\frac{1}{24(N+\frac{1}{2})}\right),\\ 
    \varphi_N\left(\frac{\frac{1}{2}}{N+\frac{1}{2}}\right)&=\frac{N+\frac{1}{2}}{2\pi\sqrt{-1}}\frac{\pi^2}{6}+\frac{1}{2}\log \left(N+\frac{1}{2}\right)+\frac{\pi \sqrt{-1}}{4}-\frac{\pi \sqrt{-1}}{12(N+\frac{1}{2})},\\
    \varphi_N\left(1-\frac{\frac{1}{2}}{N+\frac{1}{2}}\right)&=\frac{N+\frac{1}{2}}{2\pi\sqrt{-1}}\frac{\pi^2}{6}-\frac{1}{2}\log \left(N+\frac{1}{2}\right)+\frac{\pi \sqrt{-1}}{4}-\frac{\pi \sqrt{-1}}{12(N+\frac{1}{2})}.
\end{align}
\end{lemma}
The function $\varphi_N (\theta)$ is closely related to the dilogarithm function as follows.
\begin{lemma} \label{lemma-varphixi3}
    (1) For every $\theta$ with $0<\text{Re}(\theta)<1$, 
    \begin{align}
        \varphi_N(\theta)=\frac{N+\frac{1}{2}}{2\pi \sqrt{-1}}\text{Li}_2(e^{2\pi\sqrt{-1}\theta})
 -\frac{\pi \sqrt{-1}e^{2\pi\sqrt{-1}z}}{6(1-e^{2\pi\sqrt{-1}z})}\frac{1}{2N+1}+O\left(\frac{1}{(N+\frac{1}{2})^3}\right).
    \end{align}
    (2) For every $\theta$ with $0<\text{Re}(\theta)<1$, 
    \begin{align}
        \varphi_N'(\theta)=-\left(N+\frac{1}{2}\right)\log(1-e^{2\pi\sqrt{-1}\theta})+O\left(\frac{1}{N+\frac{1}{2}}\right)
    \end{align}
    (3) As $N\rightarrow \infty$, $\frac{1}{N+\frac{1}{2}}\varphi_N(\theta)$ uniformly converges to $\frac{1}{2\pi\sqrt{-1}}\text{Li}_2(e^{2\pi\sqrt{-1}\theta})$ and $\frac{1}{N+\frac{1}{2}}\varphi'_N(\theta)$ uniformly converges to $-\log(1-e^{2\pi\sqrt{-1}\theta})$ on any compact subset of $\{\theta\in \mathbb{C}|0<\text{Re}(\theta)<1\}$. 
\end{lemma}
See the literature, such as \cite{Oht16,CJ17,WongYang20-1} for the proof of Lemma \ref{lemma-varphixi}, \ref{lemma-varphixi2}, \ref{lemma-varphixi3}.

\section{Computations of the relative Reshetikhin-Turaev invariant } \label{Section-Potentialfunction} 
In this section, we compute the potential for the relative Reshetikhin-Turaev invariants $J_{N}(W(p,q),t)$.  As a preparation, we present some basic results about the continued fractions obtained in \cite{WongYang20-1}.

\subsection{Continued fractions}
For a pair of relatively prime integers $(p,q)$, let 
\begin{align}
    \frac{p}{q}=b_l-\frac{1}{b_{l-1}-\frac{1}{...-\frac{1}{b_1}}}
\end{align}
be a continued fraction. Let $A_0=\pm 1$ and $C_0=0$, and for $i\in \{1,...,l\}$, we define 
\begin{align}
    A_i=b_iA_{i-1}-C_{i-1}
\end{align}
$C_i=A_{i-1}$, then we have $\frac{A_l}{C_l}=b_l-\frac{1}{b_{l-1}-\frac{1}{...-\frac{1}{b_1}}}=\frac{p}{q}$ by induction. 
\begin{remark}
We choose $A_0=1$ or $-1$ to make sure that $A_l=p$ and $C_l=q$. In the following, we assume $A_l=p$ and $C_l=q$.
\end{remark}

Choose $p',q'\in \mathbb{Z}$ such that $p'p+q'q=1$. For $i\in \{1,...,l\}$, we introduce the quantity 

\begin{align}
    K_i=\frac{(-1)^i\sum_{j=1}^ib_jC_j}{C_i}.
\end{align}

\begin{lemma}[Lemma 3.3 in \cite{WongYang20-1}] \label{Lemma-(a)}

(a) Let $I: \{0,1,...,|q|-1\}\rightarrow \{0,...,2|q|-1\}$ be the map defined by 
    \begin{align}
        I(s)=-C_{l-1}(2s+1+K_{l-1}) \ \mod 2|q|.
    \end{align}
Then $I$ is injective with image the set of integers in $\{0,1,..,2|q|-1\}$ with parity that of $1-q$. In particular, there exist a unique $s^+\in \{0,..,|q-1|\}$ and a unique integer $m^+$ such that
\begin{align}
    I(s^+)=1-q+2m^+q,
\end{align}
and a unique $s^-\in \{0,...,|q|-1\}$ and a unique integer $m^-$ such that
\begin{align}
    I(s^-)=-1-q+2m^-q.
\end{align}
Furthermore, $s^+-s^-\equiv p' \mod q$

(b) Let $J: \{0,1,...,|q|-1\}\rightarrow \mathbb{Q}$ be the map defined by 
\begin{align}
    J(s)=\frac{2s+1}{q}+(-1)^l\sum_{i=1}^{l-1}\frac{(-1)^{i+1}K_i}{C_{i+1}}.
\end{align}
Then for $s^+$ and $s^-$ given in (a), we have
\begin{align}
    J(s^+)&\equiv \frac{p'}{q}  \mod \mathbb{Z}, \\ 
    J(s^-)&\equiv -\frac{p'}{q} \mod \mathbb{Z},
\end{align}
and 
$J(s^+)\equiv -J(s^-) \mod 2\mathbb{Z}$.

(c) Let $K:\{0,1,...,|q|-1\}\rightarrow \mathbb{Q}$ be the map defined by 
\begin{align}
    K(s)=\frac{C_{l-1}(2s+1+K_{l-1})^2}{q}+\sum_{i=1}^{l-2}\frac{C_{i}K_i^2}{C_{i+1}}.
\end{align}
Then for $s^+$ and $s^-$ given in (a), we have 
\begin{align}
    K(s^+)&\equiv -\frac{p'}{q}  \mod \mathbb{Z}, \\ 
    K(s^-)&\equiv -\frac{p'}{q} \mod \mathbb{Z}.
\end{align}
\end{lemma}

\begin{lemma} \label{lemma-s(n)}
For $1\leq n_l\leq 2N$, we let
 \begin{align}
    S(n_l)=\sum_{n_1,...,n_{l-1}=1}^{r-1}(-1)^{\sum_{j=1}^{l-1}b_jn_j}q^{\sum_{j=1}^{l-1}
    \frac{b_jn_j^{2}}{4}}\{n_1\}\prod_{j=1}^{l-1}\{n_jn_{j+1}\},
\end{align}
then we have 
\begin{align}
    S(n_l)=\tau'\sum_{s=0}^{|q|-1}e^{-\frac{\pi \sqrt{-1}}{r}\frac{C_{l-1}}{q}\left(n_l+sr+\frac{K_{l-1}r}{2}\right)^2}
    \sin \left(-\pi\left(\frac{(-1)^l}{rq}(2n_l+2sr)+\sum_{i=1}^{l-1}\frac{(-1)^iK_i}{C_{i+1}}\right)\right)
\end{align}
where 
\begin{align}
    \tau'=\frac{(-1)^{\frac{l+1}{4}}}{\sqrt{q}}2^lr^{\frac{l-1}{2}}e^{-\frac{\pi \sqrt{-1}}{r}\sum_{i=1}^{l-1}\frac{1}{C_iC_{i+1}}-\frac{\pi \sqrt{-1}r}{4}\sum_{i=1}^{l-2}\frac{C_iK_i^2}{C_{i+1}}}.
\end{align}
\end{lemma}
\begin{proof}
See Lemma 3.6 of \cite{WongYang20-1}  and the proof of Proposition 3.4 of \cite{WongYang20-1}.  
\end{proof}

\subsection{Computations of  $J_{N}(W(p,q);t)$}
Let $W\left(p,q\right)$ be the cusped 3-manifold obtained by doing $\frac{p}{q}$-surgery along the component $L_1$ of the Whitehead link as shown in Figure \ref{figureWL}.  

\begin{figure}[!htb]
\raisebox{-40pt}{
\includegraphics[width=150pt]{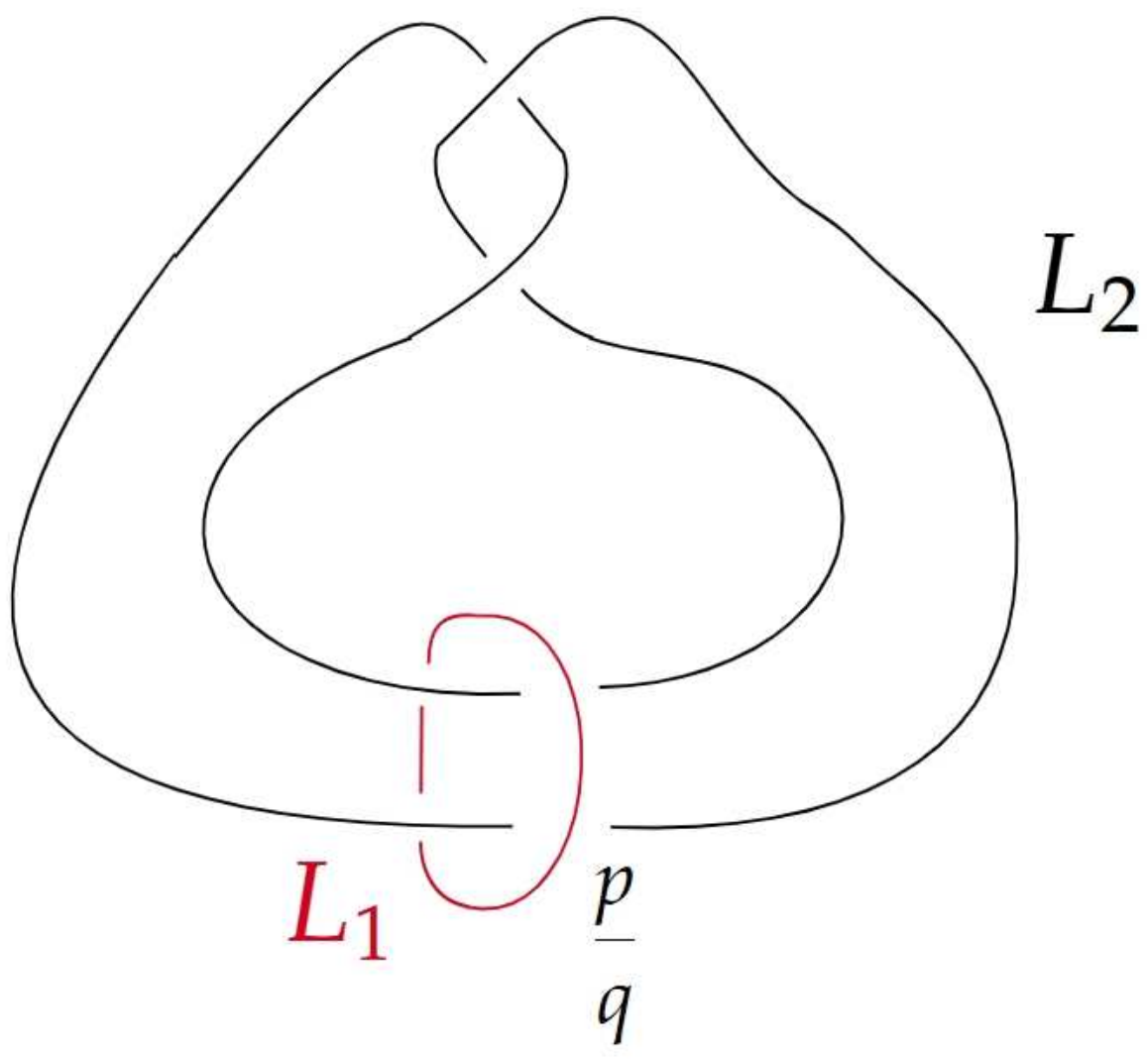}}
\caption{Doing $\frac{p}{q}$-surgery along the component $L_1$}
\label{figureWL} 
\end{figure}

Since the $\frac{p}{q}$-surgery along the unknot in $S^3$ gives the lens space $L(p,q)$, so  $W\left(p,q\right)$ is the complement of the knot $L_2$ in lens space $L(p,q)$. From \cite{Rolfsen90}, we know that doing a $\frac{p}{q}$-surgery along the component $L_1$ is equivalent to doing a surgery along a framed link $L'_1$ of $l$-components with framings $b_1,....,b_l$ from the continued fraction 
\begin{align}
\frac{p}{q}=b_l-\frac{1}{b_{l-1}-\frac{1}{\cdots -\frac{1}{b_1}}}
\end{align}
as shown in Figure \ref{figurefraction}

\begin{figure}[!htb] 
\raisebox{-40pt}{
\includegraphics[width=180pt]{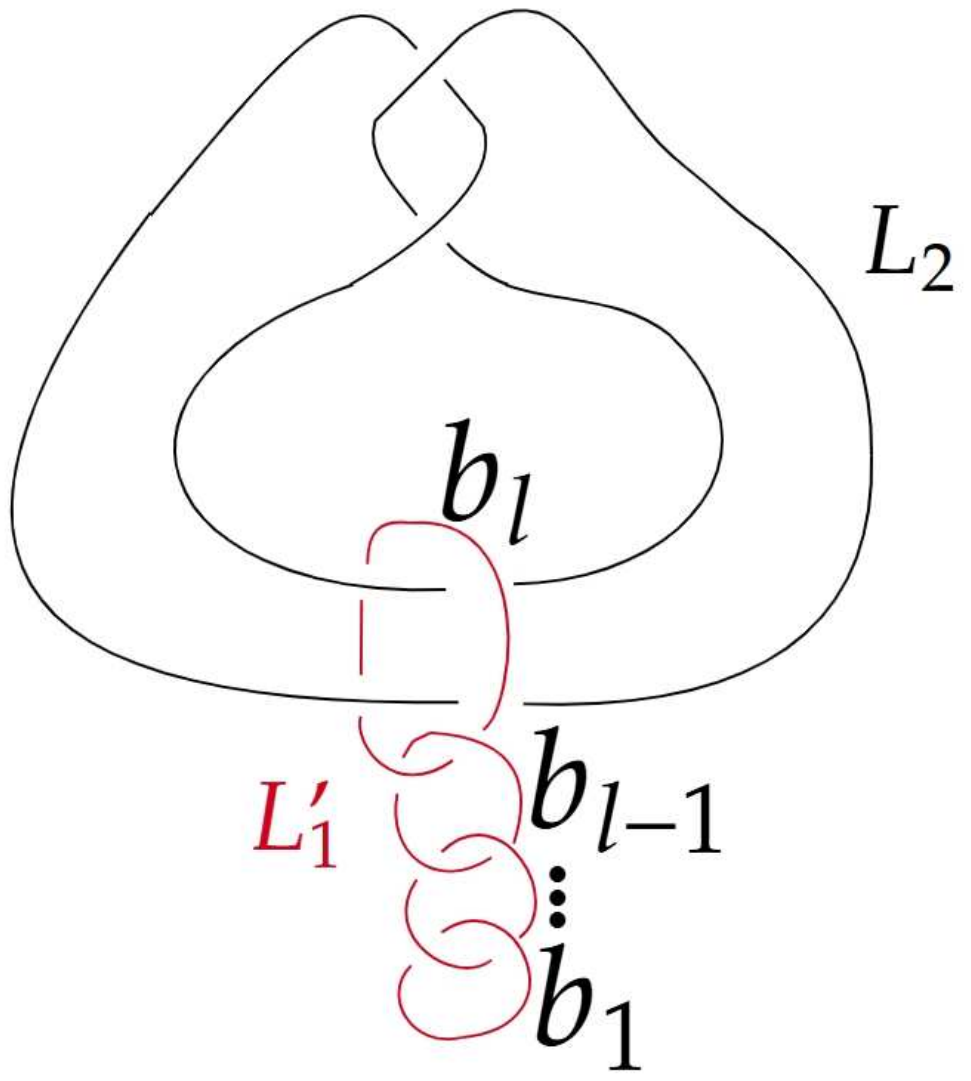}}
\caption{Doing integral surgery along $L'_1$}
\label{figurefraction} 
\end{figure}

A direct computation shows that 
\begin{align}
\langle \Omega_r\rangle_{U_+}=e^{-\left(\frac{3}{r}+\frac{r+1}{4}\right)\pi\sqrt{-1}}.    
\end{align}
Let $\sigma=\sigma(L'_1)$ be the signature of the linking matrix of the framing link $L'_1$.

Then, by using formula (\ref{formula-relativeRTformula}), 
we obtain
\begin{align}
    &RT_r(L(p,q),L_2;N-1)\\\nonumber
    &=\langle e_{N-1},\Omega_r,...,\Omega_r\rangle_{D_{L_2\cup L'_1}}\langle\Omega_r\rangle_{U_+}^{-\sigma(L'_1)}\\\nonumber
    &=\left(\frac{\sin \frac{2\pi }{r}}{\sqrt{r}}\right)^{l}e^{\sigma\left(\frac{3}{r}+\frac{r+1}{4}\right)\pi \sqrt{-1}}\sum_{n_1,...n_l=0}^{r-2}(-1)^{n_l+\sum_{j=1}^lb_jn_j}t^{\sum_{j=1}^{l}\frac{b_jn_j(n_j+2)}{4}}\\\nonumber
    &\cdot[n_1+1]\prod_{j=1}^{l-1}[(n_j+1)(n_{j+1}+1)]\langle e_{N-1},e_{n_l}\rangle_{W},
\end{align}
where the second equality comes from the fact that $e_n$ is an eigenvector of the positive and the negative twist operator with eigenvalue $(-1)^nt^{\pm \frac{n(n+2)}{4}}$, and also is an eigenvector of the circle operator $c(e_m)$ (defined by enclosing $e_n$ by $e_m$) with eigenvalue $(-1)^m\frac{[(n+1)(m+1)]}{[n+1]}$.

By Habiro's formula \cite{Hab00}, we have  
\begin{align}
    \langle e_{N-1},e_{n_l}\rangle_{W}=(-1)^{N-1+n_l}\sum_{i=0}^{\min(N-1,n_l)}(-1)^i\frac{t^{\frac{1}{4}i(i+3)}}{\{1\}}\frac{\{N+i\}!\{n_l+i+1\}!\{i\}!}{\{N-1-i\}!\{n_l-i\}!\{2i+1\}!}.
\end{align}

Note that $W\left(-1,p\right)$ is homeomorphic to the complement $S^3\setminus \mathcal{K}_p$ of the twist knot $\mathcal{K}_p$ in $S^3$. The relative Reshetikhin-Turaev invariant is equal to the (unnormalized) colored Jones polynomial of twist knot $\mathcal{K}_p$ at $t$, i.e.  $RT_r(L(-1,p),L_2;N)=\mu_r\bar{J}_{N+1}(\mathcal{K}_p;t)$.  That's why we introduce the notation
\begin{align}
\bar{J}_{N}(W(p,q);t)&=RT_{r}(L(p,q),L_2;N-1).
\end{align}
 Moreover, we let
\begin{align}
J_{N}(W(p,q);t)&=(-1)^{N-1}\frac{\{1\}}{\{N\}}RT_{r}(L(p,q),L_2;N-1)\\\nonumber
&=(-1)^{N-1}\frac{\{1\}}{\{N\}}\bar{J}_{N}(W(p,q);t)
\end{align}
and we call $J_{N}(W(p,q);t)$  the normalized relative Reshetikhin-Turaev invariant.

Hence
\begin{align}  \label{formula-relativeRT0}
    &J_{N}(W(p,q);t)\\\nonumber
    &=\frac{1}{\{N\}}\left(\frac{\sin \frac{2\pi }{r}}{\sqrt{r}}\right)^le^{\sigma\left(\frac{3}{r}+\frac{r+1}{4}\right)\pi \sqrt{-1}}\sum_{n_1,...n_l=0}^{r-2}(-1)^{\sum_{j=1}^lb_jn_j}t^{\sum_{j=1}^l\frac{b_jn_j(n_j+2)}{4}}[n_1+1]\\\nonumber
    &\cdot\prod_{j=1}^{l-1}[(n_j+1)(n_{j+1}+1)]\sum_{i=0}^{\min(N-1,n_l)}(-1)^it^{\frac{1}{4}i(i+3)}\frac{\{N+i\}!\{n_l+i+1\}!\{i\}!}{\{N-1-i\}!\{n_l-i\}!\{2i+1\}!}\\\nonumber
    &=\frac{1}{\{N\}}\left(\frac{\sin \frac{2\pi }{r}}{\sqrt{r}}\right)^le^{\sigma\left(\frac{3}{r}+\frac{r+1}{4}\right)\pi \sqrt{-1}}\sum_{n'_1,...n'_l=1}^{r-1}(-1)^{\sum_{j=1}^lb_j(n'_j-1)}t^{\sum_{j=1}^l\frac{b_j(n'^2_j-1))}{4}}\\\nonumber
    &\cdot[n'_1]\prod_{j=1}^{l-1}[n'_jn'_{j+1}]\sum_{i=0}^{\min(N-1,n'_l-1)}(-1)^it^{\frac{1}{4}i(i+3)}\frac{\{N+i\}!\{n'_l+i\}!\{i\}!}{\{N-1-i\}!\{n'_l-i-1\}!\{2i+1\}!},
\end{align}
where the last equality is obtained by changing the variables $n_i$ to $n'_i=n_i+1$ for $i\in \{1,..,l\}$. 

Let  \begin{align}
    \kappa'_N=\left(\frac{\sin \frac{2\pi}{r}}{\sqrt{r}}\right)^{l+1}e^{\sigma\left(\frac{3}{r}+\frac{r+1}{4}\right)\pi \sqrt{-1}},
\end{align}
and reordering the summations, we obtain 
\begin{align}
    J_{N}(W(p,q);t)&=\frac{2\sqrt{-1}\sqrt{r}}{\{N\}\{1\}^{l+1}}\kappa'_N\sum_{i=0}^{\min(N-1,n'_l-1)}\sum_{n'_l=1}^{r-1}S(n'_l)t^{\frac{1}{4}i(i+3)-\sum_{j=1}^l\frac{b_j}{4}}(-1)^{i-\sum_{j=1}^lb_j}\\\nonumber
    &\cdot(-1)^{b_ln'_l}t^{\frac{b_l(n'_l)^{2}}{2}}\frac{\{N+i\}!\{n'_l+i\}!\{i\}!}{\{N-1-i\}!\{n'_l-i-1\}!\{2i+1\}!}\\\nonumber
    &=-\kappa''_N\frac{\sqrt{r}}{\sin \frac{\pi}{r}}\sum_{n'_l=1}^{r-1}\sum_{i=0}^{\min(N-1,n'_l-1)}(-1)^it^{\frac{1}{4}i(i+3)}S(n'_l)(-1)^{b_ln'_l}t^{\frac{b_l (n'_l)^2}{4}}\\\nonumber
    &\cdot\frac{\{N+i\}!\{n'_l+i\}!\{i\}!}{\{N-1-i\}!\{n'_l-i-1\}!\{2i+1\}!},
\end{align}
where 
\begin{align}
\kappa''_N=\frac{(-1)^{\sum_{j=1}^lb_j}t^{-\sum_{j=1}^l\frac{b_j}{4}}}{\{1\}^{l+1}}\kappa'_N.
\end{align}
By Lemma \ref{lemma-s(n)}, we obtain
\begin{align}
    &J_{N}(W(p,q);t)\\\nonumber
    &=-\kappa''_N\tau'\frac{\sqrt{r}}{\sin \frac{\pi}{r}}\sum_{s=0}^{|q|-1}\sum_{n'_l=1}^{r-1}\sum_{i=0}^{\min(N-1,n'_l-1)}(-1)^it^{\frac{1}{4}i(i+3)}e^{-\frac{\pi \sqrt{-1}}{r}\frac{C_{l-1}}{q}\left(n'_l+sr+\frac{K_{l-1}r}{2}\right)^2}\\\nonumber
    &
    \cdot\sin \left(-\pi\left(\frac{(-1)^l}{rq}(2n'_l+2sr)+\sum_{i=1}^{l-1}\frac{(-1)^iK_i}{C_{i+1}}\right)\right)(-1)^{b_ln'_l}t^{\frac{b_ln'^{2}_l}{4}}\\\nonumber
    &\cdot\frac{\{N+i\}!\{n'_l+i\}!\{i\}!}{\{N-1-i\}!\{n'_l-i-1\}!\{2i+1\}!}\\\nonumber 
    &=-\kappa''_N\tau'\frac{\sqrt{r}}{\sin \frac{\pi}{r}}\sum_{s=0}^{|q|-1}\sum_{i=0}^{N-1}\sum_{n'_l=i+1}^{r-1-i}(-1)^it^{\frac{1}{4}i(i+3)}e^{-\frac{\pi \sqrt{-1}}{r}\frac{C_{l-1}}{q}\left(n'_l+sr+\frac{K_{l-1}r}{2}\right)^2}\\\nonumber
    &
    \cdot\sin \left(-\pi\left(\frac{(-1)^l}{rq}(2n'_l+2sr)+\sum_{i=1}^{l-1}\frac{(-1)^iK_i}
    {C_{i+1}}\right)\right)(-1)^{b_ln'_l}t^{\frac{b_ln'^{2}_l}{4}}\\\nonumber
    &\cdot\frac{\{N+i\}!\{n'_l+i\}!\{i\}!}{\{N-1-i\}!\{n'_l-i-1\}!\{2i+1\}!}.
\end{align}
We set
\begin{align}
    n'=\frac{r}{2}-n'_l, i'=\frac{r}{2}-1-i.
\end{align}
Then we have
\begin{align}
    &J_{N}(W(p,q);t)\\\nonumber
    &=-(-1)^{b_l(\frac{3}{2}N+\frac{3}{4})+l}e^{\frac{\pi\sqrt{-1}r}{4}\sum_{j=1}^{l-2}\frac{C_jK_j^2}{C_{j+1}}}\kappa''_N\tau'\frac{\sqrt{r}}{\sin \frac{\pi}{r}}\sum_{s=0}^{|q|-1}\sum_{i'=\frac{1}{2}}^{N-\frac{1}{2}}\sum_{n'=-i'}^{n'=i'} e^{\frac{\pi \sqrt{-1}}{r}(i'^2-i'-2)}\\\nonumber
    &\cdot(-1)^{-2b_ln'-2i'}e^{\frac{2\pi \sqrt{-1}}{r}\frac{b_ln'^2}{2}}e^{-\frac{\pi \sqrt{-1}}{r}\frac{C_{l-1}}{q}(-n'+(2s+1+K_{l-1})\frac{r}{2})^2-\frac{\pi \sqrt{-1}r}{4}\sum_{j=1}^{l-2}\frac{C_jK_j^2}{C_{j+1}}}\\\nonumber
    &\cdot\sin\left(\frac{2\pi n'}{rq}-\pi \left(\frac{2s+1}{q}+(-1)^l\sum_{j=1}^{l-1}\frac{(-1)^{j+1}K_j}{C_{j+1}}\right)\right)\\\nonumber
    &\cdot\frac{\{2N-i'-\frac{1}{2}\}!\{2N-n'-i'\}!\{N-i'-\frac{1}{2}\}!}{\{i'-\frac{1}{2}\}!\{i'-n'\}!\{2N-2i'\}!}.
\end{align}

By straightforward computations, we have
\begin{align}
    &e^{-\frac{\pi \sqrt{-1}}{r}\frac{C_{l-1}}{q}(-n'+(2s+1+K_{l-1})\frac{r}{2})^2-\frac{\pi \sqrt{-1}r}{4}\sum_{j=1}^{l-2}\frac{C_jK_j^2}{C_{j+1}}}\\\nonumber
    &=e^{\frac{r}{4\pi \sqrt{-1}}\left(\frac{C_{l-1}}{q}\left(\frac{2\pi n'}{r}\right)^2-\frac{2\pi C_{l-1}(2s+1+K_{l-1})}{q}\frac{2\pi n'}{r}+\pi^2 K(s)\right)},
\end{align}
where
\begin{align}
    K(s)=\frac{C_{l-1}(2s+1+K_{l-1})^2}{q}+\sum_{j=1}^{l-2}\frac{C_{j}K_j^2}{C_{j+1}},
\end{align}
and
\begin{align}
    &e^{\frac{r}{4\pi \sqrt{-1}}\left(-\frac{2\pi C_{l-1}(2s+1+K_{l-1})}{q}\frac{2\pi n'}{r}\right)}\\\nonumber
    &=e^{\frac{r}{4\pi \sqrt{-1}}\left(\frac{2\pi (I(s)+2P(s)|q|)\frac{2\pi n'}{r}}{q}\right)}\\\nonumber
    &=e^{\frac{r}{4\pi \sqrt{-1}}\left(\frac{2\pi I(s)}{q}\frac{2\pi n'}{r}\right)}e^{-2\pi \sqrt{-1}P(s)\left(\frac{|q|}{q}n'\right)}=e^{\frac{r}{4\pi \sqrt{-1}}\left(\frac{2\pi I(s)}{q}\frac{2\pi n'}{r}\right)}(-1)^{P(s)},
\end{align}
where $P(s)$ is an integer determined by 
\begin{align}
    I(s)+2P(s)|q|=-C_{l-1}(2s+1+K_{l-1})
\end{align}
as shown in Lemma \ref{Lemma-(a)}.

Moreover, we let
\begin{align}
    J(s)=\frac{2s+1}{q}+(-1)^l\sum_{j=1}^{l-1}\frac{(-1)^{j+1}K_j}{C_{j+1}}
\end{align}
and 
\begin{align}
    \kappa_N&=(-1)^{b_l(\frac{3}{2}N+\frac{3}{4})+l}e^{\frac{\pi\sqrt{-1}r}{4}\sum_{j=1}^{l-2}\frac{C_jK_j^2}{C_{j+1}}}\kappa''_N\tau'\\\nonumber
    &=(-1)^{b_l(\frac{3}{2}N+\frac{3}{4})+\sum_{j=1}^lb_j+\frac{3(l+1)}{4}}\frac{1}{2r\sqrt{q}}e^{-\frac{\pi \sqrt{-1}}{r}(\sum_{i=1}^lb_i+\sum_{i=1}^{l-1}\frac{1}{C_iC_{i+1}})}e^{\sigma\left(\frac{3}{r}+\frac{r+1}{4}\right)\pi \sqrt{-1}}.
\end{align}

Finally,  we obtain
\begin{align}
    &J_{N}(W(p,q);t)\\\nonumber
    &=-\kappa_N\frac{(-1)^{\frac{3N}{2}+1/4}\sqrt{r}}{\sin\frac{\pi }{r}}\sum_{s=0}^{|q|-1}\sum_{i'=\frac{1}{2}}^{N-\frac{1}{2}}(-1)^{-2i'}\\\nonumber
    &
    \cdot e^{\frac{\pi i(i'^2-i'-2)}{r}}\sum_{n'=-i'}^{i'}(-1)^{b_l+P(s)}e^{\frac{r}{4\pi \sqrt{-1}}\left(\frac{C_{l-1}}{q}\left(\frac{2\pi n'}{r}\right)^2+\frac{2\pi I(s)}{q}\frac{2\pi n'}{r}+\pi^2K(s)\right)}\\\nonumber
    &\cdot\sin\left(\frac{2\pi n'}{rq}-J(s)\pi\right)e^{\frac{2\pi i}{r}\frac{b_ln'^2}{2}}\frac{\{2N-i'-\frac{1}{2}\}!\{2N-n'-i'\}!\{N-i'-\frac{1}{2}\}!}{\{i'-\frac{1}{2}\}!\{i'-n'\}!\{2N-2i'\}!}\\\nonumber
    &=-\kappa_N\frac{(-1)^{\frac{3N}{2}+1/4+b_l}\sqrt{r}}{\sin\frac{\pi }{r}}\sum_{s=0}^{|q|-1}\sum_{n'=-N+\frac{1}{2}}^{N-\frac{1}{2}}\sum_{i'=\max\{-n',n'\}}^{N-\frac{1}{2}}(-1)^{-2i'}\\\nonumber
    &
    \cdot (-1)^{P(s)}e^{-\frac{r\pi K(s)\sqrt{-1}}{4}}e^{\frac{r}{4\pi \sqrt{-1}}\left((\frac{C_{l-1}}{q}-b_l)\left(\frac{2\pi n'}{r}\right)^2+\frac{2\pi I(s)}{q}\frac{2\pi n'}{r}-(\frac{2\pi i'}{r})^2+\frac{2\pi}{r}\frac{2\pi i'}{r}+2(\frac{2\pi}{r})^2\right)}\\\nonumber
    &\cdot \sin\left(\frac{2\pi n'}{rq}-J(s)\pi\right)\frac{\{2N-i'-\frac{1}{2}\}!\{2N-n'-i'\}!\{N-i'-\frac{1}{2}\}!}{\{i'-\frac{1}{2}\}!\{i'-n'\}!\{2N-2i'\}!}.
\end{align}

In conclusion, we have the formula
\begin{align}
   &J_{N}(W(p,q);t)\\\nonumber
   &=-\kappa_N\frac{(-1)^{\frac{3N}{2}+\frac{1}{4}+b_l}\sqrt{2N+1}}{\sin\frac{\pi }{2N+1}}\sum_{s=0}^{|q|-1}(-1)^{P(s)}\sum_{n'=-N+\frac{1}{2}}^{N-\frac{1}{2}}\sum_{i'=\max\{-n',n'\}}^{N-\frac{1}{2}}(-1)^{-2i'}\\\nonumber
    &\cdot
    e^{-\frac{(2N+1)\pi K(s)\sqrt{-1}}{4}}e^{\frac{2N+1}{4\pi \sqrt{-1}}\left((\frac{C_{l-1}}{q}-b_l)\left(\frac{2\pi n'}{2N+1}\right)^2+\frac{2\pi I(s)}{q}\frac{2\pi n'}{2N+1}-(\frac{2\pi i'}{2N+1})^2+\frac{2\pi}{2N+1}\frac{2\pi i'}{2N+1}+2(\frac{2\pi}{2N+1})^2\right)}\\\nonumber
    &\cdot\sin\left(\frac{2\pi n'}{(2N+1)q}-J(s)\pi\right)\frac{\{2N-\frac{1}{2}-i'\}!\{2N-n'-i'\}!\{N-i'-\frac{1}{2}\}!}{\{i'-\frac{1}{2}\}!\{i'-n'\}!\{2N-2i'\}!}.
\end{align}
By straightforward computations, we have
\begin{align}
&e^{\frac{2N+1}{4\pi \sqrt{-1}}\left((\frac{C_{l-1}}{q}-b_l)\left(\frac{2\pi n'}{2N+1}\right)^2+\frac{2\pi I(s)}{q}\frac{2\pi n'}{2N+1}-(\frac{2\pi i'}{2N+1})^2+\frac{2\pi}{2N+1}\frac{2\pi i'}{2N+1}+2(\frac{2\pi}{2N+1})^2\right)}\\\nonumber
    &\cdot (-1)^{-2i'}e^{-\frac{(2N+1)\pi K(s)\sqrt{-1}}{4}}\cdot\frac{\{2N-\frac{1}{2}-i'\}!\{2N-n'-i'\}!\{N-i'-\frac{1}{2}\}!}{\{i'-\frac{1}{2}\}!\{i'-n'\}!\{2N-2i'\}!}\\\nonumber
    &=e^{\frac{N+\frac{1}{2}}{2\pi\sqrt{-1}}\left(\frac{2\pi}{N+\frac{1}{2}}\frac{\pi i'}{N+\frac{1}{2}}+4\frac{\pi n'}{N+\frac{1}{2}}\frac{\pi i'}{N+\frac{1}{2}}-4(\frac{\pi i'}{N+\frac{1}{2}})^2-\frac{p}{q}(\frac{\pi n'}{N+\frac{1}{2}})^2+\frac{2\pi I(s)}{q}(\frac{\pi n'}{N+\frac{1}{2}})+\pi^2 K(s)\right)}\\\nonumber
    &\cdot e^{\pi\sqrt{-1}\left(\frac{1}{2}N-2i'-\frac{1}{4}-\frac{\frac{1}{2}}{N+\frac{1}{2}}-\frac{n'}{N+\frac{1}{2}}-\frac{1}{N+\frac{1}{2}}\right)}\cdot \frac{(t)_{2N-\frac{1}{2}-i'}(t)_{2N-n'-i'}(t)_{N-i'-\frac{1}{2}}}{(t)_{i'-\frac{1}{2}}(t)_{i'-n'}(t)_{2N-2i'}}.
\end{align}

Based on Lemma \ref{lemma-varphixi}, for $0\leq i'\leq \frac{N}{2}$, $0\leq i'\pm n'\leq N$, we have
\begin{align}
    (t)_{2N-i'-\frac{1}{2}}=\exp \left(\varphi_N\left(\frac{\frac{1}{2}}{N+\frac{1}{2}}\right)-\varphi_N\left(\frac{2N-i'}{N+\frac{1}{2}}-1\right)+\log 2\right),
\end{align}
\begin{align}
    (t)_{N-i'-\frac{1}{2}}=\exp \left(\varphi_N\left(\frac{\frac{1}{2}}{N+\frac{1}{2}}\right)-\varphi_N\left(\frac{(N-i'-\frac{1}{2})+\frac{1}{2}}{N+\frac{1}{2}}\right)\right),
\end{align}

\begin{align}
    (t)_{2N-n'-i'}=\exp \left(\varphi_N\left(\frac{\frac{1}{2}}{N+\frac{1}{2}}\right)-\varphi_N\left(\frac{(2N-n'-i')+\frac{1}{2}}{N+\frac{1}{2}}-1\right)+\log 2\right),
\end{align}

\begin{align}
    (t)_{i'-\frac{1}{2}}=\exp \left(\varphi_N\left(\frac{\frac{1}{2}}{N+\frac{1}{2}}\right)-\varphi_N\left(\frac{i'}{N+\frac{1}{2}}\right)\right),
\end{align}

\begin{align}
    (t)_{i'-n'}=\exp \left(\varphi_N\left(\frac{\frac{1}{2}}{N+\frac{1}{2}}\right)-\varphi_N\left(\frac{(i'-n')+\frac{1}{2}}{N+\frac{1}{2}}\right)\right).
\end{align}

\begin{align}
    (t)_{2N-2i'}=\exp \left(\varphi_N\left(\frac{\frac{1}{2}}{N+\frac{1}{2}}\right)-\varphi_N\left(\frac{(2N-2i')+\frac{1}{2}}{N+\frac{1}{2}}-1\right)+\log 2\right),
\end{align}

Hence  we have

(1) If $\frac{1}{2}\leq i'\leq \frac{N}{2}$, $0\leq i'\pm n'\leq N$, then

\begin{align}
    &\frac{(t)_{2N-i'-\frac{1}{2}}(t)_{N-i'-\frac{1}{2}}(t)_{2N-n'-i'}}{(t)_{i'-\frac{1}{2}}(t)_{i'-n'}(t)_{2N-2i'}}\\\nonumber
    &=\exp\left(\log2-\varphi_N\left(\frac{N-i'-\frac{1}{2}}{N+\frac{1}{2}}\right)-\varphi_N\left(\frac{N-i'}{N+\frac{1}{2}}\right)-\varphi_N\left(\frac{N-n'-i'}{N+\frac{1}{2}}\right) \right)\\\nonumber
    &\cdot\exp\left(\varphi_N\left(\frac{i'-n'+\frac{1}{2}}{N+\frac{1}{2}}\right)+\varphi_N\left(\frac{i'}{N+\frac{1}{2}}\right)+\varphi_N\left(\frac{N-2i'}{N+\frac{1}{2}}\right)\right).
\end{align}

(2) If $\frac{N}{2}\leq i'\leq N-\frac{1}{2}$, $0\leq i'\pm n'\leq N$, then

\begin{align}
    &\frac{(t)_{2N-i'-\frac{1}{2}}(t)_{N-i'-\frac{1}{2}}(t)_{2N-n'-i'}}{(t)_{i'-\frac{1}{2}}(t)_{i'-n'}(t)_{2N-2i'}}\\\nonumber
    &=\exp\left(\log 4-\varphi_N\left(\frac{N-i'-\frac{1}{2}}{N+\frac{1}{2}}\right)-\varphi_N\left(\frac{N-i'}{N+\frac{1}{2}}\right)-\varphi_N\left(\frac{N-n'-i'}{N+\frac{1}{2}}\right) \right)\\\nonumber
    &\cdot\exp\left(\varphi_N\left(\frac{i'-n'+\frac{1}{2}}{N+\frac{1}{2}}\right)+\varphi_N\left(\frac{i'}{N+\frac{1}{2}}\right)+\varphi_N\left(\frac{N-2i'}{N+\frac{1}{2}}+1\right)\right).
\end{align}

(3) If $\frac{N}{2}\leq i'\leq N-\frac{1}{2}$, $N\leq i'+ n'\leq 2N$ and $0\leq i'- n'\leq N$, then

\begin{align}
    &\frac{(t)_{2N-i'-\frac{1}{2}}(t)_{N-i'-\frac{1}{2}}(t)_{2N-n'-i'}}{(t)_{i'-\frac{1}{2}}(t)_{i'-n'}(t)_{2N-2i'}}\\\nonumber
    &=\exp\left(\log2-\varphi_N\left(\frac{N-i'-\frac{1}{2}}{N+\frac{1}{2}}\right)-\varphi_N\left(\frac{N-i'}{N+\frac{1}{2}}\right)-\varphi_N\left(\frac{N-n'-i'}{N+\frac{1}{2}}+1\right) \right)\\\nonumber
    &\cdot\exp\left(\varphi_N\left(\frac{i'-n'+\frac{1}{2}}{N+\frac{1}{2}}\right)+\varphi_N\left(\frac{i'}{N+\frac{1}{2}}\right)+\varphi_N\left(\frac{N-2i'}{N+\frac{1}{2}}+1\right)\right).
\end{align}

(4) If $\frac{N}{2}\leq i'\leq N-\frac{1}{2}$, $0\leq i'+ n'\leq N$ and $N\leq i'- n'\leq 2N$, then

\begin{align}
    &\frac{(t)_{2N-i'-\frac{1}{2}}(t)_{N-i'-\frac{1}{2}}(t)_{2N-n'-i'}}{(t)_{i'-\frac{1}{2}}(t)_{i'-n'}(t)_{2N-2i'}}\\\nonumber
    &=\exp\left(\log2-\varphi_N\left(\frac{N-i'-\frac{1}{2}}{N+\frac{1}{2}}\right)-\varphi_N\left(\frac{N-i'}{N+\frac{1}{2}}\right)-\varphi_N\left(\frac{N-n'-i'}{N+\frac{1}{2}}\right) \right)\\\nonumber
    &\cdot\exp\left(\varphi_N\left(\frac{i'-n'+\frac{1}{2}}{N+\frac{1}{2}}-1\right)+\varphi_N\left(\frac{i'}{N+\frac{1}{2}}\right)+\varphi_N\left(\frac{N-2i'}{N+\frac{1}{2}}+1\right)\right).
\end{align}


We introduce the variables
\begin{align}
 \theta_1=\frac{n'}{N+\frac{1}{2}}, \
    \theta_2=\frac{i'}{N+\frac{1}{2}}
\end{align}
and define the function $V_N(p,q,s,\theta_1,\theta_2)$ as follows:

(1) If $0<\theta_2<\frac{1}{2}$ and $0<\theta_2\pm \theta_1<1$, then 
\begin{align} \label{formula-VN(1)}
    &V_{N}(p,q,s,\theta_1,\theta_2)\\\nonumber
    &=\pi \sqrt{-1}\left(\frac{1}{2}-2\theta_2+2\theta_2^2-I(s)\frac{\theta
    _1}{q}-2\theta_1\theta_2+\frac{p}{2q}\theta_1
    ^2-\frac{K(s)}{2}-\frac{\frac{1}{2}}{N+\frac{1}{2}}\right.\\\nonumber
    &\left.-\frac{\theta_1}{N+\frac{1}{2}}-\frac{\theta_2}{N+\frac{1}{2}}-\frac{3}{2(N+\frac{1}{2})^2}\right)\\\nonumber
    &+\frac{1}{N+\frac{1}{2}}\left(\log2-\varphi_N\left(1-\theta_2-\frac{1}{N+\frac{1}{2}}\right)+\varphi_N\left(\theta_2 \right)-\varphi_N\left(1-\theta_1-\theta_2-\frac{\frac{1}{2}}{N+\frac{1}{2}}\right)\right.\\\nonumber
    &\left.+\varphi_N\left(\theta_2-\theta_1+\frac{\frac{1}{2}}{N+\frac{1}{2}}\right)-\varphi_N\left(1-\theta_2-\frac{\frac{1}{2}}{N+\frac{1}{2}}\right)+\varphi_N\left(1-2\theta_2-\frac{\frac{1}{2}}{N+\frac{1}{2}}\right)\right).
\end{align}

(2) If $\frac{1}{2}\leq \theta_2<1$ and $0<\theta_2\pm \theta_1<1$, then 
\begin{align}
    &V_{N}(p,q,s,\theta_1,\theta_2)\\\nonumber
    &=\pi \sqrt{-1}\left(\frac{1}{2}-2\theta_2+2\theta_2^2-I(s)\frac{\theta_1}{q}-2\theta_1\theta_2+\frac{p}{2q}\theta_1^2-\frac{K(s)}{2}-\frac{\frac{1}{2}}{N+\frac{1}{2}}\right.\\\nonumber
    &\left.-\frac{\theta_1}{N+\frac{1}{2}}-\frac{\theta_2}{N+\frac{1}{2}}-\frac{3}{2(N+\frac{1}{2})^2}\right)\\\nonumber
    &+\frac{1}{N+\frac{1}{2}}\left(\log 4-\varphi_N\left(1-\theta_2-\frac{1}{N+\frac{1}{2}}\right)+\varphi_N\left(\theta
    _2\right)-\varphi_N\left(1-\theta_1-\theta_2-\frac{\frac{1}{2}}{N+\frac{1}{2}}\right)\right.\\\nonumber
    &\left.+\varphi_N\left(\theta_2-\theta_1+\frac{\frac{1}{2}}{N+\frac{1}{2}}\right)-\varphi_N\left(1-\theta_2-\frac{\frac{1}{2}}{N+\frac{1}{2}}\right)+\varphi_N\left(2-2\theta_2-\frac{\frac{1}{2}}{N+\frac{1}{2}}\right)\right).
\end{align}

(3) If $\frac{1}{2}\leq \theta_2<1$,  $1< \theta_2+ \theta_1< 2$ and $0< \theta_2-\theta_1< 1$, then 
\begin{align}
    &V_{N}(p,q,s,\theta_1,\theta_2)\\\nonumber
    &=\pi \sqrt{-1}\left(\frac{1}{2}-2\theta_2+2\theta_2^2-I(s)\frac{\theta
    _1}{q}-2\theta_1\theta_2+\frac{p}{2q}\theta_1^2-\frac{K(s)}{2}-\frac{\frac{1}{2}}{N+\frac{1}{2}}\right.\\\nonumber
    &\left.-\frac{\theta_1}{N+\frac{1}{2}}-\frac{\theta_2}{N+\frac{1}{2}}-\frac{3}{2(N+\frac{1}{2})^2}\right)\\\nonumber
    &+\frac{1}{N+\frac{1}{2}}\left(\log 2-\varphi_N\left(1-\theta_2-\frac{1}{N+\frac{1}{2}}\right)+\varphi_N\left(
\theta_2\right)-\varphi_N\left(2-\theta_1-\theta_2-\frac{\frac{1}{2}}{N+\frac{1}{2}}\right)\right.\\\nonumber
    &\left.+\varphi_N\left(\theta_2-\theta_1+\frac{\frac{1}{2}}{N+\frac{1}{2}}\right)-\varphi_N\left(1-\theta_2-\frac{\frac{1}{2}}{N+\frac{1}{2}}\right)+\varphi_N\left(2-2\theta_2-\frac{\frac{1}{2}}{N+\frac{1}{2}}\right)\right).
\end{align}

(4) If $\frac{1}{2}\leq \theta_2<1$,  $0< \theta_2+\theta_1<1$ and $1< \theta_2-\theta_1<2$, then 
\begin{align}
    &V_{N}(p,q,s,\theta_1,\theta_2)\\\nonumber
    &=\pi \sqrt{-1}\left(\frac{1}{2}-2\theta_2+2\theta_2^2-I(s)\frac{\theta_1}{q}-2\theta_1\theta_2+\frac{p}{2q}\theta_1^2-\frac{K(s)}{2}-\frac{\frac{1}{2}}{N+\frac{1}{2}}\right.\\\nonumber
    &\left.-\frac{\theta_1}{N+\frac{1}{2}}-\frac{\theta_2}{N+\frac{1}{2}}-\frac{3}{2(N+\frac{1}{2})^2}\right)\\\nonumber
    &+\frac{1}{N+\frac{1}{2}}\left(\log 2-\varphi_N\left(1-\theta_2-\frac{1}{N+\frac{1}{2}}\right)+\varphi_N\left(\theta_2 \right)-\varphi_N\left(1-\theta_2-\theta_2-\frac{\frac{1}{2}}{N+\frac{1}{2}}\right)\right.\\\nonumber
    &\left.+\varphi_N\left(\theta_2-\theta_1-\frac{\frac{1}{2}}{N+\frac{1}{2}}\right)-\varphi_N\left(1-\theta_2-\frac{\frac{1}{2}}{N+\frac{1}{2}}\right)+\varphi_N\left(2-2\theta_2-\frac{\frac{1}{2}}{N+\frac{1}{2}}\right)\right).
\end{align}

Combining the above formulas together, we obtain 
\begin{proposition} \label{proposition-potential}
For an odd integer $r=2N+1$ with $N\geq 1$ and at the root of unity $t=e^{\frac{2\pi\sqrt{-1}}{N+\frac{1}{2}}}$, the (normalized) relative Reshetikhin-Turaev invariant $J_N(W(p,q);t)$ is given by 

    \begin{align} \label{formula-relativeRT}
     J_{N}(W(p,q);t)
      &=\kappa_N\frac{(-1)^{\frac{3N}{2}+\frac{1}{4}+b_l}\sqrt{r}}{\sin\frac{\pi }{r}}\sum_{s=0}^{|q|-1}\sum_{n'=-N+\frac{1}{2}}^{N-\frac{1}{2}}\sum_{i'=\max\{-n',n'\}}^{N-\frac{1}{2}}g_{N}(s,n',i'),
\end{align}
where 
\begin{align} \label{formula-KappaN}
    \kappa_N=(-1)^{b_l(\frac{3}{2}N+\frac{3}{4})+\sum_{j=1}^lb_j+\frac{3(l+1)}{4}}\frac{1}{2r\sqrt{q}}e^{-\frac{\pi \sqrt{-1}}{r}(\sum_{i=1}^lb_i+\sum_{i=1}^{l-1}\frac{1}{C_iC_{i+1}})}e^{\sigma\left(\frac{3}{r}+\frac{r+1}{4}\right)\pi \sqrt{-1}},
\end{align}
\begin{align}
    g_{N}(s,n',i')&=(-1)^{P(s)}
      \sin\left(\frac{n'\pi }{(N+\frac{1}{2})q}-J(s)\pi\right)e^{\left((N+\frac{1}{2})V_N\left(p,q,s,\frac{n'}{N+\frac{1}{2}},\frac{i'}{N+\frac{1}{2}}\right)\right)},
\end{align}
with the function $V_N(p,q,s,\theta_1,\theta_2)$ defined above. 
\end{proposition}



\section{Poisson summation formula} \label{Section-Poissonsummation}
In this section, we write the formula (\ref{formula-relativeRT}) as a sum of integrals with the help of Poisson summation formula. Combining formulas (\ref{formula-relativeRT0}) and (\ref{formula-relativeRT}) together,  we have
\begin{align}
&|g_N(s,n',i')|&\\\nonumber
&=\left|\sin\left(\frac{\pi n'}{(N+\frac{1}{2})q}-J(s)\pi\right)\right|\left|\frac{\{2N-i'-\frac{1}{2}\}!\{2N-n'-i'\}!\{N-i'-\frac{1}{2}\}!}{\{i'-\frac{1}{2}\}!\{i'-n'\}!\{2N-2i'\}!}\right|. 
\end{align}

By Lemmas \ref{lemma-varphixi}, \ref{lemma-varphixi2}, \ref{lemma-varphixi3} and formula (\ref{formula-Lambda(t)}), we obtain 
\begin{align}
    \log \left|\{n\}!\right|=-(N+\frac{1}{2})\Lambda\left(\frac{2n+1}{2N+1}\right)+O(\log (2N+1))
\end{align}  
for any integer $0<n<2N+1$ and at $t=e^{\frac{2\pi \sqrt{-1}}{N+\frac{1}{2}}}$.
Then we have
\begin{align}
    &\log \left|g_r(s,m',n')\right|\\\nonumber
    &=-\frac{r}{2}\Lambda\left(\frac{2(2N-n'-\frac{1}{2})}{r}\right)-\frac{r}{2}\Lambda\left(\frac{2(2N-m'-n')}{r}\right)-\frac{r}{2}\Lambda\left(\frac{2(N-n'-\frac{1}{2})}{r}\right)\\\nonumber
    &+\frac{r}{2}\Lambda\left(\frac{2(n'-\frac{1}{2})}{r}\right)+\frac{r}{2}\Lambda\left(\frac{2(n'-m')}{r}\right)+\frac{r}{2}\Lambda\left(\frac{2(2N-2n')}{r}\right)+O(\log r)\\\nonumber
    &=\frac{r}{2}\Lambda\left(\frac{2n'}{r}+\frac{3}{r}\right)+\frac{r}{2}\Lambda\left(\frac{2m'}{r}+\frac{2n'}{r}+\frac{2}{r}\right)+\frac{r}{2}\Lambda\left(\frac{2n'}{r}+\frac{2}{r}\right)\\\nonumber
    &+\frac{r}{2}\Lambda\left(\frac{2n'}{r}-\frac{1}{r}\right)+\frac{r}{2}\Lambda\left(\frac{2n'}{r}-\frac{2m'}{r}\right)-\frac{r}{2}\Lambda\left(\frac{4n'}{r}+\frac{2}{r}\right)+O(\log r)
\end{align}
where in the second ``=" , we have used the properties of the function $\Lambda(\tau)$.

We introduce the function
\begin{align}
    v_N(\theta_1,\theta_2)&=\Lambda\left(\theta_2+\frac{3}{2(N+\frac{1}{2})}\right)+\Lambda\left(\theta_1+\theta_2+\frac{1}{N+\frac{1}{2}}\right)+\Lambda\left(\theta_2+\frac{1}{N+\frac{1}{2}}\right)\\\nonumber
    &+\Lambda\left(\theta_2-\frac{1}{2(N+\frac{1}{2})}\right)+\Lambda\left(\theta_2-\theta_1\right)-\Lambda\left(2\theta_2+\frac{1}{(N+\frac{1}{2})}\right),
\end{align}
then we have
\begin{align} \label{formula-gN}
\left|g_N(s,n',i')\right|=e^{(N+\frac{1}{2})v_N\left(\frac{n'}{N+\frac{1}{2}},\frac{i'}{N+\frac{1}{2}}\right)+O(\log N)}. 
\end{align}
Set
 \begin{align}
v(\theta_1,\theta_2)=\lim_{N\rightarrow \infty}v_{N}(\theta_1,\theta_2)=3\Lambda(\theta_2)+\Lambda(\theta_2+\theta_1)+\Lambda(\theta_2-\theta_1)-\Lambda(2\theta_2).
    \end{align}  
It is easy to see that $v(\theta_1,\theta_2)\leq 0$  for $-1<\theta_1<1$ and $\frac{1}{2}\leq \theta_2\leq 1$. 
Set the region 
\begin{align}
   D=\{(\theta_1,\theta_2)\in \mathbb{R}^2| \theta_2+\theta_1>0, \theta_2-\theta_1>0, 0<\theta_2<\frac{1}{2}\},
\end{align}
which is illustrated in Figure \ref{figure:D}.

\begin{figure}[!htb] 
\begin{align*} 
\raisebox{-15pt}{
\includegraphics[width=230 pt]{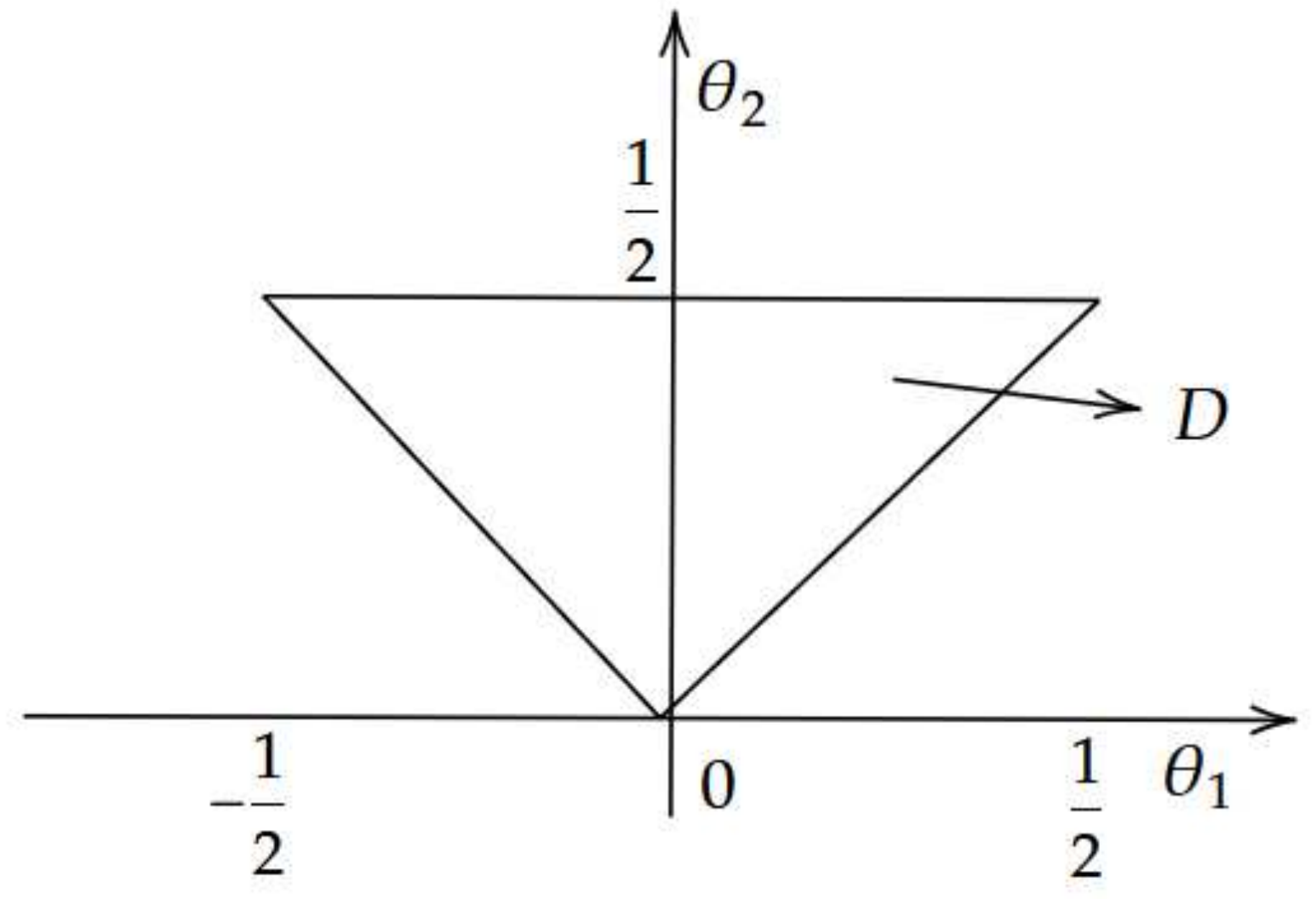}}.
\end{align*}
\caption{The region $D$}
\label{figure:D} 
\end{figure}
Note that $\left(\frac{n'}{N+\frac{1}{2}},\frac{i'}{N+\frac{1}{2}}\right)\in D$ when $i'\pm n'>0$ and $0<i'\leq \frac{N}{2}$. Let $c_0=0.122532$,  we define $D_0=D'_0\cup D''_0$, where
\begin{align}
  D'_0  =\{(\pm \theta_1,\theta_2)\in D|\theta_1\leq \theta_2\leq \frac{1}{2}-\theta_1, 0\leq \theta_1\leq c_0\}
\end{align}
and
\begin{align}
    D''_0= \{(\pm \theta_1,\theta_2)\in D|\theta_1\leq \theta_2\leq \frac{1}{2}, c_0\leq  \theta_1\leq \frac{1}{4}\}.
\end{align}
Note that the regions $D_0'$ and $D''_0$ are symmetric with respect to the $\theta_2$-axis as shown in Figure \ref{figure:D'D''}

\begin{figure}[!htb]  
\begin{align*} 
\raisebox{-15pt}{
\includegraphics[width=230 pt]{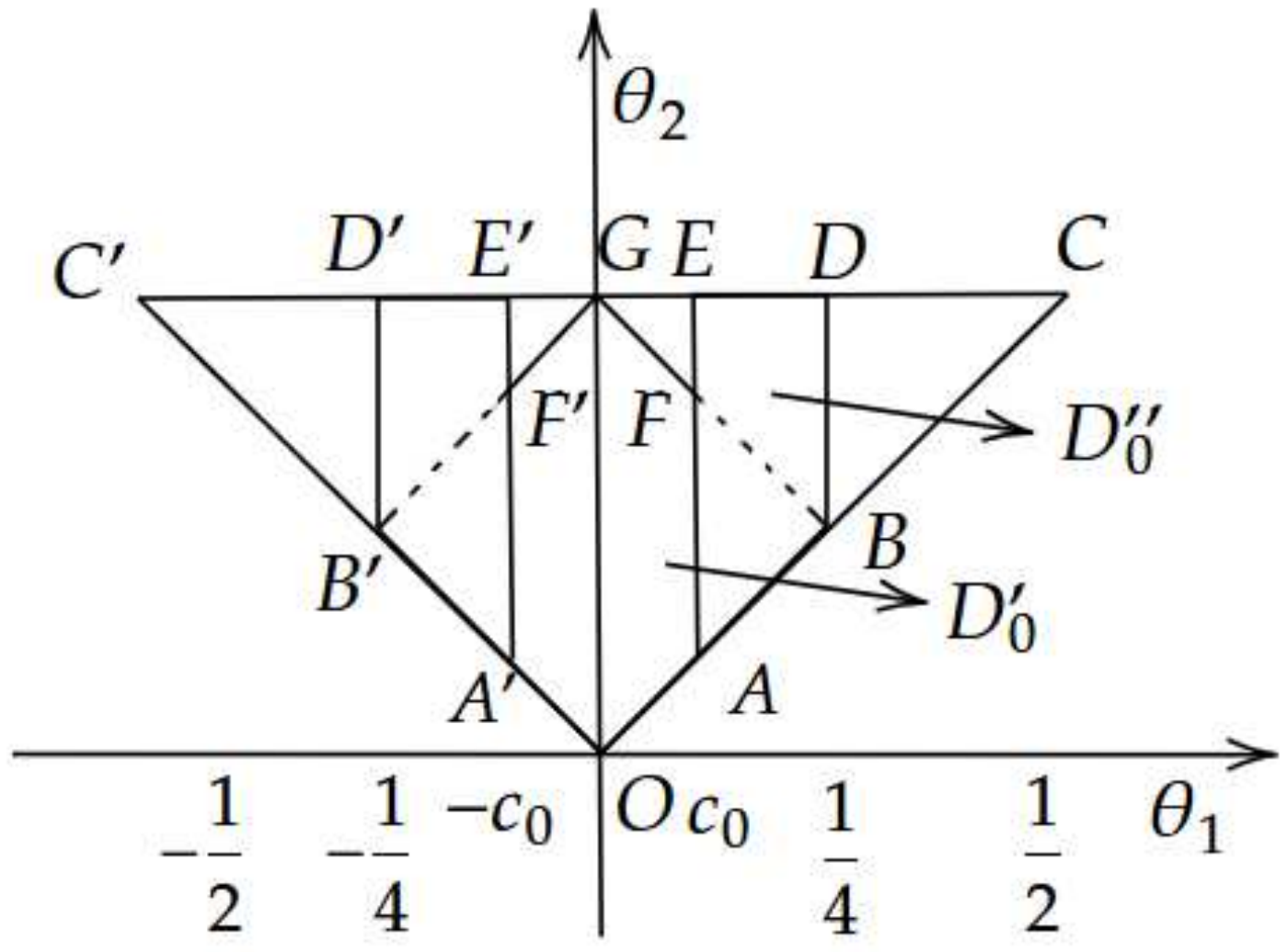}}.
\end{align*}
\caption{The region $D'_0$ and $D''_0$}
\label{figure:D'D''}
\end{figure}

 Then we have
\begin{lemma} \label{lemma-v}
  The region 
  \begin{align} \label{formula-domain}
      \{(\theta_1,\theta_2)\in D|  v(\theta_1,\theta_2)>\frac{3.374482}{2\pi }\}
  \end{align}
  is included in the region $D_0$. 
\end{lemma}
\begin{proof}
We consider the region $R^+=EFG\cup BCD$, and $R^-=E'F'G\cup B'C'D'$ as show in Figure \ref{figure:D'D''}.   We will show that when $(\theta_1,\theta_2)\in R^\pm $, then $v(\theta_1,
\theta_2)<\frac{3.374482}{2\pi }$. 
Since the function $v(\theta_1,\theta_2)$ is symmetric with respect to the $\theta_2$-axis, we only need to show that $v(\theta_1,\theta_2)<\frac{3.374482}{2\pi }$ for $(\theta_1,\theta_2)\in R^+$.

Given a constant $c\in (0,\frac{1}{2})$, as a function of $\theta_2$, we have
\begin{align}
 \frac{dv}{d\theta_2}(c,\theta_2)
 &=-3\log(2\sin(\pi \theta_2))-\log(2\sin(\pi (\theta_2+c)))\\\nonumber
 &- \log(2\sin(\pi (\theta_2-c)))+2\log(2\sin(2\pi \theta_2))\\\nonumber
 &=-\log\left(\frac{2^3\sin^3(\pi z)\sin(\pi(\theta_2+c))\sin(\pi(\theta_2-c))}{\sin^2(2\pi \theta_2)}\right)\\\nonumber
 &=-\log\left(\tan(\pi \theta_2)\left(\frac{\cos(2\pi c)-\cos(2\pi z)}{\cos(\pi \theta_2)}\right)\right).
\end{align}
For $0\leq c\leq c_0$ and $\frac{1}{2}-c<\theta_2\leq \frac{1}{2}$,  we have $\frac{1}{4}<\theta_2<\frac{1}{2}$. Hence
$
\tan(\pi \theta_2)>1.
$
Furthermore, as function of $\theta_2$, for $\frac{1}{2}-c<\theta_2\leq \frac{1}{2}$,
\begin{align}
   \frac{\cos(2\pi c)-\cos(2\pi \theta_2)}{\cos(\pi \theta_2)}&>\frac{\cos(2\pi c)-\cos(2\pi(\frac{1}{2}-c))}{\cos(\pi(\frac{1}{2}-c))}\\\nonumber
   &=\frac{2\cos(2\pi c)}{\sin(\pi c)}>\frac{2\cos(2\pi\cdot  c_0)}{\sin(\pi\cdot c_0)}>1. 
\end{align}
So we have 
\begin{align}
 \frac{dv}{d\theta_2}(c,\theta_2)<0.    
\end{align}
Hence 
\begin{align}
v(c,\theta_2)<v\left(c,\frac{1}{2}-c\right).    
\end{align}
It is easy to compute straightforward that 
\begin{align}
     v\left(c_0,\frac{1}{2}-c_0\right)<\frac{3.374482}{2\pi }.
\end{align}

On the other hand, as function of $\theta_1$,  it is easy to see 
\begin{align}
v(\theta_1,\frac{1}{2}-\theta_1)=-3\Lambda(\theta_1-\frac{1}{2})-\Lambda(2\theta_1-\frac{1}{2})+\Lambda(2\theta_1)    
\end{align}
is an increasing function for $0<\theta_1<c_0$.
Hence, for $\theta_1+\theta_2\geq \frac{1}{2}$ and $0\leq \theta_1\leq c_0$, we have
\begin{align}
 v(\theta_1,\theta_2)<v\left(\theta_1,\frac{1}{2}-\theta_1\right)<v\left(c_0,\frac{1}{2}-c_0\right)=\frac{3.3744816}{2\pi}<\frac{3.374482}{2\pi }.   
\end{align}
For $(\theta_1,\theta_2)\in D$ with $\theta_1\geq \frac{1}{4}$, it is easy to show that for a fixed $\theta_2\in [0,\frac{1}{2}]$,  $v(\theta_1,\theta_2)$ is an decreasing function as a function of $\theta_1\in (0,\frac{1}{2})$, hence 
\begin{align}
    v(\theta_1,\theta_2)\leq v\left(\frac{1}{4},\theta_2\right)\leq \max_{\theta_2\in [0,\frac{1}{2}]}v\left(\frac{1}{4},\theta_2\right)<\frac{3.374482}{2\pi }.
\end{align}

\end{proof}

\begin{remark}
We can take $\varepsilon>0$ small enough (for example $\varepsilon=0.00001$), and set 
\begin{align}
    D_\varepsilon=\{(\theta_1,\theta_2)\in D_0|d((\theta_1,\theta_2),\partial D_0)>\varepsilon \},
\end{align}
then the region (\ref{formula-domain}) can also be included in the region $D_{\varepsilon}$.
\end{remark}

Let $\zeta_{\mathbb{R}}(p,q)$ be the real part of the critical value $V^\pm(p,q;\pm\theta_1^0,\theta_2^0)$.

\begin{proposition}
When $(p,q)\in S$, there exists $\epsilon>0$, for $s\in \{1,...,|q|-1\}$  and $(\frac{n'}{N+\frac{1}{2}},\frac{i'}{N+\frac{1}{2}})$ is not in $D_0$, then 
\begin{align}
    |g_N(s,n',i')|<O\left(e^{(N+\frac{1}{2})\left(\zeta_{\mathbb{R}}(p,q)-\epsilon\right)}\right)
\end{align}
for $N$ large enough. 
\end{proposition}
\begin{proof}
Since $v_N(\theta_1,\theta_{2})$ is uniformly convergent to $v(\theta_1,\theta_2)$ on $D$,  for $(\frac{n'}{N+\frac{1}{2}},\frac{i'}{N+\frac{1}{2}})\in D\setminus D_0$,  by Lemma \ref{lemma-v}, we obtain that 
\begin{align}
    |v_{N}\left(\frac{n'}{N+\frac{1}{2}},\frac{i'}{N+\frac{1}{2}}\right)|\leq \frac{3.374482}{2\pi },
\end{align}
when $N$ is large enough. 

By Theorem \ref{theorem-volestimate}  and Corollary \ref{corollary-pq}, we have
\begin{align}
\frac{3.374482}{2\pi}<\frac{Vol\left(W\left(p,q\right)\right)}{2\pi}=\zeta_{\mathbb{R}}(p,q).    
\end{align}

By formula (\ref{formula-gN}), there exists small $\epsilon>0$, such that 
\begin{align}
|g_{N}(s,n',i')|<O\left(e^{(N+\frac{1}{2})\left(\zeta_{\mathbb{R}}(p,q)-\epsilon\right)}\right)    
\end{align}
\end{proof}

Now we construct a smooth bump function $\psi$ on $\mathbb{R}^2$ such that 
$\psi(\theta_1,\theta_2)=1$ on $(\theta_1,\theta_2)\in D_{\varepsilon}$,  $0<\psi(\theta_1,\theta_2)<1$ on $(\theta_1,\theta_2)\in D_0\setminus D_{\varepsilon}$, $\psi(\theta_1,\theta_2)=0$ for $(\theta_1,\theta_2)\notin D_0$.  Let 
 \begin{align}
     h_N(s,n',i')=\psi\left(\frac{n'}{N+\frac{1}{2}},\frac{i'}{N+\frac{1}{2}}\right)g_{N}(s,n',i'),
 \end{align}
then
\begin{align}
    J_{N}(W(p,q);t)
    &=-\kappa_N\frac{(-1)^{\frac{3N}{2}+\frac{1}{4}+b_l}\sqrt{2N+1}}{\sin\frac{\pi }{2N+1}}\sum_{s=0}^{|q|-1}\sum_{(n',i')\in (\mathbb{Z}+\frac{1}{2})^2}h_N(s,n',i')\\\nonumber&+O\left(e^{(N+\frac{1}{2})\left(\zeta_{\mathbb{R}}(p,q)-\epsilon\right)}\right).
\end{align}   
Recall the Poisson summation formula 
\begin{align}
    \sum_{(n,i)\in \mathbb{Z}^2}f(n,i)=\sum_{(m_1,m_2)\in \mathbb{Z}^2}\hat{f}(m_1,m_2)
\end{align}
where 
\begin{align}
    \hat{f}(m_1,m_2)=\int_{\mathbb{R}^2}f(u,v)e^{2\pi \sqrt{-1}m_1u+2\pi \sqrt{-1}m_2v}dudv.
\end{align}

Let $n=n'-\frac{1}{2}$ and $i=i'-\frac{1}{2}$, then 
\begin{align}
    \sum_{(n',i')\in (\mathbb{Z}+\frac{1}{2})^2}h_N(s,n',i')&=\sum_{(n,i)\in \mathbb{Z}^2}h_{N}(s,n+\frac{1}{2},i+\frac{1}{2})\\\nonumber
    &=\sum_{(m_1,m_2)\in \mathbb{Z}^2}\int_{\mathbb{R}^2}h_N(s,u+\frac{1}{2},v+\frac{1}{2})e^{2\pi \sqrt{-1}(m_1u+m_2v)}dudv.
\end{align}
Since $\theta_1=\frac{n'}{N+\frac{1}{2}}=\frac{1}{N+\frac{1}{2}}(n+\frac{1}{2})$, $\theta_2=\frac{i'}{N+\frac{1}{2}}=\frac{1}{N+\frac{1}{2}}(i+\frac{1}{2})$, we obtain 
\begin{align}
    &\int_{\mathbb{R}^2}h_N(s,u+\frac{1}{2},v+\frac{1}{2})e^{2\pi \sqrt{-1}(m_1u+m_2v)}dudv\\\nonumber
    &=(-1)^{m_1+m_2}\left(N+\frac{1}{2}\right)^2\\\nonumber
    &\cdot\int_{\mathbb{R}^2}h_N\left(s,\left(N+\frac{1}{2}\right)\theta_1,\left(N+\frac{1}{2}\right)\theta_2\right)e^{2\pi \sqrt{-1}\left((N+\frac{1}{2})m_1\theta_1+(N+\frac{1}{2})m_2\theta_2\right)}d\theta_1 d\theta_2\\\nonumber
    &=(-1)^{m_1+m_2}\left(N+\frac{1}{2}\right)^2\int_{\mathbb{R}^2}\psi(\theta_1,\theta_2)(-1)^{P(s)}\sin\left(\frac{\theta_1\pi }{q}-J(s)\pi\right)\\\nonumber
    &\cdot\exp\left(\left(N+\frac{1}{2}\right)\left(V_N(s,\theta_1,
    \theta_2)+2\pi \sqrt{-1}m_1\theta_1+2\pi \sqrt{-1}m_2\theta_2)\right)\right)d\theta_1 d\theta_2.
\end{align}

In conclusion, we have
\begin{proposition}  \label{prop-fouriercoeff}
There exists $\epsilon>0$, such that
\begin{align}
J_{N}(W(p,q);t)&=  -\kappa_N\frac{(-1)^{\frac{3N}{2}+\frac{1}{4}+b_l}\sqrt{2N+1}}{\sin\frac{\pi }{2N+1}}\sum_{s=0}^{|q|-1}  \sum_{(m_1,m_2)\in \mathbb{Z}^2}\hat{h}_N(s,m_1,m_2)\\\nonumber
&+O\left(e^{(N+\frac{1}{2})\left(\zeta_{\mathbb{R}}(p,q))-\epsilon\right)}\right),
\end{align}
with
 \begin{align}
    \hat{h}_N(s,m_1,m_2)&=(-1)^{P(s)}(-1)^{m_1+m_2}(N+\frac{1}{2})^2\\\nonumber
    &\cdot\int_{D_{0}}\psi(\theta_1,\theta_2)\sin\left(\frac{\theta_1\pi }{q}-J(s)\pi\right) e^{(N+\frac{1}{2})V_N(p,q,s,\theta_1,\theta_2;m_1,m_2)}d\theta_1d\theta_2.
\end{align}   
where $\kappa_N$ is given by (\ref{formula-KappaN}) and
\begin{align}
    V_{N}(p,q,s,\theta_1,\theta_2;m_1,m_2)=V_{N}(p,q,s,\theta_1,\theta_2)+2m_1\pi \sqrt{-1} \theta_1+2m_2\pi\sqrt{-1} \theta_2,
\end{align}
$V_{N}(p,q,s,\theta_1,\theta_2)$ is given by formula (\ref{formula-VN(1)}) since we focus on the region $D_0$. 
\end{proposition}
We set
\begin{align} \label{formula-Vpqs}
    &V(p,q,s,\theta_1,\theta_2;m_1,m_2)\\\nonumber
    &=\lim_{N\rightarrow \infty}V_{N}(p,q,s,\theta_1,\theta_2;m_1,m_2)\\\nonumber
    &=\pi \sqrt{-1}\left(\frac{p}{2q}\theta_1^2-I(s)\frac{\theta_1}{q}-2\theta_1\theta_2-2\theta_2+2\theta_2^2+2m_1\theta_1+2m_2\theta_2-\frac{K(s)}{2}+\frac{1}{2}\right)\\\nonumber
    &+\frac{1}{2\pi\sqrt{-1}}\left(\text{Li}_2(e^{2\pi \sqrt{-1}(\theta_2-\theta_1)})-\text{Li}_2(e^{-2\pi \sqrt{-1}(\theta_2+\theta_1)})-2\text{Li}_2(e^{-2\pi \sqrt{-1}\theta_2})\right. \\\nonumber &\left.+\text{Li}_2(e^{2\pi\sqrt{-1}\theta_2})+\text{Li}_2(e^{-4\pi\sqrt{-1}\theta_2})\right).
\end{align}

\section{The geometry of the critical point} \label{Section-Critical} 
The goal of this section is to interpret the geometric meaning of the critical points and the critical values for the functions  $V^\pm(p,q;\theta_1,\theta_2)$ defined as follows, see formula (\ref{formula-V+-0}). We compute the critical point equations in Section \ref{subsection-critical}. Then, following the work \cite{NR90}, we present the hyperbolic gluing equations and the Dehn filling equations for the geometry of $W(p,q)$ in section \ref{subsection-geomequation}. We show the equivalence of the critical point equation and the geometric equation.  Finally, we prove Proposition \ref{prop-crit=volume} which is the main result of this section. 

\begin{align}
    &V_{N}(s^\pm,\theta_1,\theta_2;m_1^\pm,1)\\\nonumber
    &=\pi \sqrt{-1}\left(\frac{p}{2q}\theta_1^2+\left(1\mp \frac{1}{q}\right)\theta_1-2\theta_1\theta_2+2\theta_2^2+\frac{1}{2}-\frac{K(s^\pm)}{2}\right.\\\nonumber
    &\left.-\frac{\frac{1}{2}}{N+\frac{1}{2}}-\frac{\theta_1}{N+\frac{1}{2}}-\frac{\theta_2}{N+\frac{1}{2}}-\frac{3}{2(N+\frac{1}{2})^2}\right)\\\nonumber
    &+\frac{1}{N+\frac{1}{2}}\left(\log2-\varphi_N\left(1-\theta_2-\frac{1}{N+\frac{1}{2}}\right)+\varphi_N\left(\theta_2 \right)-\varphi_N\left(1-\theta_1-\theta_2-\frac{\frac{1}{2}}{N+\frac{1}{2}}\right)\right.\\\nonumber
    &\left.+\varphi_N\left(\theta_2-\theta_1+\frac{\frac{1}{2}}{N+\frac{1}{2}}\right)-\varphi_N\left(1-\theta_2-\frac{\frac{1}{2}}{N+\frac{1}{2}}\right)+\varphi_N\left(1-2\theta_2-\frac{\frac{1}{2}}{N+\frac{1}{2}}\right)\right),
\end{align}
\begin{align}
&V(s^\pm,\theta_1,\theta_2;m_1^\pm,1)\\\nonumber
&=\pi \sqrt{-1}\left(\frac{p}{2q}\theta_1^2+\left(1\mp \frac{1}{q}\right)\theta_1-2\theta_1\theta_2+2\theta_2^2-\frac{K(s^\pm)}{2}+\frac{1}{2}\right)\\\nonumber &+\frac{1}{2\pi\sqrt{-1}}\left(\text{Li}_2(e^{2\pi \sqrt{-1}(\theta_2-\theta_1)})-\text{Li}_2(e^{-2\pi \sqrt{-1}(\theta_2+\theta_1)})-2\text{Li}_2(e^{-2\pi \sqrt{-1}\theta_2})\right.\\\nonumber &\left.+\text{Li}_2(e^{2\pi\sqrt{-1}\theta_2})+\text{Li}_2(e^{-4\pi\sqrt{-1}\theta_2})\right).    
\end{align}
We define
\begin{align} \label{formula-V+-0}
V^\pm(p,q;\theta_1,\theta_2)&=V(p,q,s^\pm,\theta_1,\theta_2;m_1^\pm,1)+\left(\frac{K(s^\pm)}{2}+\frac{p'}{2q}\right)\pi\sqrt{-1} \\\nonumber
&=\pi \sqrt{-1}\left(\frac{p}{2q}\theta_1^2+\left(1\mp \frac{1}{q}\right)\theta_1-2\theta_1\theta_2+2\theta_2^2+\frac{1}{2}+\frac{p'}{2q}\right)\\\nonumber &+\frac{1}{2\pi\sqrt{-1}}\left(\text{Li}_2(e^{2\pi \sqrt{-1}(\theta_2-\theta_1)})-\text{Li}_2(e^{-2\pi \sqrt{-1}(\theta_2+\theta_1)})-2\text{Li}_2(e^{-2\pi \sqrt{-1}\theta_2})\right.\\\nonumber &\left.+\text{Li}_2(e^{2\pi\sqrt{-1}\theta_2})+\text{Li}_2(e^{-4\pi\sqrt{-1}\theta_2})\right).    
\end{align}

By Lemma \ref{lemma-varphixi2}, we have
\begin{align}
\varphi_{N}(1-\theta)=2\pi \sqrt{-1}\left(-\frac{2N+1}{4}\left(\theta^2-\theta+\frac{1}{6}\right)+\frac{1}{12(2N+1)}\right)-\varphi_{N}(\theta).
\end{align}
Then we obtain 
\begin{align}
&V_{N}(s^\pm,m_1^\pm,1,\theta_1,\theta_2)\\\nonumber
&=\pi\sqrt{-1}\left(\frac{5}{6}+\theta_2^2-\theta_2+\left(\frac{p}{2q}+1\right)\theta_1^2\mp \frac{\theta_1}{q}-\frac{K(s^+)}{2}-\frac{2}{N+\frac{1}{2}}+\frac{\theta_2}{N+\frac{1}{2}}-\frac{15}{2(N+\frac{1}{2})^2}\right)\\\nonumber
&+\frac{1}{N+\frac{1}{2}}\left(\log 2+\varphi_{N}\left(\theta_1+\theta_2+\frac{\frac{1}{2}}{N+\frac{1}{2}}\right)+\varphi_{N}\left(-\theta_1+\theta_2+\frac{\frac{1}{2}}{N+\frac{1}{2}}\right)\right.\\\nonumber
&\left. \varphi_N\left(\theta_2\right)+\varphi_{N}\left(\theta_2+\frac{\frac{1}{2}}{N+\frac{1}{2}}\right)+\varphi_{N}\left(\theta_2+\frac{1}{N+\frac{1}{2}}\right)-\varphi_{N}\left(2\theta_2+\frac{\frac{1}{2}}{N+\frac{1}{2}}\right)\right),
\end{align}
which implies that
\begin{align} \label{formula-V+-}
&V^\pm(p,q;\theta_1,\theta_2)\\\nonumber
&=V(s^{\pm},m_1^{\pm},1,\theta_1,\theta_2)+\left(\frac{K(s^{\pm})}{2}+\frac{p'}{2q}\right)\pi\sqrt{-1}\\\nonumber
&=\pi\sqrt{-1}\left(\left(\frac{p}{2q}+1\right)\theta_1^2\mp\frac{\theta_1}{q}+\theta_2^2-\theta_2+\frac{5}{6}+\frac{p'}{2q}\right)\\\nonumber
&+\frac{1}{2\pi\sqrt{-1}}\left(\text{Li}_2(e^{2\pi\sqrt{-1}(\theta_1+\theta_2)})+\text{Li}_2(e^{2\pi\sqrt{-1}(-\theta_1+\theta_2)})+3\text{Li}_2(e^{2\pi\sqrt{-1}\theta_2})-\text{Li}_2(e^{4\pi\sqrt{-1}\theta_2})\right).
\end{align}
Obviously, we have the following symmetry 
\begin{align} \label{formula-Vsym}
V^+(p,q;\theta_1,\theta_2)=V^{-}(p,q;-\theta_1,\theta_2).    
\end{align}

\subsection{Critical point equations} \label{subsection-critical}
Now, we consider the critical point equations for functions $V^{\pm}(p,q;\theta_1,\theta_2)$ that are given by

\begin{align} \label{equation-crit1}
    V^\pm_{\theta_1}&=\pi \sqrt{-1}\left(2\left(\frac{p}{2q}+1\right)\theta_1\mp \frac{1}{q}\right)+\log\left(1-e^{2\pi \sqrt{-1}(\theta_2-\theta_1)}\right)\\\nonumber
    &-\log\left(1-e^{2\pi \sqrt{-1}(\theta_2+\theta_1)}\right)=0,
\end{align}
\begin{align} \label{equation-crit2}
    V^\pm_{\theta_2}&=\pi \sqrt{-1}\left(2\theta_2-1\right)-\log\left(1-e^{2\pi \sqrt{-1}(\theta_2-\theta_1)}\right)-\log\left(1-e^{2\pi \sqrt{-1}(\theta_2+\theta_1)}\right)\\\nonumber
    &-3\log(1-e^{2\pi \sqrt{-1}\theta_2})+2\log\left(1-e^{4\pi\sqrt{-1}\theta_2}\right)=0.
\end{align}

We let $z_1=e^{2\pi\sqrt{-1}\theta_1}$, $z_2=e^{2\pi\sqrt{-1}\theta_2}$. Then we obtain

\begin{equation}  
\left\{ \begin{aligned}
         & z_1^p(z_1-z_2)^{2q}=(z_1z_2-1)^{2q}, \\
          &      (z_1z_2-1)(z_1-z_2)(1-z_2)=z_1z_2(z_2+1)^2,
                          \end{aligned} \right.
                          \end{equation}

which implies
\begin{equation}  \label{equation-critform}
\left\{ \begin{aligned}
         &z_2^2\left(z_1+\frac{1}{z_1}+3\right)-(z_1+\frac{1}{z_1})z_2+1=0, \\
          &       z_1^p\left(\frac{z_1-z_2}{z_1z_2-1}\right)^{2q}=1.
                          \end{aligned} \right.
                          \end{equation}

\begin{example} \label{example1}
When $p=1$, $q=-2$, the critical point equations $V^+_{\theta_1}=0, V^+_{\theta_2}=0$ has a unique solution $(\theta_1^0,\theta_2^0)$  lies in region $D_0$, and
the equations $V^-_{\theta_1}=0, V^-_{\theta_2}=0$ has a unique solution $(-\theta_1^0,\theta_2^0)$  lies in region $D_0$, where
\begin{align} \label{formula-theta}
\theta_1^0&=-0.1038205182+0.1790172070\sqrt{-1} \\\nonumber
\theta_2^0&=0.1308066000+0.09218763785\sqrt{-1}. 
\end{align}
\begin{align}
z_1^0&=0.2580453976-0.19711501\sqrt{-1}, \\\nonumber
z_2^0&=0.3814962624+0.4104006092\sqrt{-1}.   
\end{align}
\begin{align}
V^+(\theta_1^0,\theta_2^0)=V^-(-\theta_1^0,\theta_2^0)=2.828122086+6.845476024\sqrt{-1}.    
\end{align}
\end{example}

\subsection{Geometric equation} \label{subsection-geomequation}
Let $W$ be the Whitehead link as shown in Figure \ref{figureWL}, 
Let $W((p_1,q_1),(p_2,q_2))$ be the hyperbolic 3-manifold obtained by doing $(\frac{p_1}{q_1},\frac{p_2}{q_2})$-Dehn surgery along the two components of $W$. 
Following the work \cite{NR90}, the edge gluing equations and Dehn filling equations for $M$ are given by
\begin{equation} \label{equation-gluing} 
\left\{ \begin{aligned}
         &\log w+\log x+\log y+\log z = -2\pi\sqrt{-1} \\
          &        \log(1-w)+\log(1-x)-\log(1-y)-\log(1-z)=0
                          \end{aligned} \right.
                          \end{equation}
and 
\begin{equation}   \label{equation-dehnfilling}
\left\{ \begin{aligned}
         & p_1u_1+q_1v_1 = \pm 2\pi\sqrt{-1} \\
          &  p_2u_2+q_2v_2 = \pm 2\pi\sqrt{-1}      
                          \end{aligned} \right.
                          \end{equation}
where 
\begin{align}
    u_1&=\log (w-1)+\log x+\log y-\log(y-1)+\pi\sqrt{-1},\\\nonumber
    u_2&=\log (w-1)+\log x+\log z-\log(z-1)+\pi\sqrt{-1},\\\nonumber
    v_1&=2\log x +2\log y+2\pi\sqrt{-1},\\\nonumber
    v_2&=2\log x +2\log z+2\pi\sqrt{-1}.
\end{align}
\begin{remark} \label{remark-conjugate}
    Note that the above equations can be viewed as the conjugated version of equations (6.1a), (6.1b), (6.3), (6.2a), (6.2b),(6.2c), (6.2d) given in \cite{NR90}. In other words, if $(w_0,x_0,y_0,z_0)$ is a solution of the above equations (\ref{equation-gluing}) and (\ref{equation-dehnfilling}), then $(\bar{w}_0,\bar{x}_0,\bar{y}_0,\bar{z}_0)$ is a solution of the corresponding equations in \cite{NR90}.   
\end{remark}

In particular, the geometric equations for  $W(p,q)$ can be written as follows
\begin{equation} \label{equation-geom1} 
\left\{ \begin{aligned}
         &\log w+\log x+\log y+\log z = -2\pi\sqrt{-1}, \\
          &        \log(1-w)+\log(1-x)-\log(1-y)-\log(1-z)=0,\\
          &2\log x +2\log y+2\pi\sqrt{-1}=0, \\
          &pu+qv=\pm 2\pi\sqrt{-1},
                          \end{aligned} \right.
                          \end{equation}
where 
\begin{align}
    u&=\log (w-1)+\log x+\log z-\log(z-1)+\pi\sqrt{-1},\\\nonumber
    v&=2\log x +2\log z+2\pi\sqrt{-1}.
\end{align}

From the first three equations of (\ref{equation-geom1}), we obtain 
\begin{align}
\left\{\begin{aligned}
    &y=-\frac{1}{x}, \\
    &w=-\frac{1}{z}, \\
    &z=x.
    \end{aligned}\right.
\end{align}
It turns out that 
\begin{align} \label{equation-u}
    u&=\log z+\log(z+1)-\log(z-1),\\ \label{equation-v}
v&=4\log \left(z\right)+2\pi \sqrt{-1}.
\end{align}

Then, the geometric equations (\ref{equation-geom1}) for $W\left(p,q\right)$ are reduced to one single equation
\begin{align} \label{equation:geometry-one}
p(\log z+\log (z+1)-\log(z-1))+q(4\log z+2\pi\sqrt{-1})=2\pi\sqrt{-1}.
\end{align}
Here we choose the equation with the left hand  $+2\pi\sqrt{-1}$, otherwise, we only need to change $(p,q)\rightarrow (-p,-q)$. 

Obviously, the equation \ref{equation:geometry-one} gives
\begin{align} \label{equation-z4p}
   z^{4q}\left(\frac{z(z+1)}{z-1}\right)^p=1.
\end{align}
\begin{example}
For $p=1$, $q=-2$,  the equation (\ref{equation:geometry-one}) has a unique solution 
\begin{align}
z^0=-0.6623589786-0.5622795125\sqrt{-1}.
\end{align}
\end{example}

\subsection{Comparison of two equations}

\begin{proposition}
The equation (\ref{equation-critform}) and the equation (\ref{equation-z4p}) are equivalent. 
\end{proposition}
\begin{proof}
Suppose $z^0$ is a solution to equation (\ref{equation-z4p}), let 
\begin{align} \label{formula-transform}
z_1^0=\frac{z^0(z^0+1)}{z^0-1}, \  z_2^0=\frac{z^0}{(z^0)^2+z^0-1},
\end{align}
by straightforward computations, we obtain $(z_1^0,z_2^0)$ satisfies the first equation in (\ref{equation-critform}). Moreover, we have
\begin{align}
\frac{z_1^0-z_2^0}{z_1^0z_2^0-1}=(z^0)^2,    
\end{align}
it follows that
\begin{align}
(z_1^0)^p\left(\frac{z_1^0-z_2^0}{z_1^0z_2^0-1}\right)^{2q}=\left(\frac{z^0(z^0+1)}{z^0-1}\right)^{p}(z^0)^{4q}=1,    
\end{align}
which is the second equation of (\ref{equation-critform}). 

Conversely, suppose $(z_1^0,z_2^0)$ is a solution to equation (\ref{equation-z4p}),  let 
\begin{align}
z^0=\frac{z_2^0(z_1^0+1)}{z_1^0z_2^0-1}, 
\end{align}
then  we have 
\begin{align}
(z^0)^2=\left(\frac{z_2^0(z_1^0+1)}{z_1^0z_2^0-1}\right)^2=\frac{z_1^0-z_2^0}{z_1^0z_2^0-1}    
\end{align}
and 
\begin{align}
\frac{z^0(z^0+1)}{z^0-1}=\frac{(2z_1^0z_2^0+z_2^0-1)z_2^0(z_1^0+1)}{(z_2^0+1)(z_1^0z_2^0-1)}=z_1^0, 
\end{align}
where we have used that $(z_1^0,z_2^0)$ satisfies the first equation in (\ref{equation-critform}).

Therefore, 
\begin{align}
(z^0)^{4q}\left(\frac{z^0(z^0+1)}{z^0-1}\right)^{p}=\left(\frac{z_1^0-z_2^0}{z_1^0z_2^0-1}\right)^{2q}(z_1^0)^p=1,  
\end{align}
it follows that $z_0$ satisfies the equation (\ref{equation-z4p}). 
\end{proof}

\begin{example}
When $p=1,q=-2$, comparing with the computations in previous Example \ref{example1}, we have
\begin{align}
    z_1^0=\frac{z^0(z^0+1)}{z^0-1}, \   z_2^0=\frac{z^0}{(z^0)^2+z^0-1}.
\end{align}
and
\begin{align}
z^0=\frac{z_2^0(z_1^0+1)}{z_1^0z_2^0-1}.    
\end{align}
\end{example}
\begin{remark} \label{remark-solution}
Note that the equations (\ref{equation-critform}) and the equation (\ref{equation-z4p}) are the exponential forms of the critical point equations (\ref{equation-crit1}), (\ref{equation-crit2})   and geometric equation (\ref{equation:geometry-one}) respectively.
\end{remark}

According to \cite{NR90}, there is a unique solution $z^0$ with $\text{Im}(z^0)<0$ (note that here we use the conjugation geometric equation, see Remark \ref{remark-conjugate}) for the geometric equation (\ref{equation:geometry-one}). 
\begin{corollary} \label{coro-critunique}
   The critical point equations (\ref{equation-crit1}), (\ref{equation-crit2}) have a unique solution $(\theta_1^0,\theta_2^0)$ with $(\text{Re}(\theta_1^0),\text{Re}(\theta_2^0))$ lies in $D_0$. 
\end{corollary}

\subsection{The critical value gives the complex volume}
Now, we reformulate a result due to Neumann-Zagier \cite{NZ85} and Yoshida \cite{Yoshida85}. We follow the notations and the statements in \cite{WongYang20-1} by Wong-Yang in our setting.  
Let $M$ be a hyperbolic 3-manifold obtained by doing a hyperbolic $\frac{p}{q}$-Dehn filling from a component $T$ of a hyperbolic link $\mathcal{L}$ in $S^3$. Suppose $m$ and $l$ are respectively the meridian and longitude of the boundary of a tubular neighborhood of this component $T$, and $u,v$  are respectively the holonomy of $m,l$, then a solution to 
\begin{align}
    pu+qv=2\pi\sqrt{-1}
\end{align}
near the complete structure gives a hyperbolic structure on the result manifold $M$. 

Let $\Phi(u)$ be the Neumann-Zagier potential function defined on the deformation space of hyperbolic structures on $S^3\setminus \mathcal{L}$ parametrized by the holonomy of the meridians $u$, which is characterized by the following differential equation
\begin{equation} \label{equation:NZpotentialfunction}
\left\{ \begin{aligned}
         \frac{\partial \Phi(u)}{\partial u} &= \frac{v}{2} \\
                  \Phi(0)&=\sqrt{-1}(Vol(S^3\setminus \mathcal{L})+\sqrt{-1}CS(S^3\setminus \mathcal{L})) \ \mod \pi^2\mathbb{Z},
                          \end{aligned} \right.
                          \end{equation}
where $S^3\setminus \mathcal{L}$ is with the complete hyperbolic metric.

We choose a curve $\Gamma$ on the boundary of a tubular neighborhood of $T$, that
is isotopic to the core curve of the filled solid torus. We choose the orientation of $\Gamma$ such that the intersection number $(pu+qv)\cdot \Gamma=1$, and let $\gamma$ be the holonomy of $\Gamma$. Then we have
\begin{align}  \label{equation-vol=NZpotential}
    Vol(M)+\sqrt{-1}CS(M)=\frac{\Phi(u)}{\sqrt{-1}}-\frac{uv}{4\sqrt{-1}}+\frac{\pi \gamma}{2}  \mod \sqrt{-1}\pi^2\mathbb{Z}.
\end{align}

\begin{proposition} \label{prop-crit=volume}
Let $(\pm \theta_1^0,\theta_2^0)$ with $(\text{Re}(\theta_1^0,\text{Re}(\theta_2^0)) \in D_{0}$ be the unique critical point of the potential function $V^\pm (p,q,\theta_1,\theta_2)$, then we have
    \begin{align}
    2\pi V^\pm (p,q;\pm\theta_1^0,\theta_2^0)\equiv Vol(W(p,q))+\sqrt{-1}CS(W(p,q)) \mod (\pi^2 \sqrt{-1}\mathbb{Z}).
\end{align}
\end{proposition}

\begin{proof}
By using the symmetry (\ref{formula-Vsym}),  we only prove the case $V^+$.  We introduce the function
\begin{align}
\tilde{V}(\theta_1,\theta_2)&=\pi\sqrt{-1}\left(\theta_1^2+\theta_2^2-\theta_2+\frac{5}{6}\right)+\frac{1}{2\pi\sqrt{-1}}\left(\text{Li}_2(e^{2\pi \sqrt{-1}(\theta_2-\theta_1)})\right.\\\nonumber &\left.+\text{Li}_2(e^{2\pi \sqrt{-1}(\theta_1+\theta_2)})+3\text{Li}_2(e^{2\pi \sqrt{-1}\theta_2})-\text{Li}_2(e^{4\pi\sqrt{-1}\theta_2})\right).     
\end{align}

Recall that
$z_1=e^{2\pi\sqrt{-1}\theta_1}$ and $u=\log \frac{z(z+1)}{z-1}$ by formula (\ref{equation-u}), under the variable transformation $z_1=\frac{z(z+1)}{z-1}$, we obtain 
\begin{align}
    u=\log\frac{z(z+1)}{z-1}=\log z_1=2\pi\sqrt{-1}\theta_1.
\end{align}

We suppose that $\theta_2(\theta_1)$ is  determined by the following equation
\begin{align} \label{formula-partialVtheta2}
&\frac{\partial \tilde{V}}{\partial \theta_2}=\pi \sqrt{-1}\left(2\theta_2-1\right)-\log\left(1-e^{2\pi \sqrt{-1}(\theta_2-\theta_1)}\right)-\log\left(1-e^{2\pi \sqrt{-1}(\theta_2+\theta_1)}\right)\\\nonumber
    &-3\log(1-e^{2\pi \sqrt{-1}\theta_2})+2\log\left(1-e^{4\pi\sqrt{-1}\theta_2}\right)=0, 
\end{align}
and define the function 
\begin{align}
\Psi(u)=2\pi\sqrt{-1}\tilde{V}(\theta_1,\theta_2(\theta_1)).    
\end{align}

Since
\begin{align}
    &\frac{d\tilde{V}(\theta_1,\theta_2(\theta_1))   }{d\theta_1}\\\nonumber
    &=\frac{\partial \tilde{V}}{\partial \theta_1}|_{(\theta_1,\theta_2(\theta_1))}+\frac{\partial \tilde{V}}{\partial \theta_2}\frac{\partial \theta_2}{\partial \theta_1}|_{(\theta_1,\theta_2(\theta_1))}\\\nonumber
    &=2\pi\sqrt{-1}\theta_1+\log(1-e^{2\pi\sqrt{-1}(\theta_2(\theta_1)-\theta_1)})-\log(1-e^{2\pi\sqrt{-1}(\theta_2(\theta_1)+\theta_1)}),
\end{align}
it follows that
\begin{align}  \label{formula-dpsiu}
\frac{d \Psi(u)}{d u}&=\frac{d\Psi(u)}{d \theta_1}\frac{d\theta_1}{du}=\frac{1}{2\pi\sqrt{-1}}\frac{d\Psi(u)}{d\theta_1}\\\nonumber
&=\frac{d\tilde{V}}{d\theta_1}=\log\left(\frac{1-e^{2\pi\sqrt{-1}(\theta_2(\theta_1)-\theta_1)}}{1-e^{2\pi\sqrt{-1}(\theta_2(\theta_1)+\theta_1)}}e^{2\pi\sqrt{-1}\theta_1}\right)\\\nonumber
&=\log\left(\frac{e^{2\pi\sqrt{-1}\theta_1}-e^{2\pi\sqrt{-1}(\theta_2(\theta_1))}}{1-e^{2\pi\sqrt{-1}(\theta_2(\theta_1)+\theta_1)}}\right).
\end{align}

Recall that $z_1=e^{2\pi\sqrt{-1}\theta_1}$ and $z_2=e^{2\pi\sqrt{-1}\theta_2(\theta_1)}$, 
under the transformation 
\begin{align}
z=\frac{z_2(z_1+1)}{z_1z_2-1},     
\end{align}
we obtain
\begin{align}
z^2=\left(\frac{z_2(z_1+1)}{z_1z_2-1}\right)^2=\frac{z_1-z_2}{z_1z_2-1},    
\end{align}
where the second ``=" is from the equality (\ref{formula-partialVtheta2}). Hence 
\begin{align}
 2\log z=\log\left(\frac{z_1-z_2}{z_1z_2-1}\right)=\log\left(\frac{z_1-z_2}{1-z_1z_2}\right)-\pi\sqrt{-1}.   
\end{align}
Therefore, by formula (\ref{formula-dpsiu}), we have
\begin{align}
\frac{d\Psi(u)}{du}
&=\log\left(\frac{z_1-z_2}{1-z_1z_2}\right)\\\nonumber
&=2\log z+\pi\sqrt{-1}\\\nonumber
&= 
\frac{v}{2}
\end{align}
where the last "=" is from formula (\ref{equation-v}).

At the initial value $u=0$, i.e. $\theta_1=0$, solving the equation (\ref{formula-partialVtheta2}) in this case, we get the unique solution given by 
\begin{align}
\theta_2(0)=\frac{\log(\frac{1+2\sqrt{-1}}{5})}{2\pi\sqrt{-1}}.    
\end{align}
So we have
\begin{align}
\Psi(0)&=2\pi\sqrt{-1}\left(\pi\sqrt{-1}\left(\theta_2(0)^2-\theta_2(0)+\frac{5}{6}\right)+\frac{1}{2\pi\sqrt{-1}}\left(\text{Li}_2(e^{2\pi \sqrt{-1}\theta_2(0)})\right.\right.\\\nonumber &\left.\left.+\text{Li}_2(e^{2\pi \sqrt{-1}\theta_2(0)})+3\text{Li}_2(e^{2\pi \sqrt{-1}\theta_2(0)})-\text{Li}_2(e^{4\pi\sqrt{-1}\theta_2(0)})\right)   \right).   
\end{align}

By using the useful dilogarithm identity shown in \cite{CZ23-1}, we obtain
\begin{align}
\Psi(0)\equiv \sqrt{-1}\left(Vol(S^3\setminus W)+\sqrt{-1}CS(S^3\setminus W)\right)  \mod \sqrt{-1}\pi^2\mathbb{Z}.    
\end{align}
Furthermore, we have 
\begin{align}
uv=u\frac{2\pi\sqrt{-1}-pu}{q}=\frac{(2\pi\sqrt{-1})^2\theta_1}{q}-\frac{p}{q}(2\pi\sqrt{-1})^2\theta_1^2.     
\end{align}
and
\begin{align}
 \gamma=-q'u+p'v=-q'u+p'\frac{ 2\pi\sqrt{-1}-pu}{q}= 2\pi\sqrt{-1}\frac{p'}{q}-\frac{u}{q}= 2\pi\sqrt{-1}\left(\frac{p'- \theta_1}{q}\right).   
\end{align}
Finally, we obtain
\begin{align}
&\frac{\Psi(u)}{\sqrt{-1}}-\frac{uv}{4\sqrt{-1}}+\frac{\pi \gamma}{2}\\\nonumber
&=2\pi\left(\pi\sqrt{-1}\left(\frac{p}{2q}+1\right)\theta_1^2-\frac{1}{q}\theta_1+\theta_2^2-\theta_2+\frac{5}{6}+\frac{p'}{2q}\right)\\\nonumber
&+\frac{1}{\sqrt{-1}}\left(\text{Li}_2(e^{2\pi \sqrt{-1}(\theta_2-\theta_1)})-\text{Li}_2(e^{-2\pi \sqrt{-1}(\theta_2+\theta_1)})+3\text{Li}_2(e^{2\pi \sqrt{-1}\theta_2})\right.\\\nonumber &\left.-\text{Li}_2(e^{4\pi\sqrt{-1}\theta_2})\right)
\end{align}    
which is equal to $2\pi V^{+}(p,q;\theta_1,\theta_2(\theta_1))$. By formula (\ref{equation-vol=NZpotential}) and (\ref{formula-Vsym}), we complete the proof of Proposition \ref{prop-crit=volume}.  
\end{proof}

\begin{lemma} \label{lemma-imtheta}
 We have
 \begin{align}
  \text{Im}(\theta_1^0)\neq 0.   
 \end{align}
\end{lemma}
\begin{proof}
   Since  $u=\log\frac{z(z+1)}{z-1}=\log z_1=2\pi\sqrt{-1}\theta_1$. Suppose $\text{Im}(\theta_1^0)= 0$,  then $u$ is purely imaginary. As a consequence, $v=\frac{2\pi\sqrt{-1}-pu}{q}$ is also purely imaginary, which implies that $\gamma=q'u-p'v$ is purely imaginary, i.e. the core curve of the filled solid torus $\Gamma$ has length zero. It is a contradiction.  
\end{proof}

\section{Asymptotic expansion of the relative Reshetikhin-Turaev invariants}
\label{Section-asympticexpansion}

The goal of this section is to estimate  each Fourier coefficients $\hat{h}_N(s,m_1,m_2)$ appearing in Proposition \ref{prop-fouriercoeff}. In Section \ref{subsection-preparation}, we establish some results which will be used in the later subsections.  In Section \ref{subsection-nneq1} 
we consider the Fourier coefficients that can be neglected. In Sections \ref{section-m_2=1}, we estimate the remained Fourier coefficients and find out that only two terms will contribute to the final form of the asymptotic expansion. At last,  we finish the proof of Theorem \ref{theorem-main} in Section \ref{subsection-final}.

\subsection{Preparations} \label{subsection-preparation}
We write the complex variables $\theta_1,\theta_2$ as $\theta_1=\theta_{1R}+\sqrt{-1}X_1$, $\theta_2=\theta_{2R}+\sqrt{-1}X_2$. 
We introduce the following function
\begin{align}
    &f(s,\theta_{1R},X_1,\theta_{2R},X_2;m_1,m_2)\\\nonumber&=\text{Re}V(s,\theta_1,\theta_2;m_1,m_2)\\\nonumber
    &=\text{Re}\left(\pi \sqrt{-1}\left(\frac{p}{2q}\theta_1^2-I(s)\frac{\theta_1}{q}-2\theta_1\theta_2-2\theta_1+2\theta_2^2+2m_1\theta_1+2m_2\theta_2\right)\right)\\\nonumber
    &+\text{Re}\left(\frac{1}{2\pi\sqrt{-1}}\left(\text{Li}_2(e^{2\pi \sqrt{-1}(\theta_2-\theta_1)})-\text{Li}_2(e^{-2\pi \sqrt{-1}(\theta_2+\theta_1)})-2\text{Li}_2(e^{-2\pi \sqrt{-1}\theta_2})\right.\right. \\\nonumber &\left.\left.+\text{Li}_2(e^{2\pi\sqrt{-1}\theta_2})+\text{Li}_2(e^{-4\pi\sqrt{-1}\theta_2})\right)\right). 
\end{align}

Then we have 
\begin{align}
    \frac{\partial f}{\partial X_1}&=\text{Re}\left(\sqrt{-1}\frac{\partial V}{\partial \theta_1}\right)=-\text{Im}\left(\frac{\partial V}{\partial \theta_1}\right)\\\nonumber
    &=-\pi\left(\frac{p}{q}\theta_{1R}-\frac{I(s)}{q}-2\theta_{2R}+2m_1\right)-\text{arg}\left(1-e^{2\pi\sqrt{-1}(\theta_2-\theta_2)}\right)\\\nonumber
    &+\text{arg}\left(1-e^{-2\pi\sqrt{-1}(\theta_2+\theta_1)}\right),
\end{align}

\begin{align}
    \frac{\partial f}{\partial X_2}&=-\pi\left(-2\theta_{1R}-2+4\theta_{2R}+2m_2\right)+\text{arg}\left(1-e^{2\pi \sqrt{-1}(\theta_2-\theta_1)}\right)\\\nonumber
    &+\text{arg}\left(1-e^{-2\pi \sqrt{-1}(\theta_2+\theta_1)}\right)+\text{arg}(1-e^{2\pi \sqrt{-1}\theta_2})-2\text{arg}\left(1+e^{-2\pi\sqrt{-1}\theta_2}\right).
\end{align}

\begin{align}
    \frac{\partial^2 f}{\partial X_1^2}=
    &=-\text{Im}\left(\sqrt{-1}V_{\theta_1\theta_1}\right)=-\text{Re} \left(V_{\theta_1\theta_1}\right)\\\nonumber
    &=2\pi\left( \text{Im}\left(\frac{1}{1-e^{2\pi\sqrt{-1}(\theta_2-\theta_1)}}\right)-\text{Im}\left(\frac{1}{1-e^{-2\pi\sqrt{-1}(\theta_2+\theta_1)}}\right)\right)\\\nonumber
    &=2\pi\left( \text{Im}\left(\frac{1}{1-e^{2\pi\sqrt{-1}(\theta_2-\theta_1)}}\right)+\text{Im}\left(\frac{1}{1-e^{2\pi\sqrt{-1}(\theta_2+\theta_1)}}\right)\right),
\end{align}

\begin{align}
    \frac{\partial^2 f}{\partial X_1\partial X_2}=2\pi\left( -\text{Im}\left(\frac{1}{1-e^{2\pi\sqrt{-1}(\theta_2-\theta_1)}}\right)+\text{Im}\left(\frac{1}{1-e^{2\pi\sqrt{-1}(\theta_2+\theta_1)}}\right)\right),
\end{align}

\begin{align} \label{formula-fzz}
    \frac{\partial^2 f}{\partial X_2^2}&=2\pi\left(\text{Im}\left(\frac{1}{1-e^{2\pi\sqrt{-1}(\theta_2-\theta_1)}}\right)+\text{Im}\left(\frac{1}{1-e^{2\pi\sqrt{-1}(\theta_2+\theta_1)}}\right)\right.\\\nonumber
    &\left.+\text{Im}\left(\frac{1}{1-e^{2\pi\sqrt{-1}\theta_2}}\right)-2\text{Im}\left(\frac{1}{1+e^{2\pi\sqrt{-1}\theta_2}}\right)\right).
\end{align}

Let 
\begin{align}
 a=\text{Im}\left(\frac{1}{1-e^{2\pi\sqrt{-1}(\theta_2-\theta_1)}}\right)   =\frac{\sin(2\pi(\theta_{2R}-\theta_{1R}) )}{e^{2\pi(X_2-X_1)}+e^{-2\pi(X_2-X_1)}-2\cos(2\pi(\theta_{2R}-\theta_{1R}))}
\end{align}
\begin{align}
b=\text{Im}\left(\frac{1}{1-e^{2\pi\sqrt{-1}(\theta_2+\theta_1)}}\right)=\frac{\sin(2\pi(\theta_{2R}+\theta_{1R}) )}{e^{2\pi(X_2+X_1)}+e^{-2\pi(X_2+X_1)}-2\cos(2\pi(\theta_{2R}+\theta_{1R}))}    
\end{align}
\begin{align}
  c=\text{Im}\left(\frac{1}{1-e^{2\pi\sqrt{-1}\theta_2}}\right)=\frac{\sin(2\pi \theta_{1R} )}{e^{2\pi X_2}+e^{-2\pi X_2}-2\cos(2\pi \theta_{2R})}  
\end{align}
\begin{align}
    d=-\text{Im}\left(\frac{1}{1+e^{2\pi\sqrt{-1}\theta_2}}\right)=\frac{\sin(2\pi \theta_{1R} )}{e^{2\pi X_2}+e^{-2\pi X_2}+2\cos(2\pi \theta_{2R})}  
\end{align}

Then the Hessian matrix of $f$ is given by 
\begin{align}
  Hess(f)= 2\pi\begin{pmatrix}
     a+b & b-a \\
     b-a & a+b+c+2d
    \end{pmatrix}.  
\end{align}

For $0<\text{Re}(\theta_2\pm \theta_1)<\frac{1}{2}$, $0<\text{Re} (\theta_2)<\frac{1}{2}$, we have $a>0, b>0, c>0$, $d>0$.  
Then it is easy to get that $a+b>0$ and $\det(\text{Hess}(f))>0$. 
Hence, we let 
\begin{align}
D_{H}=\{(\theta_{1R},\theta_{2R})\in \mathbb{R}^2| 0<\theta_{2R}\pm \theta_{1R}<\frac{1}{2} \},    
\end{align}
which is shown in Figure \ref{figure:DH}. 

We have
\begin{proposition} \label{Prop-Hessian}
The Hessian matrix of $f$ is positive on $D_H$.      
\end{proposition}

\begin{figure}[!htb]
\begin{align*} 
\raisebox{-15pt}{
\includegraphics[width=230 pt]{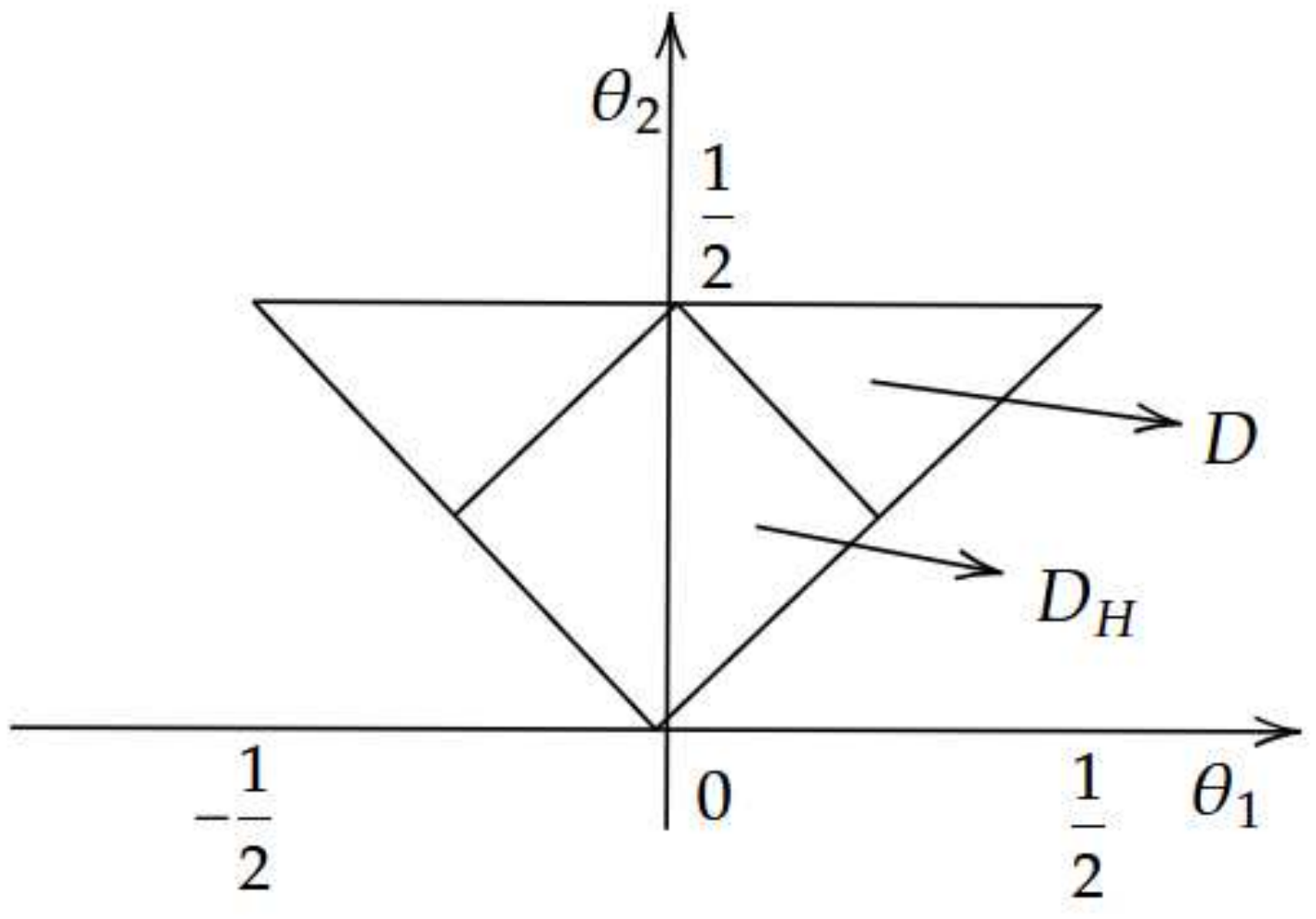}}.
\end{align*}
\caption{The region $D_H$}  \label{figure:DH} 
\end{figure}

\subsection{Fourier coefficients that can be neglected} \label{subsection-nneq1}
Motivated by Lemma \ref{lemma-Li2}, we introduce the following  function for $(\theta_{1R},\theta_{2R})\in D$.

\begin{equation} 
F(X_1,X_2;m_1,m_2)=\left\{ \begin{aligned}
         &0  &  \ (\text{if} \ X_2-X_1\geq 0) \\
         &\left(\theta_{2R}-\theta_{1R}-\frac{1}{2}\right)(X_2-X_1) & \ (\text{if} \ X_2-X_1<0)
                          \end{aligned} \right.
                          \end{equation}
\begin{equation*}
 +\left\{ \begin{aligned}
         &\left(\frac{1}{2}-(\theta_{2R}+\theta_{1R})\right)(X_2+X_1)  &  \ (\text{if} \ X_2+X_1\geq 0) \\
         &0 & \ (\text{if} \ X_2+X_1<0)
                          \end{aligned} \right.
                          \end{equation*}
\begin{equation*}
 +\left\{ \begin{aligned}
         &2\left(\frac{1}{2}-\theta_{2R}\right)X_2  &  \ (\text{if} \ X_2\geq 0) \\
         &0 & \ (\text{if} \ X_2<0)
                          \end{aligned} \right.
                          \end{equation*}
                          
                          \begin{equation*}
 +\left\{ \begin{aligned}
         &0  &  \ (\text{if} \ X_2\geq 0) \\
         &\left(\theta_{2R}-\frac{1}{2}\right)X_2 & \ (\text{if} \ X_2<0)
                          \end{aligned} \right.
                          \end{equation*}
\begin{equation*}
 +\left\{ \begin{aligned}
         &\left(4\theta_{2R}-1\right)X_2  &  \ (\text{if} \ X_2\geq 0) \\
         &0 & \ (\text{if} \ X_2<0)
                          \end{aligned} \right.
                          \end{equation*}        
\begin{equation*}
  -\frac{p}{2q}\theta_{1R} X_2+\frac{I(s)}{2q}X_1+(\theta_{1R}X_2+\theta_{2R}X_1)+X_2-2\theta_{2R}X_2-m_1X_1-m_2X_2
                          \end{equation*}

So we have
\begin{equation} 
F(X_1,X_2;m_1,m_2)=\left\{ \begin{aligned}
         &0  &  \ (\text{if} \ X_2-X_1\geq 0) \\
         &\left(\theta_{2R}-\theta_{1R}-\frac{1}{2}\right)(X_2-X_1) & \ (\text{if} \ X_2-X_1<0)
                          \end{aligned} \right.
                          \end{equation}
\begin{equation*}
 +\left\{ \begin{aligned}
         &\left(\frac{1}{2}-(\theta_{2R}+\theta_{1R})\right)(X_2+X_1)  &  \ (\text{if} \ X_2+X_1\geq 0) \\
         &0 & \ (\text{if} \ X_2+X_1<0)
                          \end{aligned} \right.
                          \end{equation*}
\begin{equation*}
 +\left\{ \begin{aligned}
         &2\theta_{2R}X_2  &  \ (\text{if} \ X_2\geq 0) \\
         &\left(\theta_{1R}-\frac{1}{2}\right)X_2 & \ (\text{if} \ X_2<0)
                          \end{aligned} \right.
                          \end{equation*}
     
\begin{equation*}
  \left(-\frac{p}{2q}\theta_{1R}+\theta_{2R}+\frac{I(s)}{2q}-m_1\right) X_1+(1+\theta_{1R}-2\theta_{2R}-m_2)X_2
                          \end{equation*}

  Note that $F(X_1,X_2;m_1,m_2)$ is a piecewise linear function,  we subdivide the plane $\{(X_1,X_2)\in \mathbb{R}^2\}$ into six regions to discuss the asymptotic property of this function.

(I) When $X_2\geq 0$ and $X_2-X_1\leq 0$, then 
\begin{align}
    &F(X_1,X_2;m_1,m_2)\\\nonumber
    &=(1-2\theta_{2R})X_1+(2\theta_{2R}-2\theta_{1R})X_2\\\nonumber
     &+\left(-\frac{p}{2q}\theta_{1R}+\theta_{2R}+\frac{I(s)}{2q}-m_1\right) X_1+(1+\theta_{1R}-2\theta_{2R}-m_2)X_2\\\nonumber
     &=\left(-\frac{p}{2q}\theta_{1R}-\theta_{2R}+1-m_1+\frac{I(s)}{2q}\right)X_1+(1-\theta_{2R}-m_2)X_2.
\end{align}

(II) When $X_2-X_1\geq 0$ and $X_2+X_1\geq 0$, then 
\begin{align}
    &F(X_1,X_2;m_1,m_2)\\\nonumber
    &=\left(\frac{1}{2}-(\theta_{2R}+\theta_{1R})\right)X_1+\left(\frac{1}{2}+\theta_{2R}-\theta_{1R}\right)X_2\\\nonumber
     &+\left(-\frac{p}{2q}\theta_{1R}+\theta_{2R}+\frac{I(s)}{2q}-m_1\right) X_1+(1+\theta_{1R}-2\theta_{2R}-m_2)X_2\\\nonumber
     &=\left(-\left(\frac{p}{2q}+1\right)\theta_{1R}+ \frac{I(s)}{2q}+\frac{1}{2}-m_1\right)X_1+\left(\frac{3}{2}-\theta_{2R}-m_2\right)X_2.
\end{align}

Hence, for $m_2\geq 2$, then $\frac{3}{2}-\theta_{2R}-m_2<0$, it follows that $F(\theta_1,\theta_2;m_1,m_2)\rightarrow -\infty$ as $\theta_2\rightarrow +\infty$.

(III) When $X_2\geq 0$ and $X_2+X_1\leq 0$, then 
\begin{align}
    &F(X_1,X_2;m_1,m_2)\\\nonumber
    &=2\theta_{2R}X_2+\left(-\frac{p}{2q}\theta_{1R}+\theta_{2R}+\frac{I(s)}{2q}-m_1\right) X_1+(1+\theta_{1R}-2\theta_{2R}-m_2)X_2\\\nonumber
     &=\left(-\frac{p}{2q}\theta_{1R}+\theta_{2R}+\frac{I(s)}{2q}-m_1\right)X_1+(1+\theta_{2R}-m_2)X_2
\end{align}

(IV) When $X_2\leq 0$ and $X_2-X_1\geq 0$, then 
\begin{align}
    &F(X_1,X_2;m_1,m_2)\\\nonumber
    &=\left(\theta_{2R}-\frac{1}{2}\right)X_2+\left(-\frac{p}{2q}\theta_{1R}+\theta_{2R}+\frac{I(s)}{2q}-m_1\right) X_1+(1+\theta_{1R}-2\theta_{2R}-m_2)X_2\\\nonumber
     &=\left(-\frac{p}{2q}\theta_{1R}+\theta_{2R}+\frac{I(s)}{2q}-m_1\right)X_1+\left(\frac{1}{2}+\theta_{1R}-\theta_{2R}-m_2\right)X_2
\end{align}

(V) When $X_2+X_1\leq 0$ and $X_2-X_1\leq 0$, then 
\begin{align}
    &F(X_1,X_2;m_1,m_2)\\\nonumber
    &=\left(\frac{1}{2}-\theta_{2R}+\theta_{1R}\right)X_1+\left(2\theta_{2R}-1-\theta_{1R}\right)X_2\\\nonumber
     &+\left(-\frac{p}{2q}\theta_{1R}+\theta_{2R}+\frac{I(s)}{2q}-m_1\right)X_1+(1+\theta_{1R}-2\theta_{2R}-m_2)X_2\\\nonumber
     &=\left(\left(1-\frac{p}{2q}\right)\theta_{1R}+\frac{I(s)}{2q}+\frac{1}{2}-m_1\right)X_1-m_2X_2
\end{align}
Hence, for $m_2\leq -1$, then $F(X_1,X_2;m_1,m_2)\rightarrow -\infty$ as $X_2\rightarrow -\infty$. 

(VI) When $X_2\leq 0$ and $X_2+X_1\geq 0$, then 
\begin{align}
    &F(X_1,X_2;m_1,m_2)\\\nonumber
    &=\left(1-2\theta_{2R}\right)X_1+\left(\theta_{2R}-2\theta_{1R}-\frac{1}{2}\right)X_2\\\nonumber
     &+\left(-\frac{p}{2q}\theta_{1R}+\theta_{2R}+\frac{I(s)}{2q}-m_1\right)X_1+(1+\theta_{1R}-2\theta_{2R}-m_2)X_2\\\nonumber
     &=\left(-\frac{p}{2q}\theta_{1R}-\theta_{2R}+1+\frac{I(s)}{2q}-m_1\right)X_1+\left(-\theta_{1R}-\theta_{2R}+\frac{1}{2}-m_2\right)X_2
\end{align}

Therefore, by cases (II) and (V), we obtain
\begin{lemma}
    For any $(\theta_{1R},\theta_{2R})\in D$, when $m_2\geq 2$, for a fixed $X_1$, $F(X_1,X_2;m_1,m_2)$ is a decreasing function with respect to $X_2$, and 
    \begin{align}
        F(X_1,X_2;m_1,m_2)\rightarrow -\infty, \ \text{as} \ X_2\rightarrow +\infty.
    \end{align}
    when $m_2\leq -1$, for a fixed $X_1$, $F(X_1,X_2;m_1,m_2)$ is a decreasing function with respect to $X_2$, and 
      \begin{align}
        F(X_1,X_2;m_1,m_2)\rightarrow -\infty, \ \text{as} \ X_2\rightarrow -\infty.
    \end{align}
\end{lemma}
As a consequence of Lemma \ref{lemma-Li2}, we obtain
\begin{corollary} \label{corollary-m_2m_1}
  For any $(\theta_{1R},\theta_{2R})\in D$ and $s\in \{0,...,2|q|-1\}$,
  
  (i) when $m_2\geq 2$,  for a fixed $X_1$,
    \begin{align}
        \text{Re}V(s,\theta_{1R}+\sqrt{-1}X_1,\theta_{2R}+\sqrt{-1}X_2;m_1,m_2)\rightarrow -\infty, \ \text{as} \ X_2\rightarrow +\infty.
    \end{align}
    
   (ii) when $m_2\leq -1$,  for a fixed $X_1$,
      \begin{align}
       \text{Re}V(s,\theta_{1R}+\sqrt{-1}X_1,\theta_{2R}+\sqrt{-1}X_2;m_1,m_2)\rightarrow -\infty, \ \text{as} \ X_2\rightarrow -\infty.
    \end{align}   
   
\end{corollary}

Furthermore, we have
\begin{lemma} \label{lemma-m2=0}
For any $(\theta_{1R},\theta_{2R})\in D$ and $s\in \{0,...,2|q|-1\}$, 
when $m_2=0$, for a fixed $X_1=0$,        \begin{align}
       \text{Re}V(s,\theta_{1R},\theta_{2R}+\sqrt{-1}X_2;m_1,0)\rightarrow 0, \ \text{as} \ X_2\rightarrow -\infty.
    \end{align}  
\end{lemma}
\begin{proof}
Since
\begin{align}
&\text{Re}V(s,\theta_{1R},\theta_{2R}+\sqrt{-1}X_2;m_1,0)\\\nonumber
  &=2\pi (\theta_{1R}-2\theta_{2R}+1)X_2+\text{Re}\left(\frac{1}{2\pi\sqrt{-1}}\left(\text{Li}_2(e^{2\pi\sqrt{-1}(\theta_{2R}-\theta_{1R}+\sqrt{-1}X_2)})\right.\right.\\\nonumber
  &-\left.\left.\text{Li}_2(e^{-2\pi\sqrt{-1}(\theta_{2R}+\theta_{1R}+\sqrt{-1}X_2)})-2\text{Li}_2(e^{-2\pi\sqrt{-1}(\theta_{2R}+\sqrt{-1}X_2)})\right.\right.
  \\\nonumber
  &+\left.\left.\text{Li}_2(e^{2\pi\sqrt{-1}(\theta_{2R}+\sqrt{-1}X_2)})+\text{Li}_2(e^{-4\pi\sqrt{-1}(\theta_{2R}+\sqrt{-1}X_2)})
  \right)\right).
\end{align}
By using the formula (\ref{formula-key}),  we obtain
\begin{align}
    &\text{Re}V(s,\theta_{1R},\theta_{2R}+\sqrt{-1}X_2;m_1,0)\\\nonumber
    &=\text{Re}\left(\frac{1}{2\pi\sqrt{-1}}\left(\text{Li}_2(e^{2\pi X_2}e^{-2\pi\sqrt{-1}(\theta_{2R}-\theta_{1R})})\text{Li}_2(e^{2\pi X_2}e^{-2\pi\sqrt{-1}(\theta_{2R}+\theta_{1R})})\right.\right.
  \\\nonumber
  &-2\text{Li}_2(e^{2\pi X_2}e^{-2\pi\sqrt{-1}\theta_{2R}})+\left.\left.\text{Li}_2(e^{2\pi X_2}e^{-2\pi\sqrt{-1}\theta_{2R}})+\text{Li}_2(e^{4\pi X_2}e^{-4\pi\sqrt{-1}\theta_{2R}})
  \right)\right).
\end{align}
Then it is easy to see that 
\begin{align}
\lim_{X_2\rightarrow \infty}\text{Re} V(s,\theta_{1R},\theta_{2R}+\sqrt{-1}X_2;m_1,0)=0.    
\end{align}

\end{proof}

\begin{proposition} \label{prop-neglected}
    When $m_2\geq 2$ or $m_2\leq 0$, then for any $s\in\{0,...,|q|-1\}$ and $m_1\in \mathbb{Z}$, there exists $\epsilon>0$ such that
    \begin{align}
        \hat{h}_N(m_1,m_2)=O(e^{\left(N+\frac{1}{2}\right)(\zeta_\mathbb{R}(p,q))-\epsilon)}).
    \end{align}
\end{proposition}
    \begin{proof}
    We note that $V_N(s,\theta_1,\theta_2;m_1,m_2)$ uniformly converges to $V(s,\theta_1,\theta_2;m_1,m_2)$, we show the existence of a homotopy $D_{\delta}$ ($0\leq \delta\leq \delta_0$) between $D_0$ and $D_{\delta_0}$ and such that 
    \begin{align}
        D_{\delta_0}\subset \{(\theta_1,\theta_2)\in \mathbb{C}^2| \text{Re} V(s,\theta_1,\theta_2;m_1,m_2)<\zeta_{\mathbb{R}}(p,q)-\epsilon\},  \label{formula-Ddelta}\\ 
        \partial D_{\delta}\subset \{(\theta_1,\theta_2)\in \mathbb{C}^2| \text{Re} V(s,\theta_1,\theta_2;m_1,m_2)<\zeta_{\mathbb{R}}(p,q)-\epsilon\}, \label{formula-partialD}
    \end{align}

    For each fixed $(\theta_{1},\theta_{2})\in D_0$,  we move $(X_1,X_2)$ from $(0,0)$ along the flow $(0,-\frac{\partial f}{\partial X_2})$. Then  the value of $\text{Re} V(s,\theta_{1R}+0\sqrt{-1},\theta_{2R}+X_2\sqrt{-1};m_1,m_2)$ monotonically decreases. As for (\ref{formula-partialD}), since $\partial D_0\subset \{(\theta_{1},\theta_{2})\in \mathbb{C}^2| \text{Re} V(s,\theta_{1},\theta_{2})<\zeta_{\mathbb{R}}(p,q)-\epsilon\}$ and the value of $\text{Re} V$ monotonically decreases, hence (\ref{formula-partialD}) holds. As for (\ref{formula-Ddelta}), since the value of $\text{Re} V$ uniformly goes to $-\infty$ by Corollary \ref{corollary-m_2m_1} or goes to $0$ by Lemma \ref{lemma-m2=0}, (\ref{formula-Ddelta}) holds for sufficiently large $\delta_0$. Therefore, such a required homotopy exists.  
\end{proof}

Proposition \ref{prop-neglected} 
shows that the Fourier coefficients $\hat{h}_{N}(s,m_1,m_2)$  with $m_2\neq 1$  can be neglected when we study the asymptotic expansion for $\bar{J}_{N}(W(p,q);t)$. Hence in the following, we focus on the case $m_2=1$.

\subsection{The cases $m_2=1$}
\label{section-m_2=1}
\subsubsection{The cases $(s,m_1)=(s^\pm,m_1^{\pm})$}
Recall that
\begin{align}
&V_{N}(s^\pm,m_1^\pm,1,\theta_1,\theta_2)\\\nonumber
&=\pi\sqrt{-1}\left(\frac{5}{6}+\theta_2^2-\theta_2+\left(\frac{p}{2q}+1\right)\theta_1^2\mp \frac{\theta_1}{q}-\frac{K(s^+)}{2}-\frac{2}{N+\frac{1}{2}}+\frac{\theta_2}{N+\frac{1}{2}}-\frac{15}{2(N+\frac{1}{2})^2}\right)\\\nonumber
&+\frac{1}{N+\frac{1}{2}}\left(\log 2+\varphi_{N}\left(\theta_1+\theta_2+\frac{\frac{1}{2}}{N+\frac{1}{2}}\right)+\varphi_{N}\left(-\theta_1+\theta_2+\frac{\frac{1}{2}}{N+\frac{1}{2}}\right)\right.\\\nonumber
&\left. \varphi_N\left(\theta_2\right)+\varphi_{N}\left(\theta_2+\frac{\frac{1}{2}}{N+\frac{1}{2}}\right)+\varphi_{N}\left(\theta_2+\frac{1}{N+\frac{1}{2}}\right)-\varphi_{N}\left(2\theta_2+\frac{\frac{1}{2}}{N+\frac{1}{2}}\right)\right),
\end{align}
and
\begin{align} 
&V^\pm(p,q;\theta_1,\theta_2)\\\nonumber
&=V(s^{\pm},m_1^{\pm},1,\theta_1,\theta_2)+\left(\frac{K(s^{\pm})}{2}+\frac{p'}{2q}\right)\pi\sqrt{-1}\\\nonumber
&=\pi\sqrt{-1}\left(\left(\frac{p}{2q}+1\right)\theta_1^2\mp\frac{\theta_1}{q}+\theta_2^2-\theta_2+\frac{5}{6}+\frac{p'}{2q}\right)\\\nonumber
&+\frac{1}{2\pi\sqrt{-1}}\left(\text{Li}_2(e^{2\pi\sqrt{-1}(\theta_1+\theta_2)})+\text{Li}_2(e^{2\pi\sqrt{-1}(-\theta_1+\theta_2)})+3\text{Li}_2(e^{2\pi\sqrt{-1}\theta_2})-\text{Li}_2(e^{4\pi\sqrt{-1}\theta_2})\right).
\end{align}

\begin{proposition} \label{prop-VNe-xpansion}
For any $L>0$, in the region
\begin{align}
\{(\theta_1,\theta_2)\in \mathbb{C}^2|(\text{Re}(\theta_1),\text{Re}(\theta_1))\in D'_{0}, |\text{Im}(\theta_1)|<L, |\text{Im}(\theta_2)|<L\},    
\end{align}
we have 
\begin{align} \label{formula-VNexpansion}
    &V_{N}(s^\pm,\theta_1,\theta_2,m_1^\pm,1)\\\nonumber
    &=V(s^\pm,\theta_1,\theta_2,m_1^\pm,1)+\frac{1}{N+\frac{1}{2}}\left(\pi\sqrt{-1}\left(-2+\theta_2\right)+\log 2 -\frac{3}{2}\log(1-e^{2\pi\sqrt{-1}\theta_2})\right.\\\nonumber
    &\left.-\frac{1}{2}\log(1-e^{2\pi\sqrt{-1}(\theta_2+\theta_1)})-\frac{1}{2}\log(1-e^{2\pi\sqrt{-1}(\theta_2-\theta_1)})+\frac{1}{2}\log(1-e^{4\pi\sqrt{-1}\theta_2})\right)\\\nonumber
    &+\frac{1}{(N+\frac{1}{2})^2}w_{N}(\theta_1,\theta_2)
\end{align}
with $|w_{N}(\theta_1,\theta_2)|$  bounded from above by a constant independent of $N$. 
\end{proposition}
\begin{proof}
By using Taylor expansion, together with Lemma \ref{lemma-varphixi3}, we have

\begin{align}
    &\varphi_N\left(\theta_2+\frac{1}{N+\frac{1}{2}}\right)\\\nonumber
    &=\varphi_N(\theta_2)+\varphi'_N(\theta_2)\frac{1}{N+\frac{1}{2}}+\frac{\varphi''_{N}(\theta_2)}{2}\frac{1}{(N+\frac{1}{2})^2}+O\left(\frac{1}{(N+\frac{1}{2})^2}\right)\\\nonumber
    &=\frac{N+\frac{1}{2}}{2\pi\sqrt{-1}}\text{Li}_2(e^{2\pi\sqrt{-1}\theta_2})-\frac{\pi\sqrt{-1}}{6(2N+1)}\frac{e^{2\pi\sqrt{-1}\theta_2}}{1-e^{2\pi\sqrt{-1}\theta_2}}\\\nonumber
    &-\log(1-e^{2\pi\sqrt{-1}\theta_2})+\frac{\pi\sqrt{-1}}{(N+\frac{1}{2})}\frac{e^{2\pi\sqrt{-1}\theta_2}}{1-e^{2\pi\sqrt{-1}\theta_2}}+O\left(\frac{1}{(N+\frac{1}{2})^2}\right). 
    \end{align}
Then, we expand $\varphi_{N}\left(\theta_1+\theta_2+\frac{\frac{1}{2}}{N+\frac{1}{2}}\right)$, $\varphi_{N}\left(-\theta_1+\theta_2+\frac{\frac{1}{2}}{N+\frac{1}{2}}\right)$, $\varphi_{N}\left(\theta_2\right)$, $\varphi_{N}\left(\theta_2+\frac{\frac{1}{2}}{N+\frac{1}{2}}\right)$ and $\varphi_{N}\left(2\theta_2+\frac{\frac{1}{2}}{N+\frac{1}{2}}\right)$ similarly, we obtain 
the formula (\ref{formula-VNexpansion}), where
\begin{align}
    w_{N}(\theta_1,\theta_2)&=\frac{\pi\sqrt{-1}}{6}\left(6\frac{e^{2\pi\sqrt{-1}\theta_2}}{1-e^{2\pi\sqrt{-1}\theta_2}}-\frac{e^{4\pi\sqrt{-1}\theta_2}}{1-e^{4\pi\sqrt{-1}\theta_2}}\right.\\\nonumber
    &\left.+\frac{e^{2\pi\sqrt{-1}(\theta_2+\theta_1)}}{1-e^{2\pi\sqrt{-1}(\theta_2+\theta_1)}}+\frac{e^{2\pi\sqrt{-1}(\theta_2-\theta_1)}}{1-e^{2\pi\sqrt{-1}(\theta_2-\theta_1)}}-45+O\left(\frac{1}{N+\frac{1}{2}}\right)\right).
\end{align}

\end{proof}

\begin{proposition} \label{prop-D0''}
There exists $\epsilon>0$, such that
\begin{align}
  \int_{D''_{0}}\psi(\theta_1,\theta_2)\sin\left(\frac{\theta_1\pi }{q}-J(s^\pm)\pi\right) e^{(N+\frac{1}{2})\left(V_N(s^\pm,\theta_1,\theta_2;m^\pm_1,1)\right)}d\theta_1d\theta_2=O\left(e^{(N+\frac{1}{2})\left(\zeta_{\mathbb{R}}(p,q)-\epsilon\right)}\right).  
\end{align}
\end{proposition}
\begin{proof}
By the definition of the region $D''_0$, we have
\begin{align} \label{formula-D0''}
&|\int_{D''_{0}}\psi(\theta_1,\theta_2)\sin\left(\frac{\theta_1\pi }{q}-J(s^\pm)\pi\right) e^{(N+\frac{1}{2})\left(V_N(s^+,\theta_1,\theta_2;m_1^\pm,1)\right)}d\theta_1d\theta_2|\\\nonumber
&=|\int_{-\frac{1}{4}}^{-c_0}\int_{D_0''(c)} \psi(\theta_1,\theta_2)\sin\left(\frac{\theta_1\pi }{q}-J(s^\pm)\pi\right) e^{(N+\frac{1}{2})\left(V_N(s^\pm,\theta_1,\theta_2;m_1^\pm,1)\right)} d\theta_2d\theta_1|\\\nonumber
&+|\int_{c_0}^{\frac{1}{4}}\int_{D_0''(c)} \psi(\theta_1,\theta_2)\sin\left(\frac{\theta_1\pi }{q}-J(s^\pm)\pi\right) e^{(N+\frac{1}{2})\left(V_N(s^\pm,\theta_1,\theta_2;m_1^\pm,1)\right)} d\theta_2d\theta_1|\\\nonumber
&\leq \int_{-\frac{1}{4}}^{-c_0}|\int_{D_0''(c)} \psi(\theta_1,\theta_2)\sin\left(\frac{\theta_1\pi }{q}-J(s^\pm)\pi\right) e^{(N+\frac{1}{2})\left(V_N(s^\pm,\theta_1,\theta_2;m_1^\pm,1)\right)} d\theta_2|d\theta_1\\\nonumber
&+\int_{c_0}^{\frac{1}{4}}|\int_{D_0''(c)} \psi(\theta_1,\theta_2)\sin\left(\frac{\theta_1\pi }{q}-J(s^\pm)\pi\right) e^{(N+\frac{1}{2})\left(V_N(s^\pm,\theta_1,\theta_2;m_1^\pm,1)\right)} d\theta_2|d\theta_1
\end{align}     
By Proposition \ref{prop-D0'c}, we have 
\begin{align}
|\int_{D''_{0}(c)}\psi(\theta_1,\theta_2)\sin\left(\frac{\theta_1\pi }{q}-J(s^\pm)\pi\right) e^{(N+\frac{1}{2})\left(V_N(s^\pm,\theta_1,\theta_2;m_1^\pm,1)\right)}d\theta_2|<C\left(e^{(N+\frac{1}{2})\left(\zeta_{\mathbb{R}}(p,q)-\epsilon\right)}\right),   
\end{align}
for some constant $C$ independent of $c$. By formula (\ref{formula-D0''}), Proposition \ref{prop-D0''} is proved.
\end{proof}

Proposition \ref{prop-D0''} implies that the integral over $D''$ can be neglected. So in the following, we consider the integral over the region $D'_0$.  We use the construction shown in \cite{WongYang20-1}. Let $(\pm \theta_1^0,\theta_2^0)$ be the unique critical point of $V^\pm$ in $D_0$. Let $S^\pm =S_{top}^\pm \cup S_{side}^\pm \cup (D'_{\frac{\delta}{2}}\setminus D'_{\delta})$ be the union of $D'_{\frac{\delta}{2}}\setminus D_{\delta}$ with two surfaces 
\begin{align}
S^\pm_{top}=\{(\theta_1,\theta_2)\in D'_{\delta\mathbb{C}}|(\text{Im}(\theta_1),\text{Im}(\theta_2))=(\pm \text{Im}(\theta_1^0),\text{Im}(\theta_2^0))\} 
\end{align}
and 
\begin{align}
S_{side}^\pm=\{(\theta_{1}\pm \sqrt{-1}t\text{Im}(\theta_1^0),\theta_{2}+\sqrt{-1}t\text{Im}(\theta_2^0))|(\theta_{1},\theta_{2})\in \partial D'_{\delta},t\in [0,1]\}.    
\end{align}

By Proposition \ref{Prop-Hessian}, the Hessian matrix $\text{Hess}(\text{Re}V^\pm)$ of the function $\text{Re}V^\pm$ in $(\text{Re}(\theta_1),\text{Re}(\theta_2))$  is positive definite in the region $D'_0$,  since $\text{Re}V^\pm$ is harmonic,  $\text{Hess}(\text{Re}V^\pm)$ of the function $\text{Re}V^\pm$ in $(\text{Im}(\theta_1),\text{Im}(\theta_2))$ is negative definite. Hence $\text{Re}V^\pm$ is strictly concave down in $(\text{Re}(\theta_1),\text{Re}(\theta_2))$. Therefore, $(\pm\theta_1^0,\theta_2^0)$ is the only absolute maximum on $S_{top}^\pm$. 

We introduce the following saddle point theorem which will be used to calculate the asymptotic expansion of $\hat{h}_{N}(s^\pm, m_1^\pm,1)$. 
\begin{theorem} \cite{AGP} \label{Theorem-Fedo}
    Let $m \geq 1$ be an integer, and $S$
 an $m$-dimensional smooth
compact real sub-manifold of $\mathbb{C}^m$ with connected boundary. We denote $z = (z^1,..., z^m)\in \mathbb{C}^m$ and
and $dz = dz^1\cdots dz^m$. Let $g(z)$ and $V(z)$ be two complex-valued functions analytic
on a domain $D$ such that $S\subset D\subset \mathbb{C}^m$. We consider the integral
\begin{align}
    \int_{S}g(z)e^{\lambda V(z)}dz
\end{align}
with parameter $\lambda\in \mathbb{R}$. 

Assume that $\max_{z\in S}ReV(z)$ is attained only at a point $z_0$, which is an interior point of $S$ and a simple saddle point of $S$ (i.e. $\nabla V(z_0)=0$, $\det Hess V(z_0)\neq 0$ ), then as $\lambda\rightarrow +\infty$, there is the asymptotic expansion 
\begin{align} \label{formula-integral}
    \int_{S}g(z)e^{\lambda V(z)}dz=\left(\frac{2\pi }{\lambda}\right)^{\frac{m}{2}}\frac{e^{\lambda V(z_0)}}{\sqrt{-\det \text{Hess}(V)(z_0)}}\left(g(z_0)+\sum_{k=1}c_k\lambda^{-k}\right)
\end{align}
where the $c_k$ are complex numbers and the choice of branch for the root $\sqrt{-\det \text{Hess}(V)(z_
0)}$ depends on the orientation of the contour $S$.
\end{theorem}

\begin{proposition}  \label{prop-s+-}
We have the following asymptotic expansion: 
\begin{align}
&|\int_{D'_{0}}\psi(\theta_1,\theta_2)\sin\left(\frac{\theta_1\pi }{q}-J(s^\pm)\pi\right) e^{(N+\frac{1}{2})\left(V_N(s^\pm,\theta_1,\theta_2;m_1^\pm,1)\right)}d\theta_1d\theta_2|\\\nonumber
&=\frac{2\sin\left(\frac{\pm\theta_1^0\pi}{q}-J(s^\pm)\pi\right)e^{-(N+\frac{1}{2})\left(\frac{K(s^\pm)}{2}+\frac{p'}{2q}\right)\pi\sqrt{-1}}}{(N+\frac{1}{2})\sqrt{1-(z_2^0)^2}\sqrt{H(p,q;z_1^0,z_2^0)}}e^{(N+\frac{1}{2})\zeta(p,q)}\left(1+O\left(\frac{1}{N+\frac{1}{2}}\right)\right).
\end{align}
\end{proposition}
\begin{proof}
By the analyticity, the integral on the region $D'_0$ is equal to the integral on $S^\pm=S^\pm_{top}\cup S^\pm_{side}$. By Proposition \ref{prop-VNe-xpansion}, we have
\begin{align}
&e^{(N+\frac{1}{2})V_N(s^+,\theta_1,\theta_2,m_1^+,1)}\\\nonumber
&= 2e^{(N+\frac{1}{2})\left(V^\pm (\theta_1,\theta_2;m_1^\pm ,1)+\frac{w_N(\theta_1,\theta_2)}{(N+\frac{1}{2})^2}\right)}\\\nonumber
&\cdot e^{\left(\pi\sqrt{-1}\theta_2-\frac{3}{2}\log(1-e^{2\pi\sqrt{-1}\theta_2})-\frac{1}{2}\log(1-e^{2\pi\sqrt{-1}(\theta_2+\theta_1)})-\frac{1}{2}\log(1-e^{2\pi\sqrt{-1}(\theta_2-\theta_1)})+\frac{1}{2}\log(1-e^{4\pi\sqrt{-1}\theta_2})\right)}.
\end{align}
We introduce
\begin{align}
&g(\theta_1,\theta_2)=2\psi(\theta_1,\theta_2)\sin\left(\frac{\theta_1\pi }{q}-J(s^\pm)\pi\right)\\\nonumber
&\cdot e^{\left(\pi\sqrt{-1}\theta_2-\frac{3}{2}\log(1-e^{2\pi\sqrt{-1}\theta_2})-\frac{1}{2}\log(1-e^{2\pi\sqrt{-1}(\theta_2+\theta_1)})-\frac{1}{2}\log(1-e^{2\pi\sqrt{-1}(\theta_2-\theta_1)})+\frac{1}{2}\log(1-e^{4\pi\sqrt{-1}\theta_2})\right)}.
    \end{align}

Let $(\pm \theta_1^0,\theta_2^0)$ be the unique critical point of $V^\pm$ in $D'_0$, by Proposition \ref{Prop-Hessian}, $\det(\text{Hess} V^\pm)(\pm \theta_1^0,\theta_2^0)\neq 0$.

By using the one-dimensional saddle point method as shown in appendix, we can show that 
\begin{align}
 \int_{S^\pm_{side}}\psi(\theta_1,\theta_2)\sin\left(\frac{\theta_1\pi }{q}-J(s^+)\pi\right) e^{(N+\frac{1}{2})\left(V_N(s^\pm,\theta_1,\theta_2;m_1^\pm,1)\right)}d\theta_1d\theta_2 =O\left(e^{(N+\frac{1}{2})(\zeta_{\mathbb{R}}(p,q)-\epsilon)}\right)  
\end{align}
for some $\epsilon>0$.

Now we compute the integral
\begin{align}
 |\int_{S^\pm _{top}}g(\theta_1,\theta_2) e^{(N+\frac{1}{2})V(s^\pm,\theta_1,\theta_2;m_1^\pm,1)}d\theta_1d\theta_2|  
\end{align}
by using Theorem \ref{Theorem-Fedo}. 

By previous discussions, $\text{Re} V^\pm$
 attains its maximal value at $(\pm \theta_1^0,\theta_2^0)$ which is an interior point of $S^\pm_{top}$. 
Therefore,  by applying formula (\ref{formula-integral}), we obtain
\begin{align}
 &|\int_{S^\pm_{top}}g(\theta_1,\theta_2) e^{(N+\frac{1}{2})V (s^\pm,\theta_1,\theta_2,m_1^{\pm},1)}d\theta_1d\theta_2|\\\nonumber
 &=\left(\frac{\pi }{N+\frac{1}{2}}\right)\frac{g(\pm\theta_1^0,\theta_2^0)e^{(N+\frac{1}{2})V(s^\pm,\pm\theta_1^0,\theta_2^0,m_1^\pm,1)}}{\sqrt{\det \left( -\frac{1}{2}\text{Hess}(V^\pm)(\theta_1^0,\theta_2^0)\right)}}\left(1+O\left(\frac{1}{N+\frac{1}{2}}\right)\right), 
\end{align}
with
\begin{align}
&g(\pm\theta_1^0,\theta_2^0)=2\sin\left(\frac{\theta_1^0\pi }{q}-J(s^\pm)\pi\right)\\\nonumber
&\cdot e^{\left(\pi\sqrt{-1}\theta_2^0-\frac{3}{2}\log(1-e^{2\pi\sqrt{-1}\theta_2^0})-\frac{1}{2}\log(1-e^{2\pi\sqrt{-1}(\theta_2^0+\theta_1^0)})-\frac{1}{2}\log(1-e^{2\pi\sqrt{-1}(\theta_2^0-\theta_1^0)})+\frac{1}{2}\log(1-e^{4\pi\sqrt{-1}\theta_2^0})\right)}\\\nonumber
&=2\sin\left(\frac{\pm\theta_1^0\pi }{q}-J(s^\pm)\pi\right)\frac{\sqrt{-1}}{\sqrt{1-(z_2^0)^2}},    
\end{align}
where in the second $``="$, we have used 
\begin{align}
&V_{\theta_2}(s^\pm,\pm\theta_1^0,\theta_2^0,m_1^{\pm},1)\\\nonumber
&=\pi\sqrt{-1}(2\theta_2^0-1)-\log(1-e^{2\pi\sqrt{-1}(\theta_1^0+\theta_2^0)})-\log(1-e^{2\pi\sqrt{-1}(\theta_2^0-\theta_1^0)})\\\nonumber
&-3\log(1-e^{2\pi\sqrt{-1}\theta_2^0})+2\log(1-e^{4\pi\sqrt{-1}\theta_2^0})=0.    
\end{align}
Moreover, the determinant of the Hessian matrix at $(\pm \theta_1^0,\theta_2^0)$ is given by 
\begin{align}
    &\det \left(-\frac{1}{2}\text{Hess}(\hat{V})(\pm \theta_1^0,\theta_2^0)\right)=(-\frac{1}{2})^2(2\pi\sqrt{-1})^2H(p,q;z_1^0,z_2^0),
\end{align}
where 
\begin{align}
&H(p,q;z_1^0,z_2^0)\\\nonumber
&=\left(\frac{p}{2q}+1\right)\left(1+\frac{3z_2^0}{1-z_2^0}-\frac{4(z_2^0)^2}{1-(z_2^0)^2}\right)+\left(\frac{p}{2q}+2+\frac{3z_2^0}{1-z_2^0}-\frac{4(z_2^0)^2}{1-(z_2^0)^2}\right)\\\nonumber
&\cdot\left(\frac{z_1^0z_2^0}{1-z_1^0z_2^0}+\frac{z_2^0}{z_1^0-z_2^0}\right)+\frac{4z_1^0(z_2^0)^2}{(1-z_2^0z_1^0)(z_1^0-z_2^0)}    
\end{align}

Therefore, we have
\begin{align}
    &\int_{S^\pm_{top}}g(\theta_1,\theta_2) e^{(N+\frac{1}{2})V(s^\pm,\theta_1,\theta_2,m_1^{\pm},1)}d\theta_1d\theta_2\\\nonumber
    &=\frac{2\sin\left(\frac{\pm\theta_1^0\pi}{q}-J(s^\pm)\pi\right)e^{-(N+\frac{1}{2})\left(\frac{K(s^\pm)}{2}+\frac{p'}{2q}\right)\pi\sqrt{-1}}}{(N+\frac{1}{2})\sqrt{1-(z_2^0)^2}\sqrt{H(p,q;z_1^0,z_2^0)}}e^{(N+\frac{1}{2})\zeta(p,q)}\left(1+O\left(\frac{1}{N+\frac{1}{2}}\right)\right). 
\end{align}
\end{proof}
Therefore,
\begin{align}
 &\hat{h}_{N}(s^{\pm}, m_1^{\pm},1)\\\nonumber
 &=(-1)^{P(s^{\pm})+m_1^{\pm}+1+\frac{3N}{2}+\frac{1}{4}+b_l}\kappa_N\frac{\sqrt{2N+1}}{\sin\frac{\pi}{2N+1}}(N+\frac{1}{2})\\\nonumber
 &\cdot\frac{2\sin\left(\frac{\pm\theta_1^0\pi}{q}-J(s^\pm)\pi\right)e^{-(N+\frac{1}{2})\left(\frac{K(s^\pm)}{2}+\frac{p'}{2q}\right)\pi\sqrt{-1}}}{\sqrt{1-(z_2^0)^2}\sqrt{H(p,q;z_1^0,z_2^0)}}e^{(N+\frac{1}{2})\zeta(p,q)}\left(1+O\left(\frac{1}{N+\frac{1}{2}}\right)\right)\\\nonumber
        &=(-1)^{P(s^{\pm})+m_1^{\pm}+\frac{3l}{4}+\frac{3N}{2}+\sum_{j=1}^{l-1}b_j}e^{(N+\frac{1}{2})\pi\sqrt{-1}\left(\frac{3}{2}b_l-(\frac{K(s^\pm)}{2}+\frac{p'}{2q})\right)}e^{\sigma\left(\frac{3}{r}+\frac{r+1}{4}\right)\pi \sqrt{-1}}\\\nonumber
        &\cdot\frac{\sqrt{2N+1}}{2\sin\frac{\pi}{2N+1}\sqrt{q}}(\pm \omega(p,q))e^{(N+\frac{1}{2})\zeta(p,q)}e^{-\frac{\pi \sqrt{-1}}{r}(\sum_{i=1}^lb_i+\sum_{i=1}^{l-1}\frac{1}{C_iC_{i+1}})}\left(1+O\left(\frac{1}{N+\frac{1}{2}}\right)\right)\\\nonumber
        &=(-1)^{P(s^{\pm})+m_1^{\pm}+\frac{3l}{4}+\frac{3N}{2}+\sum_{j=1}^{l-1}b_j}e^{(N+\frac{1}{2})\pi\sqrt{-1}\left(\frac{3}{2}b_l-(\frac{K(s^\pm)}{2}+\frac{p'}{2q})\right)}e^{\sigma\left(\frac{3}{r}+\frac{r+1}{4}\right)\pi \sqrt{-1}}\\\nonumber
        &\cdot\frac{\sqrt{2N+1}}{2\sin\frac{\pi}{2N+1}\sqrt{q}}(\pm \omega(p,q))e^{(N+\frac{1}{2})\zeta(p,q)}\left(1+O\left(\frac{1}{N+\frac{1}{2}}\right)\right),
\end{align}
where 
\begin{align} \label{formula-omegapq}
\omega(p,q)=\frac{\sin\left(\frac{\theta_1^0\pi}{q}-J(s^+)\pi\right)}{\sqrt{1-(z_2^0)^2}\sqrt{H(p,q;z_1^0,z_2^0)}}.    
\end{align}

\subsubsection{The cases $(s,m_1)\neq (s^\pm,m_1^\pm)$}
\begin{proposition} \label{prop-sneq+-}
For $(s,m_1)\neq (s^\pm,m_1^\pm)$, there exists $\epsilon>0$, such that
\begin{align}
|\int_{D_{0}}\psi(\theta_1,\theta_2)\sin\left(\frac{\theta_1\pi }{q}-J(s)\pi\right) e^{(N+\frac{1}{2})\left(V_N(s,\theta_1,\theta_2;m_1,1)\right)}d\theta_1d\theta_2|<Ce^{(N+\frac{1}{2})(\zeta_{\mathbb{R}}(p,q)-\epsilon)}.   
\end{align}
\end{proposition}
The proof of this Proposition follows directly from the proof of Proposition 6.6 in \cite{WongYang20-1}.  We omit here.

\subsection{Final proof} \label{subsection-final}
Now we can finish the proof of Theorem \ref{theorem-main} as follows.
\begin{proof}
Combining Propositions \ref{prop-fouriercoeff}, \ref{prop-neglected}, \ref{prop-sneq+-} and \ref{prop-s+-} together, we obtain
    \begin{align}
        &J_{N}(W(p,q);t)\\\nonumber
        &=\hat{h}_N(s^+,m^+,1)+\hat{h}_N(s^-,m^-,1)+O(e^{(N+\frac{1}{2})(\zeta_\mathbb{R}(p,q))-\epsilon}),
        \\\nonumber
        &=(-1)^{P(s^{+})+m_1^{+}+\frac{3l}{4}+\frac{3N}{2}+\sum_{j=1}^{l-1}b_j}e^{(N+\frac{1}{2})\pi\sqrt{-1}\left(\frac{3}{2}b_l-(\frac{K(s^+)}{2}+\frac{p'}{2q})\right)}e^{\sigma\left(\frac{3}{r}+\frac{r+1}{4}\right)\pi \sqrt{-1}}\\\nonumber
        &\cdot\frac{\sqrt{2N+1}}{\sin\frac{\pi}{2N+1}\sqrt{q}}\omega(p,q)e^{(N+\frac{1}{2})\zeta(p,q)}\left(1+O\left(\frac{1}{N+\frac{1}{2}}\right)\right),
    \end{align}
    with $\omega(p,q)$  given by formula (\ref{formula-omegapq}), where in the last $``="$, we have used the symmetry $\hat{h}_{N}(s^+,m^+,1)=\hat{h}_{N}(s^-,m^-,1)$ as shown in \cite{WongYang20-1}. For brevity, we set 
    \begin{align}
      C_N(p,q)=(-1)^{P(s^{+})+m_1^{+}+\frac{3l}{4}+\frac{3N}{2}+\sum_{j=1}^{l-1}b_j}e^{(N+\frac{1}{2})\pi\sqrt{-1}\left(\frac{3}{2}b_l-(\frac{K(s^+)}{2}+\frac{p'}{2q})\right)}e^{\sigma\left(\frac{3}{r}+\frac{r+1}{4}\right)\pi \sqrt{-1}}
    \end{align}
    which is a constant of norm $1$ independent of the geometric structure. 
\end{proof}

\section{Asymptotic expansion of the Turaev-Viro invariant} \label{Section-TV}

By formula (\ref{formula-TV=RT}) , the Turaev-Viro invariant for $W(p,q)$ (we omit the constant $C$ here) is given by 
\begin{align}
 TV_r(W(p,q),t)=\mu_r^2\sum_{a=0}^{N-1}|\bar{J}_{N-a}(W(p,q);t)|^2.   
\end{align}

In order to compute the asymptotic expansion for $TV_{r}(W(p,q),t)$, we divide the whole procedure into two steps. The first step is to use the two-dimensional saddle point method to derive an asymptotic expansion formula for each term 
\begin{align}
\bar{J}_{N-a}(W(p,q);t)    
\end{align}
where $a\in \{0,..,N-1\}$. We can check the conditions for using the saddle point method hold when the parameter $x=\frac{a+\frac{1}{2}}{N+\frac{1}{2}}$ is small, and we present the asymptotic expansion of $J_{N-a}$ in the form of Proposition \ref{proposition-JN-a}.   

\subsection{Basic computations}
For $m\in \{1,..,N\}$, we have
\begin{align}
    &J_{m}(W(p,q);t)\\\nonumber
    &=\frac{(-1)^{m-1}\{1\}}{\{m\}}\left(\frac{\sin \frac{2\pi }{r}}{\sqrt{r}}\right)^{l+1}e^{\sigma\left(\frac{3}{r}+\frac{r+1}{4}\right)\pi i}\sum_{n_1,...n_l=0}^{r-2}(-1)^{n_l+\sum_{j=1}^lb_jn_j}t^{\sum_{j=1}^l\frac{b_jn_j(n_j+2)}{4}}\\\nonumber
    &\cdot[n_1+1]\prod_{j=1}^{l-1}[(n_j+1)(n_{j+1}+1)]\langle e_{m-1},e_{n_l}\rangle_{W}.
\end{align}
Let $m=N-a$, with the similar method as presented in Section \ref{Section-Potentialfunction},  we obtain
\begin{proposition}  \label{proposition-JN-a0}
\begin{align}
   &J_{N-a}(W(p,q);t)\\\nonumber
   &=-\kappa_N\frac{(-1)^{\frac{3N}{2}+\frac{1}{4}+b_l}\sqrt{2N+1}}{\sin\frac{2\pi(a+\frac{1}{2})}{2N+1}}\sum_{s=0}^{|q|-1}(-1)^{P(s)}\sum_{n'=-N+\frac{1}{2}}^{N-\frac{1}{2}}\sum_{i'=\max\{-n',n',a+\frac{1}{2}\}}^{N-\frac{1}{2}}\\\nonumber
   &\cdot\sin\left(\frac{\pi n'}{(N+\frac{1}{2})q}-J(s)\pi\right)
   e^{(N+\frac{1}{2})V_{N}\left(p,q;s,\frac{a+\frac{1}{2}}{N+\frac{1}{2}},\frac{n'}{N+\frac{1}{2}},\frac{i'}{N+\frac{1}{2}}\right)},   
   \end{align}
where
\begin{align}
    &V_N(p,q;s,x,
    \theta_1,\theta_2)\\\nonumber
    &=\pi \sqrt{-1}\left(\frac{1}{2}-2\theta_2-2x\theta_2-2\theta_1\theta_2+2\theta_2^2+\frac{p}{2q}\theta_1^2-\frac{I(s)}{q}\theta_1-\frac{1}{2}K(s)\right.\\\nonumber
    &\left.-\frac{x+\theta_1+\frac{1}{2}}{N+\frac{1}{2}}-\frac{1}{(N+\frac{1}{2})^2}\right)+\frac{1}{N+\frac{1}{2}}\left(\log 2-\varphi_N\left(1-x-\theta_2-\frac{\frac{1}{2}}{N+\frac{1}{2}}\right)\right.\\\nonumber
    &\left.+\varphi_N\left(-x+\theta_2+\frac{\frac{1}{2}}{N+\frac{1}{2}}\right)-\varphi_{N}\left(1-\theta_1-\theta_2-\frac{\frac{1}{2}}{N+\frac{1}{2}}\right)\right.\\\nonumber
    &\left.+\varphi_N\left(\theta_2-\theta_1+\frac{\frac{1}{2}}{N+\frac{1}{2}}\right)-\varphi_{N}\left(1-\theta_2-\frac{\frac{1}{2}}{N+\frac{1}{2}}\right)+\varphi_{N}\left(1-2\theta_2-\frac{\frac{1}{2}}{N+\frac{1}{2}}\right)\right).
\end{align}
Note that we have introduced the variables
\begin{align}
    x=\frac{a+\frac{1}{2}}{N+\frac{1}{2}}, \ \theta_1=\frac{n'}{N+\frac{1}{2}}, \
    \theta_2=\frac{i'}{N+\frac{1}{2}},
\end{align}
with
\begin{align}
    \frac{\frac{1}{2}}{N+\frac{1}{2}}\leq x&=\frac{a+\frac{1}{2}}{N+\frac{1}{2}}\leq \frac{N-\frac{1}{2}}{N+\frac{1}{2}}, \\\nonumber
    -\theta_2\leq \theta_1&=\frac{n'}{N+\frac{1}{2}}\leq \theta_2,
    \\\nonumber
        x \leq \frac{i'}{N+\frac{1}{2}}&=\theta_2\leq \frac{N-\frac{1}{2}}{N+\frac{1}{2}}.     
\end{align}
\end{proposition}
Note that the expression of $V_{N}$ is actually  different when $(x,\theta_1,\theta_2)$ lies in different regions, see Proposition \ref{proposition-potential}, which is the special case of the above Proposition with $a=\frac{1}{2}$. 

We also define
\begin{align}
    &V(p,q;s,x,\theta_1,\theta_2)\\\nonumber
    &=\lim_{N\rightarrow \infty}V_N(p,q;s,x,\theta_1,\theta_2)\\\nonumber
    &=\pi \sqrt{-1}\left(\frac{1}{2}-2\theta_2-2x\theta_2-2\theta_1\theta_2+2\theta_2^2+\frac{p}{2q}\theta_1^2-\frac{I(s)}{q}\theta_1-\frac{1}{2}K(s)\right)\\\nonumber
    &+\frac{1}{2\pi\sqrt{-1}}\left(-\text{Li}_2(e^{2\pi\sqrt{-1}(-x-\theta_2)})+\text{Li}_2(e^{2\pi\sqrt{-1}(\theta_2-x)})-\text{Li}_2(e^{2\pi\sqrt{-1}(-\theta_1-\theta_2)})\right.\\\nonumber
    &\left.+\text{Li}_2(e^{2\pi\sqrt{-1}(\theta_2-\theta_1)})-\text{Li}_2(e^{-2\pi\sqrt{-1}\theta_2})+\text{Li}_2(e^{-4\pi\sqrt{-1}\theta_2}) \right). 
\end{align}

In particular, when $x=0$, $V(p,q;s,0,\theta_1,\theta_2)$ is just the function $V(p,q;s,\theta_1,\theta_2,0,0)$ shown in formula $(\ref{formula-Vpqs})$. 

\subsection{Hessian matrix}
The Hessian matrix plays important roles in proving the asymptotic expansion formula by using saddle point method.  We consider the Hessian matrix 
\begin{align}
f^{\pm}(p,q;x,\theta_1,X_1,\theta_2,X_2)=\text{Re}V^{\pm}(p,q;x,\theta_1+\sqrt{-1}X_1,\theta_2+\sqrt{-1}X_2).    
\end{align} 

Since $I(s^\pm)=\pm 1-q+2m^{\pm}q$, we have
\begin{align}
    &V(p,q;s^\pm,m^{\pm},1,x,
    \theta_1,\theta_2)\\\nonumber
    &=\pi \sqrt{-1}\left(\frac{1}{2}-2x\theta_2-2\theta_1\theta_2+2\theta_2^2+\frac{p}{2q}\theta_1^2+(1\pm \frac{1}{q})\theta_1-\frac{1}{2}K(s^\pm)\right)\\\nonumber
    &+\frac{1}{2\pi\sqrt{-1}}\left(-\text{Li}_2(e^{-2\pi\sqrt{-1}(x+\theta_2)})+\text{Li}_2(e^{2\pi\sqrt{-1}(-x+\theta_2)})\right.\\\nonumber
    &\left.-\text{Li}_2(e^{2\pi\sqrt{-1}(-\theta_1-\theta_2)})+\text{Li}_2(e^{2\pi\sqrt{-1}(\theta_2-\theta_1)})-\text{Li}_2(e^{-2\pi\sqrt{-1}\theta_2})+\text{Li}_2(e^{-4\pi\sqrt{-1}\theta_2})\right).
\end{align}
We define
\begin{align}
&V^{\pm}(p,q;x,\theta_1,\theta_2)\\\nonumber
&=V(p,q;s^\pm,m^{\pm},1,x,
    \theta_1,\theta_2)+\left(\frac{K(s^{\pm})}{2}+\frac{p'}{2q}\right)\pi\sqrt{-1}\\\nonumber
    &=\pi \sqrt{-1}\left(\frac{1}{2}+\frac{p'}{2q}^2-2x\theta_2-2\theta_1\theta_2+2\theta_2^2+\frac{p}{2q}\theta_1^2+(1\pm \frac{1}{q})\theta_1\right)\\\nonumber
    &+\frac{1}{2\pi\sqrt{-1}}\left(-\text{Li}_2(e^{-2\pi\sqrt{-1}(x+\theta_2)})+\text{Li}_2(e^{2\pi\sqrt{-1}(-x+\theta_2)})\right.\\\nonumber
    &\left.-\text{Li}_2(e^{2\pi\sqrt{-1}(-\theta_1-\theta_2)})+\text{Li}_2(e^{2\pi\sqrt{-1}(\theta_2-\theta_1)})-\text{Li}_2(e^{-2\pi\sqrt{-1}\theta_2})+\text{Li}_2(e^{-4\pi\sqrt{-1}\theta
    _2})\right),
\end{align}
then we have
\begin{align}
V^{\pm}_{\theta_1}=\pi\sqrt{-1}\left(\mp \frac{1}{q}+1-2\theta_2+\frac{p}{q}\theta
_1\right)-\log(1-e^{2\pi\sqrt{-1}(-\theta_1-\theta_2)})+\log(1-e^{2\pi\sqrt{-1}(\theta_2-\theta_1)}),    
\end{align}
and 
\begin{align}
V^{\pm}_{\theta_2}&=\pi\sqrt{-1}(-2x-2\theta_1+4\theta_2)-\log(1-e^{2\pi\sqrt{-1}(-x-\theta_2)})-\log(1-e^{2\pi\sqrt{-1}(\theta_2-x)})\\\nonumber 
&-\log(1-e^{2\pi\sqrt{-1}(-\theta_1-\theta_2)})-\log(1-e^{2\pi\sqrt{-1}(\theta_2-\theta_1)})\\\nonumber
&-\log(1-e^{-2\pi\sqrt{-1}\theta_2})+2\log(1-e^{-4\pi\sqrt{-1}\theta_2}). 
\end{align}
Set $z_1=e^{2\pi\sqrt{-1}\theta_1}$ and $z_2=e^{2\pi\sqrt{-1}\theta_2}$, 
then we have
\begin{align}
V^{\pm}_{\theta_1\theta_1}&=\pi\sqrt{-1}\frac{p}{q}-\frac{2\pi\sqrt{-1}e^{2\pi\sqrt{-1}(-\theta
_1-\theta_2)}}{1-e^{2\pi\sqrt{-1}(-\theta_1-\theta_2)}}+\frac{2\pi\sqrt{-1}e^{2\pi\sqrt{-1}(-\theta
_1+\theta_2)}}{1-e^{2\pi\sqrt{-1}(-\theta_1+\theta_2)}}\\\nonumber
&=2\pi\sqrt{-1}\left(\frac{p}{2q}+\frac{1}{1-z_1z_2}-\frac{1}{1-z_1/z_2}\right),
\end{align}

\begin{align}
V^{\pm}_{\theta_1\theta_2}&=-2\pi\sqrt{-1}-\frac{2\pi\sqrt{-1}e^{2\pi\sqrt{-1}(-\theta
_1-\theta_2)}}{1-e^{2\pi\sqrt{-1}(-\theta_1-\theta_2)}}-\frac{2\pi\sqrt{-1}e^{2\pi\sqrt{-1}(-\theta
_1+\theta_2)}}{1-e^{2\pi\sqrt{-1}(-\theta_1+\theta_2)}}\\\nonumber
&=2\pi\sqrt{-1}\left(-1+\frac{1}{1-z_1z_2}+\frac{1}{1-z_1/z_2}\right),
\end{align}
and
\begin{align}
V^{\pm}_{\theta_2\theta_2}&=4\pi\sqrt{-1}-\frac{2\pi\sqrt{-1}e^{2\pi\sqrt{-1}(-x-\theta_2)}}{1-e^{2\pi\sqrt{-1}(-x-\theta_2)}}+\frac{2\pi\sqrt{-1}e^{2\pi\sqrt{-1}(-x+\theta_2)}}{1-e^{2\pi\sqrt{-1}(-x+\theta_2)}}\\\nonumber
&-\frac{2\pi\sqrt{-1}e^{2\pi\sqrt{-1}(-\theta_1-\theta_2)}}{1-e^{2\pi\sqrt{-1}(-\theta_1-\theta_2)}}+\frac{2\pi\sqrt{-1}e^{2\pi\sqrt{-1}(-\theta_1+\theta_2)}}{1-e^{2\pi\sqrt{-1}(-\theta_1+\theta_2)}}\\\nonumber
&-\frac{2\pi\sqrt{-1}e^{-2\pi\sqrt{-1}\theta_2}}{1-e^{-2\pi\sqrt{-1}\theta_2}}+\frac{8\pi\sqrt{-1}e^{-4\pi\sqrt{-1}\theta_2}}{1-e^{-4\pi\sqrt{-1}\theta_2}}\\\nonumber
&=2\pi\sqrt{-1}\left(2+\frac{1}{1-e^{2\pi\sqrt{-1}x}z_2}-\frac{1}{1-e^{2\pi\sqrt{-1}x}/z_2}+\frac{1}{1-z_1z_2}\right.\\\nonumber
&\left.-\frac{1}{1-z_1/z_2}+\frac{1}{1-z_2}-\frac{4}{1-z_2^2}\right).
\end{align}

We introduce 
\begin{align}
f^{\pm}(p,q;x,\theta_1,X_1,\theta_2,X_2)=\text{Re}V^{\pm}(p,q;x,\theta_1+\sqrt{-1}X_1,\theta_2+\sqrt{-1}X_2)    
\end{align}
then the Hessian matrix of $f$ is given by 
\begin{align}
\text{Hess}(f)=2\pi \begin{pmatrix}
    b+c & b-c \\
    b-c & b+c+d
\end{pmatrix}
\end{align}
where 
\begin{align}
b=\text{Im}\frac{1}{1-z_1z_2}=\frac{\sin(2\pi(\theta_1+\theta_2))}{e^{2\pi(X_1+X_2)}+e^{-2\pi(X_1+X_2)}-2\cos(2\pi(\theta_1+\theta_2))}    
\end{align}
\begin{align}
c=-\text{Im}\frac{1}{1-z_1/z_2}=\frac{\sin(2\pi(\theta_2-\theta_1))}{e^{2\pi(X_2-X_1)}+e^{-2\pi(X_2-X_1)}-2\cos(2\pi(\theta_2-\theta_1))}    
\end{align}
\begin{align}
d&=\text{Im}\left(\frac{1}{1-e^{2\pi\sqrt{-1}x}z_2}-\frac{1}{1-e^{2\pi\sqrt{-1}x}/z_2}+\frac{1}{1-z_2}-\frac{4}{1-z_2^2}\right)\\\nonumber
&=\text{Im}\left(\frac{2-2\cos(2\pi x)z_2}{1-2\cos(2\pi x)z_2+z_2^2}-\frac{1}{1-z_2}+\frac{2z_2-2}{1-z_2^2}\right).
\end{align}

\begin{lemma}
Given an $\epsilon>0$, if $\epsilon<\theta_2<\frac{1}{2}$ and $0<x<\min\{\frac{1}{\pi}\arccos{\sqrt{\frac{3+\sqrt{9+16\cos^2(2\pi\epsilon)}}{8}}},\frac{1}{6}\}$, then we have $d>0$.
\end{lemma}
\begin{proof}
First, note that when $0<\theta_2<\frac{1}{2}$, we have 
\begin{align}
    \text{Im}\left(\frac{2z_2-2}{1-z_2^2}\right)=2\sin(2\pi \theta_2)\frac{e^{2\pi X_2}+e^{-2\pi X_2 }-2\cos(2\pi \theta_2)}{e^{4\pi X_2}+e^{-4\pi X_2 }-2\cos(4\pi \theta_2)}>0.
\end{align}

On the other hand, if $\epsilon<\theta_2<\frac{1}{2}$ and $0<x<\frac{1}{\pi}\arccos{\sqrt{\frac{3+\sqrt{9+16\cos^2(2\pi\epsilon)}}{8}}}$, then we have 
\begin{align}
    \cos^2(\pi x)>\frac{3+\sqrt{9+16\cos^2(2\pi \theta_2)}}{8}
\end{align}
which implies $4\cos^4(\pi x)-3\cos^2(\pi x)-\cos^2(2\pi\theta_2)>0$, i.e.
\begin{align}
    \cos^2(2\pi \theta_2)-(2\cos(2\pi x)-1)\cos^2(\pi x)<0
\end{align}
Hence 
\begin{align}
    &\cos^2(2\pi \theta_2)-(2\cos(2\pi x)-1)(\cos^2(2\pi \theta_2)+\sin^2(2\pi x))\\\nonumber
    &=4\sin^2(\pi x)(\cos^2(2\pi \theta_2)-(2\cos(2\pi x)-1)\cos^2(\pi x))<0. 
\end{align}

If $x<\frac{1}{6}$, then $2\cos(2\pi x)-1>0$. Therefore, as a quadratic polynomial of $e^{2\pi X_2}+e^{-2\pi X_2}$,  
\begin{align}
    &(2\cos(2\pi x)-1)(e^{2\pi X_2}+e^{-2\pi X_2})^2-4\cos(2\pi \theta_2)(e^{2\pi X_2}+e^{-2\pi X_2})\\\nonumber&+4\cos^2(2\pi \theta_2)+4\sin^2(2\pi x)>0
\end{align}
it implies that 
\begin{align}
    &2\frac{-2\cos(2\pi \theta_2)+\cos(2\pi x)(e^{2\pi X_2}+e^{-2\pi X_2})}{e^{4\pi X_2}+2\cos(4\pi \theta_2)+e^{-4\pi X_2}-4\cos(2\pi x)\cos(2\pi\theta_2)(e^{2\pi X_2}+e^{-2\pi X_2})+4\cos^2(2\pi x)}\\\nonumber&>\frac{1}{e^{2\pi X_2}-2\cos(2\pi \theta_2)+e^{-2\pi X_2}}
\end{align}
hence
\begin{align} \label{formula-inequ}
   &2\frac{-\sin(4\pi \theta_2)+\cos(2\pi x)\sin(2\pi \theta_2)e^{2\pi X_2}+\sin(2\pi \theta_2)e^{-2\pi X_2}}{e^{4\pi X_2}+2\cos(4\pi \theta_2)+e^{-4\pi X_2}-4\cos(2\pi x)\cos(2\pi\theta_2)(e^{2\pi X_2}+e^{-2\pi X_2})+4\cos^2(2\pi  x)}\\\nonumber & >\frac{\sin(2\pi \theta_2)}{e^{2\pi X_2}-2\cos(2\pi \theta_2)+e^{-2\pi X_2}}.
\end{align}
By straightforward computations, we have
\begin{align}
&\text{Im}\left(\frac{2-2\cos(2\pi x)z_2}{1-2\cos(2\pi x)z_2+z_2^2}-\frac{1}{1-z_2}\right)\\\nonumber
&=\text{Im}\left(\frac{2-2\cos(2\pi x)e^{2\pi\sqrt{-1}(\theta_2+\sqrt{-1}X_2)}}{1-2\cos(2\pi x)e^{2\pi\sqrt{-1}(\theta_2)}}-\frac{1}{1-e^{2\pi\sqrt{-1}(\theta_2+\sqrt{-1}X_2)}}\right)\\\nonumber
&=2\frac{-\sin(4\pi \theta_2)+\cos(2\pi x)\sin(2\pi \theta_2)e^{2\pi X_2}+\sin(2\pi \theta_2)e^{-2\pi X_2}}{e^{4\pi X_2}+2\cos(4\pi \theta_2)+e^{-4\pi X_2}-4\cos(2\pi x)\cos(2\pi\theta_2)(e^{2\pi X_2}+e^{-2\pi X_2})+4\cos^2(2\pi x)}\\\nonumber & -\frac{\sin(2\pi \theta_2)}{e^{2\pi X_2}-2\cos(2\pi \theta_2)+e^{-2\pi X_2}}>0
\end{align}
by formula (\ref{formula-inequ}). 

Therefore, 
\begin{align}
    d=\text{Im}\left(\frac{2-2\cos(2\pi x)z_2}{1-2\cos(2\pi x)z_2+z_2^2}-\frac{1}{1-z_2}+\frac{2z_2-2}{1-z_2^2}\right)>0.
\end{align}
\end{proof}

\begin{proposition} \label{proposition-HessianPositive}
When $0\leq \theta_1+\theta_2, \theta_2-\theta_1\leq \frac{1}{2}$,  $\epsilon \leq \theta_2\leq \frac{1}{2}$, and 
\begin{align}
    0<x<\min\{\frac{1}{\pi}\arccos{\sqrt{\frac{3+\sqrt{9+16\cos^2(2\pi\epsilon)}}{8}}},\frac{1}{6}\},
\end{align}
 then we have
\begin{align}
  \text{Hess}(f)=2\pi \begin{pmatrix}
    b+c & b-c \\
    b-c & b+c+d
\end{pmatrix} 
\end{align}
is positive definite.
\end{proposition}

We set $x_0=0.01$ which is determined by the computations in the following. Suppose $a\in [0,N-1]$ is a sequence integers of $N$ satisfying $x=\frac{a+\frac{1}{2}}{N+\frac{1}{2}}$,   we will consider the asymptotic expansion of $J_{N-a}(W(p,q);t)$ for $x\in [0,x_0)$ and $x\in [x_0,1)$ respectively.

\subsection{Two-dimensional saddle point method}
Fix $x\in [0,x_0)$, we consider  the critical point equations 
\begin{align} \label{equation-crit1x}
V^{+}_{\theta_1}&=\pi\sqrt{-1}(-\frac{1}{q}+1-2\theta_2+\frac{p}{q}\theta_1)-\log(1-e^{2\pi\sqrt{-1}(-\theta_1-\theta_2)})\\\nonumber
&+\log(1-e^{2\pi\sqrt{-1}(\theta_2-\theta_1)})=0,    
\end{align}
and 
\begin{align} \label{equation-crit2x}
V^{+}_{\theta_2}&=\pi\sqrt{-1}(-2x-2\theta_1+4\theta_2)-\log(1-e^{2\pi\sqrt{-1}(-x-\theta_2)})-\log(1-e^{2\pi\sqrt{-1}(\theta_2-x)})\\\nonumber 
&-\log(1-e^{2\pi\sqrt{-1}(-\theta_1-\theta_2)})-\log(1-e^{2\pi\sqrt{-1}(\theta_2-\theta_1)})\\\nonumber
&-\log(1-e^{-2\pi\sqrt{-1}\theta_2})+2\log(1-e^{-4\pi\sqrt{-1}\theta_2})=0.
\end{align}

\begin{proposition} \label{proposition-existence-x}
For
$x\in [0,x_0)$ and  $p\geq 1000$ or $q\geq 1000$,  the above critical point equations (\ref{equation-crit1x}) and (\ref{equation-crit2x})  has a unique solution $(\theta_1(x),\theta_2(x))$  lies in the region $D(x)$, where $D(x)=\{(x,\theta_1,\theta_2)|(\theta_1,\theta_2)\in D\}$. 
\end{proposition}

\begin{proof}
Let $a=\frac{p}{q}$ and $\gamma=\frac{1}{q}$,
 for  
$x\in [0,x_0)$ and  $q\geq 1000$,  with the  tedious estimation of hyperbolic trig and trig functions and via Poincar\'e-Miranda Theorem, we obtain 
the critical point equations (\ref{equation-crit1x}) and (\ref{equation-crit2x}) has a solution $\theta_1(x)=\theta_{1R}(x)-\sqrt{-1}\theta_{1I}(x)$ lies in the following square region
\begin{align}
&[\frac{2\cos(\pi x)\sqrt{1+\sin^2(\pi x)}}{a^2+(4a+8)\cos^2(\pi x)}(1-\gamma)\gamma,\frac{2\cos(\pi x)\sqrt{1+\sin^2(\pi x)}}{a^2+(4a+8)\cos^2(\pi x)}(1+\gamma)\gamma]\\\nonumber
&\times [\frac{a+2\cos^2(\pi x)}{a^2+(4a+8)\cos^2(\pi x)}(1-\gamma)\gamma,\frac{a+2\cos^2(\pi x)}{a^2+(4a+8)\cos^2(\pi x)}(1+\gamma)\gamma].    
\end{align} 
Similarly, let $\tilde{\gamma}=\frac{1}{p}$,  for $x\in [0,x_0)$ and $p\geq 1000$,  the square region is given by  
\begin{align}
&[\frac{2a\cos(\pi x)\sqrt{1+\sin^2(\pi x)}}{a^2+(4a+8)\cos^2(\pi x)}(1-\tilde{\gamma})\tilde{\gamma},\frac{2a\cos(\pi x)\sqrt{1+\sin^2(\pi x)}}{a^2+(4a+8)\cos^2(\pi x)}(1+\tilde{\gamma})\tilde{\gamma}]\\\nonumber
&\times [\frac{a(a+2\cos^2(\pi x))}{a^2+(4a+8)\cos^2(\pi x)}(1-\tilde{\gamma})\tilde{\gamma},\frac{a(a+2\cos^2(\pi x))}{a^2+(4a+8)\cos^2(\pi x)}(1+\tilde{\gamma})\tilde{\gamma}].    
\end{align}

It implies that the critical point equations $(\ref{equation-crit1x})$ and $(\ref{equation-crit2x})$ has a solution $(\theta_1(x),\theta_2(x))$ lies in region $D(x)$. 
Then, similar to the proof in \cite{CZ23-1}, we can show that  $(\theta_1(x),\theta_2(x))$ is the unique solution in $D(x)$.

\end{proof}

In the following, we always assume that $p\geq 1000$ or $q\geq 1000$, so  
we set
\begin{align}
   \zeta(p,q;x)=V^+(p,q;\theta_1(x),\theta_2(x)),  
\end{align}
we have the following expansion formula which can be regarded as a generalization of Theorem \ref{theorem-main}. 
\begin{theorem} \label{proposition-JN-a}
 Suppose $a\in [0, N-1]$ is a sequence of integers related to $N$ with ratio  $x=\frac{a+\frac{1}{2}}{N+\frac{1}{2}}\in [0,\epsilon)$, then we have the following asymptotic expansion

\begin{align}
 J_{N-a}(W(p,q);t)\sim_{ N\rightarrow \infty }\frac{\sqrt{2N+1}}{\sin\left(\frac{(a+\frac{1}{2})\pi}{N+\frac{1}{2}}\right)\sqrt{q}}h\left(\frac{a+\frac{1}{2}}{N+\frac{1}{2}}\right)e^{(N+\frac{1}{2})\zeta\left(p,q;\frac{a+\frac{1}{2}}{N+\frac{1}{2}}\right)}
\end{align}
where $h(x)$ is a smooth function of $x$ with
\begin{align}
h(0)=C_{N}(p,q)\omega(p,q).   
\end{align} 
\end{theorem}
The proof of Theorem \ref{proposition-JN-a} is complicated. But the method is same to the proof of Theorem  \ref{theorem-main}. We just outline the key computations in the following.

\begin{lemma} \label{lemma-zetadx}
For $0<x < x_0$, we have
\begin{align}
\frac{d \text{Re} \zeta(p,q;x)}{dx}<0.     
\end{align}
\end{lemma}
\begin{proof}
\begin{align}
\frac{d\zeta(p,q;x)}{dx}&=\text{Re}\left(\frac{\partial V^+}{\partial x}+\frac{\partial V^+}{\partial \theta_1}\frac{\partial \theta_1}{\partial x}+\frac{\partial V^+}{\partial \theta_2}\frac{\partial \theta_2}{\partial x}\right)\\\nonumber
&=\text{Re}\left(\frac{\partial V^+}{\partial x}\right)\\\nonumber
&=\text{Re}(-2\pi\sqrt{-1}\theta_2(x)-\log(1-e^{-2\pi\sqrt{-1}(x+\theta_2(x))})+\log(1-e^{2\pi\sqrt{-1}(\theta_2(x)-x)}))\\\nonumber
&=\text{Re}\left(\log(e^{-\pi\sqrt{-1}(\theta_2(x)-x)}-e^{\pi\sqrt{-1}(\theta_2(x)-x)})-\log(e^{\pi\sqrt{-1}(\theta_2(x)+x)}-e^{-\pi\sqrt{-1}(\theta_2(x)+x)})\right)
\end{align}     

Furthermore, let $\theta_2(x)=\theta_{2R}(x)+\sqrt{-1}\Theta_2(x)$, we have
\begin{align}
&\text{Re}\left(\log(e^{-\pi\sqrt{-1}(\theta_2(x)-x)}-e^{\pi\sqrt{-1}(\theta_2(x)-x)}\right)\\\nonumber
&=\text{Re}\left(\log(e^{-\pi\sqrt{-1}(\theta_{2R}(x)-x)}e^{\pi\Theta_{2}(x)}-e^{\pi\sqrt{-1}(\theta_{2R}(x)-x)})e^{-\pi\Theta_2(x)}\right)\\\nonumber
&=\text{Re}\left(\log(\cos(\pi(\theta_{2R}(x)-x))(e^{\pi\Theta_2(x)}-e^{-\pi\Theta_2(x)})-\sqrt{-1}\sin(\pi(\theta_{2R}(x)-x))(e^{\pi\Theta_2(x)}+e^{-\pi\Theta_2(x)}))\right)\\\nonumber
&=\frac{1}{2}\log(e^{2\pi \Theta_2(x)}+e^{-2\pi\Theta_2(x)}-2\cos(2\pi(\theta_{2R}(x)-x))).
\end{align}
Similarly, we have
\begin{align}
&\text{Re}\left(\log(e^{-\pi\sqrt{-1}(\theta_2(x)+x)}-e^{\pi\sqrt{-1}(\theta_2(x)+x)})\right)\\\nonumber
&=\frac{1}{2}\log(e^{2\pi \Theta_2(x)}+e^{-2\pi\Theta_2(x)}-2\cos(2\pi(\theta_{2R}(x)+x))).
\end{align}

Moreover, for $0<x<x_0=0.01$ and $0<\theta_2(x)<\frac{1}{2}$, we have
\begin{align}
\cos(2\pi(\theta_{2R}(x)-x))>\cos(2\pi(\theta_{2R}(x)+x)).     
\end{align}

It implies that 
\begin{align}
 \frac{d\text{Re}\zeta(p,q;x)}{dx}=\frac{1}{2}\log\left(\frac{e^{2\pi \Theta_2(x)}+e^{-2\pi\Theta_2(x)}-2\cos(2\pi(\theta_{2R}(x)-x))}{e^{2\pi \Theta_2(x)}+e^{-2\pi\Theta_2(x)}-2\cos(2\pi(\theta_{2R}(x)+x))}\right)<0.   
\end{align}
\end{proof}

\begin{lemma}
For $0<x<x_0$, we have
\begin{align} \label{formula-dReVdx}
\frac{d}{dx} \text{Re}V^+\left(p,q;x,0,\frac{1}{2\pi }\arccos{\sqrt{\frac{(4\cos(2\pi x)+1)^2-9}{16}}}\right)>0.   
\end{align}
\end{lemma}
\begin{proof}
We let 
\begin{align}
f(x,\theta_1,\theta_2)=\text{Re}V^+\left(p,q;x,\theta_1,\theta_2\right).     
\end{align}
Then we have
\begin{align}
f_x(x,0,\theta_2)=\log |\frac{\sin(\pi(\theta_2-x))}{\sin(\pi(\theta_2+x))}|    
\end{align}
and 
\begin{align}
f_{\theta_2}(x,0,\theta_2)=\log|\frac{\cos^2(\pi \theta_2)}{2\sin(\pi(\theta_2+x))\sin(\pi(\theta_2-x))\sin(\pi \theta_2)}|.    
\end{align}

Hence
\begin{align}
&\frac{df(x,0,\frac{1}{2\pi }\arccos{\sqrt{\frac{(4\cos(2\pi x)+1)^2-9}{16}}})}{dx}\\\nonumber
&=\log |\frac{\sin(\pi(\theta_2-x))}{\sin(\pi(\theta_2+x))}|+\frac{d\theta_2(x)}{dx}\log|\frac{\cos^2(\pi \theta_2(x))}{2\sin(\pi(\theta_2(x)+x))\sin(\pi(\theta_2(x)-x))\sin(\pi \theta_2(x))}|,
\end{align}
where 
\begin{align}
\theta_2(x)=\frac{1}{2\pi }\arccos{\sqrt{\frac{(4\cos(2\pi x)+1)^2-9}{16}}}.    
\end{align}

It is easy to obtain that when $0<x<\frac{1}{6}$, 
\begin{align}
\frac{d\theta_2(x)}{dx}=\frac{4(4\cos(2\pi x)+1)\sin(2\pi x)}{\sqrt{64-((4\cos(2\pi x)+1)^2-17)^2}}>0.
\end{align}
Moreover, when $0<x<0.164412$, 
\begin{align}
  |\frac{\cos^2(\pi \theta_2(x))}{2\sin(\pi(\theta_2(x)+x))\sin(\pi(\theta_2(x)-x))\sin(\pi \theta_2(x))}|>1.   
\end{align}
Therefore, we obtain $\frac{df(x,0,\frac{1}{2\pi }\arccos{\sqrt{\frac{(4\cos(2\pi x)+1)^2-9}{16}}})}{dx}>0$. 
\end{proof}

Similar to the proof of  Proposition \ref{proposition-existence-x}, we can obtain the following inequalities
\begin{align} \label{inequality-theta1}
    \frac{d}{d\theta_1} \text{Re}V^+(p,q;x,\theta_1,\theta_2(x,\theta_1))<0
\end{align}
for $\theta_1>0$,
\begin{align}  \label{inequality-x}
   \frac{d}{d x} \text{Re}V^+(p,q;x,c_0,\theta_2(x,c_0))<0  
\end{align}
for $0\leq x\leq x_0$, and
 \begin{align}  \label{inequality-volest}
  2\pi \text{Re}\zeta(p,q;x)>Vol(W(p,q))-100\pi^2x^2.  
  \end{align}

By straightforward computations, we have
\begin{align}
  \text{Re} V^+\left(p,q;x,0,\frac{1}{2\pi }\arccos{\sqrt{\frac{(4\cos(2\pi x)+1)^2-9}{16}}}\right)|_{x=0}=\text{Re}V^+(p,q;0,0,0)=0, 
\end{align}
and by Proposition \ref{prop-crit=volume}, for $(p,q)\in S(p,q)$, 
\begin{align}
  2\pi \text{Re}\zeta(p,q;0)=Vol(W(p,q))>3.662.
\end{align}

By using the formula (\ref{inequality-volest}), we have
\begin{align} \label{inequality-zeta}
2\pi \text{Re}\zeta(p,q;x)&\geq 2\pi \text{Re} \zeta(p,q;x_0)\\\nonumber
& >Vol(W(p,q))-100\pi^2 x_0^2\\\nonumber
&>2\pi \text{Re}V^+\left(p,q;x_0,0,\frac{1}{2\pi }\arccos{\sqrt{\frac{(4\cos(2\pi x_0)+1)^2-9}{16}}}\right). 
\end{align}

\begin{proposition} \label{prop-0dim}
For $(\theta_1,\theta_2)\in D(x)$ with $x\in [0,x_0)$, $\theta_2\in [0,\frac{1}{2\pi }\arccos{\sqrt{\frac{(4\cos(2\pi x)+1)^2-9}{16}}} )$, we denote such region by $D(x)_0$.
Then we have
\begin{align}
 2\pi \text{Re}V^+(p,q;x,\theta_1,\theta_2)< 2\pi \text{Re}\zeta(p,q,x)-\epsilon
\end{align}
for some small $\epsilon>0$.
\end{proposition}
\begin{proof}
For $(\theta_1,\theta_2)\in D(x)_0$ , we have
\begin{align}
    &\max_{(\theta_1,\theta_2)\in D(x)}2\pi \text{Re}V^+(p,q;x,\theta_1,\theta_2)\\\nonumber
    &=  \max_{\theta_2\in [0,\frac{1}{2\pi}\arccos \sqrt{\frac{(4\cos(2\pi x)+1)^2-9}{16}})} 2\pi \text{Re}V^+(p,q;x,0,\theta_2)\\\nonumber
   &= 2\pi \text{Re}V^+\left(p,q;x,0,\frac{1}{2\pi }\arccos{\sqrt{\frac{(4\cos(2\pi x)+1)^2-9}{16}}}\right)\\\nonumber
   &\leq 2\pi \text{Re}V^+\left(p,q;x_0,0,\frac{1}{2\pi }\arccos{\sqrt{\frac{(4\cos(2\pi x_0)+1)^2-9}{16}}}\right)\\\nonumber
   & < 2\pi\text{Re}\zeta(p,q;x) 
\end{align}
By using inequality (\ref{inequality-zeta}) and formula (\ref{formula-dReVdx}), we obtain Proposition \ref{prop-0dim}.
\end{proof}

\begin{proposition}  \label{prop-onedim}
For $x\in [0,x_0)$ and $c_1\geq c_0= 0.122532$, we have
\begin{align}
    2\pi \zeta(p,q;x) \geq 2\pi\text{Re}V^+(p,q;x,c_1,\theta_2(x,c_1)).  
\end{align}
\end{proposition}
\begin{proof}
Since $\frac{d}{d\theta_1} \text{Re}V^+(p,q;x,\theta_1,\theta_2(x,\theta_1))<0$ for $\theta_1>0$, we have 
\begin{align}
 2\pi\text{Re}V^+(p,q;x,c_1,\theta_2(x,c_1))<2\pi    \text{Re}V^+(p,q;x,c_0,\theta_2(x,c_0)),
\end{align}
for $c_1\geq c_0$.

When $0\leq x<x_0$, by Lemma \ref{lemma-zetadx} and formula (\ref{inequality-volest}), we have
\begin{align}
    2\pi \text{Re}\zeta(p,q;x)&> 2\pi \text{Re}\zeta(p,q;x_0)\\\nonumber
    &>Vol(W(p,q))-100\pi^2x_0^2\\\nonumber
    &>3.5633>3.37448=2\pi \text{Re}V^+(p,q;0,c_0,\theta_2(0,c_0))\\\nonumber
    &\geq 2\pi \text{Re}V^+(p,q;x,c_0,\theta_2(x,c_0)). 
\end{align}
Therefore, we obtain 
\begin{align}
    2\pi \zeta(p,q;x) \geq 2\pi\text{Re}V^+(p,q;x,c_1,\theta_2(x,c_1)).  
\end{align}
\end{proof}

For $(\theta_1,\theta_2)\in \Delta(x)$, where $\Delta(x)$ is the triangle region with three vertices $(0,\frac{1}{2}), (c_0,\frac{1}{2})$ and $(c_0,\frac{1}{2}-c_2)$ on the $D(x)$, similarly as in the proof of \ref{lemma-v}, we have
\begin{align} \label{formula-triangle}
2\pi \text{Re}V^+(p,q;x,\theta_1,\theta_2)&<   2\pi \text{Re}V^+\left(p,q;x,c_0,\frac{1}{2}-c_0\right)\\\nonumber
&< 2\pi\text{Re}V^+\left(p,q;0,c_0,\frac{1}{2}-c_0\right)\\\nonumber
&=2\pi \text{Re}V^+(p,q;0,c_0,\theta_2(0,c_0)),
\end{align}
where the last ``=" is due to the definition of $c_0$.

Now, we outline the proof Theorem \ref{proposition-JN-a} with the same method that we used in the proof of Theorem \ref{theorem-main}.  
\begin{proof}
First, we apply the Poisson summation formula to Proposition \ref{proposition-JN-a0}. By studying each Fourier coefficients, only the following two remained integrals  
\begin{align}
\int_{D(x)}\sin\left(\frac{\pi \theta_1}{q}-J(s^+)\pi\right)e^{(N+\frac{1}{2})V_N(p,q;s^{\pm},x,\theta_1,\theta_2)}d\theta_1d\theta_2
\end{align}
need to be considered. Then we consider the functions $V^\pm(p,q;x,\theta_1,\theta_2)$ and check the conditions for the saddle point method. Since $V^+(p,q;x,-\theta_1,\theta_2)=V^-(p,q;x,\theta_1,\theta_2)$, we only study the function $V^+(p,q;x,\theta_1,\theta_2)$. 

For $(\theta_1,\theta_2)\in D(x)$ with $x\in [0,x_0)$, by Proposition \ref{prop-0dim}, the integral over the region $D(x)_0$ for $0\leq x<x_0$  can be neglected.  For $x\in [0,x_0)$, the function $\text{Re}V^+(p,q;x,\theta_1,\theta_2)$ and region $D(x)$ is symmetric with respect to $\theta_1=0$,  
 By Proposition \ref{prop-onedim}, using the one-dimensional saddle point method similar to Appendix \ref{subsection-onedim} and formula (\ref{formula-triangle}),  we only need to consider the integral over the region 
 \begin{align}
   D(x)_2=\{(\theta_1,\theta_2)\in D(x)|\theta_1\leq \theta_2\leq \frac{1}{2}-\theta_1, 0\geq \theta_1\leq c_0\}-D(x)_0.   
 \end{align}
  By Proposition \ref{proposition-HessianPositive}, the Hessian matrix is positive on the region $D(x)_2$.  Finally, we  apply two-dimensional saddle point method on $D(x)_2$ to obtain the asymptotic expansion formula presented in Theorem \ref{proposition-JN-a}.  
\end{proof}

\subsection{The case $x\geq x_0$}
For the case $x\in [x_0,1)$, we have  
\begin{proposition} \label{proposition-JN-a2}
 Suppose $a\in [0, N-1]$ is a sequence of integers related to $N$ with ratio  $x= \frac{a+\frac{1}{2}}{N+\frac{1}{2}}\in [x_0,1)$, then there exists $\epsilon>0$ such that

\begin{align}
 J_{N-a}(W(p,q);t)= O\left(e^{\frac{(N+\frac{1}{2})}{2\pi }(Vol(W(p,q))-\epsilon)}\right). 
\end{align}
\end{proposition}
\begin{proof}
When $x\geq \frac{1}{4}$, by straightforward computation, we have 
\begin{align}
2\pi \max_{x\geq \frac{1}{4},(\theta_1,\theta_2)\in D(x)}\text{Re}V^+(p,q;x,\theta_1,\theta_2)=3.25273.    
\end{align}
Hence for $x\geq \frac{1}{4}$ and 
 $(\theta_1,\theta_2)\in D(x)$, we have
 \begin{align}
     2\pi \text{Re} V^+(p,q;x,\theta_1,\theta_2)<3.25273<Vol(W(p,q))-\epsilon 
 \end{align}
for some $\epsilon>0$. 

When $x_0\leq x\leq \frac{1}{4}$, for $(\theta_1,\theta_2)\in D(x)$, by formula (\ref{inequality-theta1}), we have
\begin{align}
 2\pi \text{Re}V^+(p,q;x,\theta_1,\theta_2(x,\theta_1))& \leq  2\pi \text{Re}V^+(p,q;x_0,0,\theta_2(x_0,0))\\\nonumber
 &=3.37448<Vol(W(p,q))-\epsilon
\end{align}
for some $\epsilon>0$. Then, we prove Proposition \ref{proposition-JN-a2} by using the one-dimensional saddle point method as in Appendix \ref{subsection-onedim}.
\end{proof}

\subsection{Asymptotic expansion} \label{subsection-asym}
Now we study the large $N$ asymptotic expansion  for Turaev-Viro invariants $TV_{r}(W(p,q),t)$.  By Proposition ,  we get
\begin{align}
TV_{r}(W(p,q),t)&=\mu_r^2\sum_{a=0}^{N-1}\big| \bar{J}_{N-a}(W(p,q);t) \big|^2\\\nonumber
&\sim_{N\rightarrow \infty} \mu_r^2\sum_{a : \ x\in [0,x_0) }\big| \bar{J}_{N-a}(W(p,q);t) \big|^2.
\end{align}

By Proposition \ref{proposition-JN-a}, we have
\begin{align}
\bar{J}_{N-a}(W(p,q);t)&=(-1)^{N-a-1}\frac{\{N-a\}}{\{1\}}J_{N-a}(W(p,q);t)\\\nonumber
&\sim_{N\rightarrow \infty}(-1)^{N-a-1}\frac{\sqrt{2N+1}}{\sqrt{q}\sin\left(\frac{\pi}{N+\frac{1}{2}}\right)}h\left(\frac{a+\frac{1}{2}}{N+\frac{1}{2}}\right)e^{(N+\frac{1}{2})\zeta(p,q;\frac{a+\frac{1}{2}}{N+\frac{1}{2}})}.
\end{align}
Since 
\begin{align}
    \frac{\sqrt{2N+1}}{\sqrt{q}\sin\left(\frac{\pi}{N+\frac{1}{2}}\right)}\sim_{N\rightarrow \infty}\sqrt{\frac{2}{q}}\frac{(N+\frac{1}{2})^{\frac{3}{2}}}{\pi},
\end{align}
we obtain
\begin{align}
TV_{r}(W(p,q),t)\sim_{N\rightarrow \infty}\mu_r^2\sum_{a: \ x\in [0,\epsilon)}\frac{2(N+\frac{1}{2})^3}{q\pi^2 }\big|h\left(\frac{a+\frac{1}{2}}{N+\frac{1}{2}}\right)\big|^2e^{(N+\frac{1}{2})2\text{Re}\zeta(p,q;\frac{a+\frac{1}{2}}{N+\frac{1}{2}})}.    
\end{align}

We need the following
\begin{proposition}\label{propostion-WongAu} \cite{WongAu17}
Let $f(x)$ be an analytic function defined on a domain $D$ containing $[\alpha,\beta]$. Assume that 

(1) $x^0\in [\alpha,\beta]$ is the only critical point of $\text{Re}(f)$ along $[\alpha,\beta]$ on which $\text{Re}(f)$ attains its maximum; 

(2) $x^0$ is non-degenerate with $(\text{Re}f)''(x^0)<0$. 

Then for any positive $C^1$ function $h(x)$ on $[\alpha,\beta]$, we have the following asymptotic equivalence:
\begin{align}
&\int_{\alpha}^{\beta}h(x)|e^{(N+\frac{1}{2})f(x)}|dx\\\nonumber
&=\sum_{\alpha\leq \frac{a+\frac{1}{2}}{N+\frac{1}{2}}\leq \beta }\left(\frac{1}{N+\frac{1}{2}}\right)h\left(\frac{a+\frac{1}{2}}{N+\frac{1}{2}}\right)\Big|e^{(N+\frac{1}{2})f(\frac{a+\frac{1}{2}}{N+\frac{1}{2}})}\Big|\left(1+O\left(\frac{1}{(N+\frac{1}{2})^{\frac{3}{2}}}\right)\right).
\end{align}
\end{proposition}

According to Lemma \ref{lemma-imtheta}, we can choose small enough $\epsilon>0$, such that the conditions (1) and (2) holds for interval $[0,\epsilon)$, hence by Proposition \ref{propostion-WongAu}, we have

\begin{align}
    &\sum_{a: \ x\in [0,\epsilon)}\frac{2(N+\frac{1}{2})^3}{q\pi^2 }\big|h\left(\frac{a+\frac{1}{2}}{N+\frac{1}{2}}\right)\big|^2e^{(N+\frac{1}{2})2\text{Re}\zeta(p,q;\frac{a+\frac{1}{2}}{N+\frac{1}{2}})}\\\nonumber
    &\sim_{N\rightarrow \infty} \frac{2(N+\frac{1}{2})^4}{q\pi^2}\int_{0}^{\epsilon}|h(x)|^2e^{(N+\frac{1}{2})2\text{Re}\zeta(p,q;x)}dx     
\end{align}
Recall the definition
\begin{align}
|h(0)|^2=\Big|\frac{\sin^2\left(\frac{\pi \theta_1^0}{q}-J(s^+)\pi\right)}{(1-(z_2^0)^2)H(p,q;z_1^0,z_2^0)}\Big|,
\end{align}
\begin{align}
2\pi \text{Re}\zeta(p,q;0)=Vol(W(p,q)),    
\end{align}
and by Lemma \ref{lemma-imtheta},
\begin{align}
(\text{Re}\zeta)''(p,q;0)=-4\pi\text{Im}\left(\frac{1}{1-z_2^0}\right)<0.    
\end{align}

By Laplace method, we have

\begin{align}
&\int_{0}^{\epsilon}|h(x)|^2e^{(N+\frac{1}{2})2\text{Re}\zeta(p,q;x)}dx\\\nonumber
&\sim_{N\rightarrow \infty} \frac{1}{2}\frac{\sqrt{\pi}}{\sqrt{N+\frac{1}{2}}}\Big|\frac{\sin^2\left(\frac{\pi \theta_1^0}{q}-J(s^+)\pi\right)}{(1-(z_2^0)^2)H(p,q;z_1^0,z_2^0)}\Big|\frac{1}{2\sqrt{\pi}\sqrt{\text{Im}\left(\frac{1}{1-z_2^0}\right)}}e^{\frac{N+\frac{1}{2}}{\pi}Vol(W(p,q))}    
\end{align}
Note that there is a factor $\frac{1}{2}$ in the above formula since the maximum point lies on the boundary.

Therefore,
\begin{align}
    &TV_{r}(W(p,q),t)\\\nonumber
    &\sim_{N\rightarrow \infty}\mu_r^2\Big|\frac{\sin^2\left(\frac{\pi \theta_1^0}{q}-J(s^+)\pi\right)}{q(1-(z_2^0)^2)H(p,q,z_1^0,z_2^0)}\Big|\frac{(N+\frac{1}{2})^{\frac{7}{2}}}{2\pi^2\sqrt{\text{Im}\left(\frac{1}{1-z_2^0}\right)}}e^{\frac{(N+\frac{1}{2})}{\pi}\text{Vol}(W(p,q))}.
\end{align}
Since
\begin{align}
    \mu_r^2=\left(\frac{\sqrt{2}\sin\left(\frac{\pi}{N+\frac{1}{2}}\right)}{\sqrt{N+\frac{1}{2}}}\right)^2\sim_{N\rightarrow \infty}\frac{2\pi^2}{(N+\frac{1}{2})^3}, 
\end{align}
finally, we obtain 
\begin{align}
    &TV_{r}(W(p,q),t)\\\nonumber
    &\sim_{N\rightarrow \infty}\Big|\frac{\sin^2\left(\frac{\pi \theta_1^0}{q}-J(s^+)\pi\right)}{q(1-(z_2^0)^2)H(p,q,z_1^0,z_2^0)}\Big|\frac{(N+\frac{1}{2})^{\frac{1}{2}}}{\sqrt{\text{Im}\left(\frac{1}{1-z_2^0}\right)}}e^{\frac{(N+\frac{1}{2})}{\pi}\text{Vol}(W(p,q))}.
\end{align}

\begin{example} \label{example-4_1}
When $p=q=1$, in this case $J(s^+)=1$. 
The critical point equations (\ref{equation-crit1}) and (\ref{equation-crit2})
have a unique solution $(\theta_1^0,\theta_2^0)$ with $z_2^0=\frac{1}{2}+\frac{\sqrt{-3}}{2}$, 
$z_1^0=-\frac{1}{4}+\frac{3\sqrt{-3}}{4}-\frac{1}{4}\sqrt{-42-6\sqrt{-3}}.$
Then, we have
\begin{align}
 \sin^2(\pi \theta_1^0)=\frac{1-\cos(2\pi \theta_1^0)}{2}=\frac{1-\frac{1}{2}(z_1^0+(z_1^0)^{-1})}{2}=\frac{5}{8}-\frac{6}{8}\sqrt{-3}.   
\end{align}
\begin{align}
1-(z_2)^2&=\frac{3}{2}-\frac{1}{2}\sqrt{-3}, \\\nonumber
H(1,1,z_1^0,z_2^0)&=\frac{1}{2}+\sqrt{-3}.   
\end{align}

Hence
\begin{align}
    \Big|\frac{\sin^2\left(\frac{\pi \theta_1^0}{q}-J(s^+)\pi\right)}{(1-(z_2^0)^2)H(p,q,z_1^0,z_2^0)}\Big|=\frac{1}{2\sqrt{3}}.
\end{align}
and 
\begin{align}
    \frac{1}{\sqrt{\text{Im}\left(\frac{1}{1-z_2^2}\right)}}=\frac{\sqrt{2}}{(\sqrt{3})^{\frac{1}{2}}}
\end{align}

so we obtain 
\begin{align}
 TV_r(S^3\setminus 4_1;t)\sim_{N\rightarrow \infty}\frac{(N+\frac{1}{2})^{\frac{1}{2}}}{\sqrt{2}(\sqrt{3})^{\frac{3}{4}}}e^{\frac{(N+\frac{1}{2})}{\pi}\text{Vol}(S^3\setminus 4_1)}   
\end{align}
which is exactly the asymptotic expansion formula obtained by Wong and Au ( see  Theorem 8 in \cite{WongAu17}).

\end{example}

\subsection{Conditions for Laplace method}
We check the conditions for Proposition \ref{propostion-WongAu}, so we can use the Laplace method in previous Section \ref{subsection-asym}. 
Let 
\begin{align}
   \zeta(p,q;x)=V^+(p,q;\theta_1(x),\theta_2(x)),  
\end{align}
then we have
\begin{align}
 &\frac{d\text{Re}\zeta(p,q;x)}{dx}\\\nonumber
 &=\text{Re}\frac{d V^\pm}{dx}=\text{Re}\left(\frac{\partial V^\pm}{\partial x}+\frac{\partial V^\pm}{\partial \theta_1}\frac{\partial \theta_1}{\partial x}+\frac{\partial V^\pm}{\partial \theta_2}\frac{\partial \theta_2}{\partial x}\right)\\\nonumber
 &=\text{Re}\left(\frac{\partial V^\pm}{\partial x}\right)\\\nonumber
 &=\text{Re}\left(\pi\sqrt{-1}(-2\theta_2(x)-\log(1-e^{2\pi\sqrt{-1}(-x-\theta_2(x))}))+\log(1-e^{2\pi\sqrt{-1}(\theta_2(x)-x)})\right),
\end{align}
hence
\begin{align}
  &\left(\frac{d\text{Re}\zeta(p,q;x)}{dx}\right)|_{x=0}\\\nonumber
  &=\text{Re}\left(\pi\sqrt{-1}(-2\theta_2(0))-\log(1-e^{-2\pi\sqrt{-1}\theta_2(0)})+\log(1-e^{2\pi\sqrt{-1}\theta_2(0)})\right)  \\\nonumber
  &=\text{Re}\left(\pi\sqrt{-1}(-2\theta_2(0))+\log\frac{1-e^{2\pi\sqrt{-1}\theta_2(0)}}{1-e^{-2\pi\sqrt{-1}\theta_2(0)}}\right)\\\nonumber
  &=\text{Re}\left(\pi\sqrt{-1}(-2\theta_2(0))+\log\frac{e^{\pi\sqrt{-1}\theta_2(0)}(e^{-\pi\sqrt{-1}\theta_2(0)-e^{\pi\sqrt{-1}\theta_2(0)})}}{e^{-\pi\sqrt{-1}\theta_2(0)}(e^{\pi\sqrt{-1}\theta_2(0)}-e^{-\pi\sqrt{-1}\theta_2(0)})}\right)\\\nonumber
  &=\text{Re}(\pi\sqrt{-1}+\log(-e^{2\pi\sqrt{-1}\theta_2(0)}))\\\nonumber
  &=\text{Re}(\pi\sqrt{-1}(-2\theta_2(0))+2\pi\sqrt{-1}\theta_2(0)+\pi\sqrt{-1})=0.
\end{align}

Taking the derivative with respect to $x$ for formulas (\ref{equation-crit1x}) and (\ref{equation-crit2x}), we obtain
\begin{align} \label{equation-linear1}
    &\left(\frac{p}{2q}-\frac{1}{e^{2\pi\sqrt{-1}(\theta_1+\theta_2)}-1}+\frac{1}{e^{2\pi\sqrt{-1}(-\theta_2+\theta_1)}-1}\right)\frac{d\theta_1}{dx}\\\nonumber
    &-\left(1+\frac{1}{e^{2\pi\sqrt{-1}(\theta_2+\theta_1)}-1}+\frac{1}{e^{2\pi\sqrt{-1}(-\theta_2+\theta_1)}-1}\right)\frac{d\theta_2}{dx}=0
\end{align}
and 
\begin{align} \label{equation-linear2}
&\left(1+\frac{1}{e^{2\pi\sqrt{-1}(\theta_2+\theta_1)}-1}+\frac{1}{e^{2\pi\sqrt{-1}(-\theta_2+\theta_1)}-1}\right)\frac{d\theta_1}{dx}+\left(2-\frac{1}{e^{2\pi\sqrt{-1}(-\theta_2+x)}-1}\right.\\\nonumber&\left.+\frac{1}{e^{2\pi\sqrt{-1}(-\theta_2+x)}-1}-\frac{1}{e^{2\pi\sqrt{-1}(\theta_2+\theta_1)}-1}+\frac{1}{e^{2\pi\sqrt{-1}(-\theta_2+\theta_1)}-1}-\frac{1}{e^{2\pi\sqrt{-1}\theta_2}-1}\right.\\\nonumber
&\left.\frac{4}{e^{4\pi\sqrt{-1}\theta_2}-1} \right)\frac{d\theta_2}{dx}=1+\frac{1}{e^{2\pi\sqrt{-1}(\theta_2+x)}-1}+\frac{1}{e^{2\pi\sqrt{-1}(-\theta_2+x)}-1}.    
\end{align}
Solving the linear equations (\ref{equation-linear1}) and (\ref{equation-linear2}),  we  get
\begin{align}
 \frac{d\theta_2}{d x}|_{x=0}=0.    
\end{align}

Furthermore, 
\begin{align} \label{formula-dzetadx}
 &\frac{d^2\text{Re}\zeta(p,q;x)}{dx^2}|_{x=0}\\\nonumber
 &=\text{Re}\left(\left(-2\pi\sqrt{-1}-\frac{2\pi\sqrt{-1}e^{-2\pi\sqrt{-1}\theta_2(0)}}{1-e^{-2\pi\sqrt{-1}\theta_2(0)}}-\frac{2\pi\sqrt{-1}e^{2\pi\sqrt{-1}\theta_2(0)}}{1-e^{2\pi\sqrt{-1}\theta_2(0)}}\right)\frac{d\theta_2}{dx}(0)\right.\\\nonumber
 &\left.-\frac{2\pi\sqrt{-1}e^{-2\pi\sqrt{-1}\theta_2(0)}}{1-e^{-2\pi\sqrt{-1}\theta_2(0)}}+\frac{2\pi\sqrt{-1}e^{2\pi\sqrt{-1}\theta_2(0)}}{1-e^{2\pi\sqrt{-1}\theta_2(0)}} \right)\\\nonumber
 &=\text{Re}\left(2\pi\sqrt{-1}\left(-\frac{1/z_2}{1-1/z_2}+\frac{z_2}{1-z_2}\right)\right)\\\nonumber
 &=-2\pi \text{Im}\left(\frac{2}{1-z_2}-1\right)\\\nonumber
 &=-4\pi\text{Im}\left(\frac{1}{1-z_2}\right). 
\end{align}

\begin{lemma} \label{lemma-zeta''}
    \begin{align}
        \frac{d^2\text{Re} \zeta(p,q;x)}{dx^2}|_{x=0}<0. 
    \end{align}
\end{lemma}
\begin{proof}
We only need to show that $\text{Im}\left(\frac{1}{1-z_2}\right)>0.$
Recall the variable transformation formula (\ref{formula-transform}), 
\begin{align}
z_2^0=\frac{z^0}{(z^0)^2+z^0-1},   
\end{align}
we have
\begin{align}
\frac{1}{1-z_2^0}=1+\frac{1}{z^0-(z^0)^{-1}}=1+\frac{1}{2\sqrt{-1}\text{Im}(z^0)}   
\end{align}
Then, by Remark \ref{remark-conjugate}, we obtain
\begin{align}
\text{Im}\left(\frac{1}{1-z_2^0}\right)=\frac{1}{-2\text{Im}(z^0)}>0.     
\end{align}

\end{proof}

\section{Appendix} \label{Section-Appendices}

\subsection{Volume estimate}

\begin{theorem} \label{theorem-volestimate}
   Let $W(p,q)$ be the cusped hyperbolic 3-manifold obtained by doing $(p,q)$-surgery along one component of the Whitehead link $W$, then we have
    \begin{align}
   Vol(W(p,q))\geq \left(1-\frac{2\pi^2}{(p+2q)^2+4q^2}\right)^{\frac{3}{2}}Vol(S^3\setminus W)
\end{align}
\end{theorem}
\begin{proof}
    By Fulter-Kalfagianni-Purcell's Theorem 1 \cite{FKP08}, if $M$ is obtained from the complement of a hyperbolic link $\mathcal{L}$ in $S^3$ by Dehn filling along a boundary curve $\gamma$ on one component of the boundary of $S^3\setminus \mathcal{L}$, then 
\begin{align}
    Vol(M)\geq \left(1-\left(\frac{2\pi}{L(\gamma)}\right)^2\right)^{\frac{3}{2}}Vol(S^3\setminus \mathcal{L}),
\end{align}
    where $L(\gamma)$ is the length of $\gamma$ in the induced Euclidean metric on the boundary of the embedded horoball neighborhood of the cusp. For the Whitehead link complement $S^3\setminus W$, the boundary of the maximal horoball neighborhood is a tiling of eight isosceles right triangle of hypotenuse  $\sqrt{2}$ (c.f. Neumann-Reid's paper \cite{NR90}). As a consequence, the meridian $m$ and the longitude $l$ gives $L(m)=\sqrt{2}$ and $L(l)=4$. The angle between $m$ and $l$ is $\frac{\pi}{4}$  Therefore, 
    \begin{align}
       L(\gamma)&=L(pm+ql)=\sqrt{(\sqrt{2}p)^2+(4q)^2+2\frac{\sqrt{2}}{2}(\sqrt{2}p)(4q)}\\\nonumber
       &=\sqrt{2p^2+16q^2+8pq}=\sqrt{2((p+2q)^2+4q^2)}. 
    \end{align}
So we obtain 
\begin{align}
   Vol(W(p,q))\geq \left(1-\frac{2\pi^2}{(p+2q)^2+4q^2}\right)^{\frac{3}{2}}Vol(S^3\setminus W)
\end{align}
\end{proof}

\begin{corollary} \label{corollary-pq}
Let  
\begin{align} \label{set-S}
  S=\{(p,q)\in \mathbb{Z}^2|p,q \ 
 \text{relatively prime}\}-T_1\cup T_2\cup T_3\cup T_4,  
\end{align}
where
$$T_1=\{(-9,1),(-8,1),(-7,1),(-6,1),(-5,1),(1,1),(2,1),(3,1),(4,1),(5,1)\}$$
$$T_{2}=\{(-11,2),(-9,2),(-7,2),(-5,2),(-3,2),(-1,2),(1,2),(3,2)\}$$
$$T_3=\{(-11,3),(-10,3),(-8,3),(-7,3),(-5,3),(-4,3),(-2,3),(-1,3)\}$$
$$T_4=\{(-9,4),(-7,4)\}.$$

Then for $(p,q)\in S$, we have
\begin{align}
Vol(W(p,q))>3.374482.    
\end{align}
\end{corollary}
\begin{proof}
    By Theorem \ref{theorem-volestimate}, we have $Vol(W(p,q))>3.374482$ for $
(p+2q)^2+4q^2\geq 370$. Then compute the volume for the remained finite cases by Snappy, we obtain the Corollary.
\end{proof}

\subsection{One-dimensional saddle point method}  \label{subsection-onedim}
\subsubsection{The case $(s^\pm,m_1^\pm)$}
Recall that
\begin{align} 
&V^\pm(p,q;\theta_1,\theta_2)\\\nonumber
&=V(s^{\pm},m_1^{\pm},1,\theta_1,\theta_2)+\left(\frac{K(s^{\pm})}{2}+\frac{p'}{2q}\right)\pi\sqrt{-1}\\\nonumber
&=\pi\sqrt{-1}\left(\left(\frac{p}{2q}+1\right)\theta_1^2\mp\frac{\theta_1}{q}+\theta_2^2-\theta_2+\frac{5}{6}+\frac{p'}{2q}\right)\\\nonumber
&+\frac{1}{2\pi\sqrt{-1}}\left(\text{Li}_2(e^{2\pi\sqrt{-1}(\theta_1+\theta_2)})+\text{Li}_2(e^{2\pi\sqrt{-1}(-\theta_1+\theta_2)})+3\text{Li}_2(e^{2\pi\sqrt{-1}\theta_2})-\text{Li}_2(e^{4\pi\sqrt{-1}\theta_2})\right).
\end{align}
For $\theta_1=c\in (-\frac{1}{2},\frac{1}{2})$, we have
\begin{align}
    &V^+(p,q;c,\theta_2)=\pi\sqrt{-1}\left(\left(\frac{p}{2q}+1\right)c^2\mp\frac{c}{q}+\theta_2^2-\theta_2+\frac{5}{6}+\frac{p'}{2q}\right)\\\nonumber
&+\frac{1}{2\pi\sqrt{-1}}\left(\text{Li}_2(e^{2\pi\sqrt{-1}(c+\theta_2)})+\text{Li}_2(e^{2\pi\sqrt{-1}(-c+\theta_2)})+3\text{Li}_2(e^{2\pi\sqrt{-1}\theta_2})-\text{Li}_2(e^{4\pi\sqrt{-1}\theta_2})\right).
\end{align}
Hence
\begin{align}
    \frac{dV^+}{d\theta_2}&=\pi\sqrt{-1}\left(-1+2\theta_2\right)-\log(1-e^{2\pi\sqrt{-1}(\theta_2-c)})-\log(1-e^{2\pi\sqrt{-1}(\theta_2+c)})\\\nonumber
    &-3\log(1-e^{2\pi\sqrt{-1}\theta_2})+2\log(1-e^{4\pi\sqrt{-1}\theta_2})
\end{align}

The critical point equation $\frac{dV^+}{d\theta_2}=0$ leds to 
\begin{align} \label{formula-equationC}
    \left(C+\frac{1}{C}+3\right)z_2^2-\left(C+\frac{1}{C}\right)z_2+1=0.
\end{align}
where we have let $z_2=e^{2\pi\sqrt{-1}\theta_2}$ and $C=e^{2\pi\sqrt{-1}c}$. Clearly, $C+\frac{1}{C}=2\cos(2\pi c)$, hence the equation (\ref{formula-equationC}) becomes 
\begin{align}
(2\cos(2\pi c)+3)z_2^2-2\cos (2\pi c)z_2+1=0.    
\end{align}
Solving this quadratic equation, we obtain
\begin{align}
z_2=\frac{1-2\sin^2(\pi c)\pm 2\sqrt{-1}\sqrt{1-\sin^4(\pi c)}}{5-4\sin^2(\pi c)}.     
\end{align}
By using the condition that $c\in (-\frac{1}{2},\frac{1}{2})$ and $\text{Re}(\theta_2)\in [0,\frac{1}{2})$, we have that 
\begin{align} \label{formula-theta2c}
\theta_2(c)=\frac{\log\left(\frac{1-2\sin^2(\pi c)+2\sqrt{-1}\sqrt{1-\sin^4(\pi c)}}{5-4\sin^2(\pi c)}\right)}{2\pi\sqrt{-1}}    
\end{align}
is the unique solution to the one-dimensional critical point equation $\frac{dV^+}{d\theta_2}=0$. 

We compute the derivative
\begin{align}
 &\frac{d\text{Re}(V^+(p,q;c,\theta_2(c)))}{dc}\\\nonumber
 &=\text{Re}\left(\frac{dV^+(p,q;c,\theta_2(c))}{dc}\right)\\\nonumber
 &=\text{Re}\left(\frac{\partial V^+}{\partial \theta_2}\frac{d\theta_2(c)}{dc}+\frac{\partial V^+}{\partial \theta_1}\frac{dc}{dc}\right)\\\nonumber
 &=\text{Re}\left(\frac{\partial V^+}{\partial \theta_1}(p,q;c,\theta_2(c))\right)\\\nonumber
 &=\text{Re}\left(\pi \sqrt{-1}\left(\frac{p}{q}c+\left(1-\frac{1}{q}\right)-2\theta_2(c)\right)\right. \\\nonumber
 &\left.+\log(1-e^{2\pi\sqrt{-1}(\theta_2(c)-c)})- \log(1-e^{-2\pi\sqrt{-1}(\theta_2(c)+c)})\right).  
\end{align}

Since $e^{2\pi\sqrt{-1}c}=C$, $e^{2\pi\sqrt{-1}\theta_2(c)}=z_2(c)$, and by using 
\begin{align}
C^{-1}&=\cos(2\pi c)-\sqrt{-1}\sin(2\pi c), \\\nonumber
z_2(c)^{-1}&=\cos(2\pi c)-2\cos(\pi c)\sqrt{-1}\sqrt{1+\sin^2(\pi c)},
\end{align}
we obtain 
\begin{align}
 &\frac{d\text{Re}(V^+(p,q;c,\theta_2(c)))}{dc}\\\nonumber
 &=\text{Re}\left(-2\pi\sqrt{-1} \theta_2(c)+\log\left(1-z_2(c)C^{-1}\right)-\log\left(1-z_2(c)^{-1}C^{-1}\right)\right)\\\nonumber
 &=\text{Re}\left(\log\left(\frac{1}{z_2(c)}\right)+\log\left(\frac{1-z_2(c)C^{-1}}{1-z_2(c)^{-1}C^{-1}}\right)\right)\\\nonumber
 &=\text{Re}\log\left(\frac{z_2(c)^{-1}-C^{-1}}{C-z_2(c)^{-1}}\right)\\\nonumber
 &=\text{Re}\left(\log\left(\frac{\sin(\pi c)-\sqrt{\sin^2(\pi c)+1}}{\sin(\pi c)+\sqrt{\sin^2(\pi c)+1}}\right)\right)\\\nonumber
 &=2\log\left(\sqrt{\sin^2(\pi c)+1}-\sin(\pi c)\right).
\end{align}

For $0<c<\frac{1}{2}$, we have $\sqrt{\sin^2(\pi c)+1}-\sin(\pi c)<0$, hence when $c_0\leq c\leq \frac{1}{4}$, we obtain 
\begin{align} \label{formula-oneVleq}
2\pi \text{Re}V^+(p,q;c,\theta_2(c))&< 2\pi \text{Re}V^+(p,q;c_0,\theta_2(c_0))\\\nonumber
& =2\pi \text{Re}V^+(0.122532,.1946407106+.1185471546\sqrt{-1},p,q)\\\nonumber &=3.3744812 \leq  3.374482.
\end{align}

By formula (\ref{formula-fzz}), we have
\begin{align}
 f_{X_2X_2}&=   2\pi\left(\frac{\sin(2\pi(\theta_{1R}+\theta_{2R}))}{e^{2\pi (X_2+X_1)}+e^{-2\pi (X_2+X_1)}-2\cos(2\pi(\theta_{1R}+\theta_{2R}))}\right.\\\nonumber
 &\left.+\frac{\sin(2\pi(\theta_{2R}-\theta_{1R}))}{e^{2\pi (X_2-X_1)}+e^{-2\pi (X_2-X_1)}-2\cos(2\pi(\theta_{2R}-\theta_{1R}))}\right.\\\nonumber
 &\left.+\frac{\sin(2\pi \theta_{2R})(3e^{2\pi X_1}+3e^{-2\pi X_2}-2\cos(2\pi \theta_{2R}))}{e^{4\pi X_2}+e^{-4\pi X_2}-2\cos(4\pi \theta_{2R})}\right),
\end{align}
hence
\begin{align}
f_{X_2X_2}|_{\theta_{1R}=c,X_1=0}&=2\pi\left(\frac{\sin(2\pi(c+\theta_{2R}))}{e^{2\pi X_2}+e^{-2\pi X_2}-2\cos(2\pi(c+\theta_{2R}))}\right.\\\nonumber
 &\left.+\frac{\sin(2\pi(\theta_{2R}-c))}{e^{2\pi X_2}+e^{-2\pi X_2}-2\cos(2\pi(\theta_{2R}-c))}\right.\\\nonumber
 &\left.+\frac{\sin(2\pi \theta_{2R})(3e^{2\pi X_2}+3e^{-2\pi X_2}-2\cos(2\pi \theta_{2R}))}{e^{4\pi X_2}+e^{-4\pi X_2}-2\cos(4\pi \theta_{2R})}\right).
\end{align}

\begin{lemma} \label{lemma-onehess}
Given $0\leq c<\frac{1}{4}$, we have     \begin{align}
f_{X_2X_2}|_{\theta_{1R}=c,X_1=0}>0 \ \text{for} \ c<\theta_{2R}<\frac{1}{2}.  
\end{align}
\end{lemma}
\begin{proof}
Clearly, 
\begin{align}
  \frac{\sin(2\pi \theta_{2R})(3e^{2\pi X_2}+3e^{-2\pi X_2}-2\cos(2\pi \theta_{2R}))}{e^{4\pi X_2}+e^{-4\pi X_2}-2\cos(4\pi \theta_{2R})}>0  
\end{align}
If $0\leq c<\frac{1}{4}$ and $c<\theta_{2R}<\frac{1}{2}$, then 
$  \sin(2\pi \theta_{2R})>0,   $ and
\begin{align}
    \cos(2\pi c)(e^{2\pi X_2}+e^{-2\pi X_2})-2\cos(2\pi \theta_{2R})>2\cos(2\pi c)-2\cos(2\pi \theta_{2R})>0,
\end{align}
Hence
\begin{align}
    &\frac{\sin(2\pi(c+\theta_{2R}))}{e^{2\pi X_2}+e^{-2\pi X_2}-2\cos(2\pi(c+\theta_{2R}))}+\frac{\sin(2\pi(\theta_{2R}-c))}{e^{2\pi X_2}+e^{-2\pi X_2}-2\cos(2\pi(\theta_{2R}-c))}\\\nonumber
    &=2\sin(2\pi \theta_{2R})\frac{\cos(2\pi c)(e^{2\pi X_2}+e^{-2\pi X_2})-2\cos(2\pi \theta_{2R})}{(e^{2\pi X_2}+e^{-2\pi X_2}-2\cos(2\pi(c+\theta_{2R})))(e^{2\pi X_2}+e^{-2\pi X_2}-2\cos(2\pi(\theta_{2R}-c)))}>0.
\end{align}
So we complete the proof of the Lemma. 
\end{proof}

We define

\begin{align}
    D''_0(c)=\{(c,z)\in D''_0|c_0\leq  c\leq \frac{1}{4} \}
\end{align}

\begin{lemma}
For $\theta_{1R}\in D''_0(c)$, we have 
\begin{align}
    \text{Re}V^+(p,q;c,\theta_2) \  \text{goes to $\infty$ uniformly, as } \ X_2^2\rightarrow \infty.  
\end{align}
\end{lemma}
\begin{proof}

 We introduce the function 
\begin{equation} 
F(\theta_{2R},X)=\left\{ \begin{aligned}
         &\left(\frac{1}{2}-\theta_{2R}\right)X_2  &  \ (\text{if} \ X_2\geq 0) \\
         &\left(2\theta_{2R}-\frac{3}{2}\right)X_2 & \ (\text{if} \ X_2<0)
                          \end{aligned} \right.
                          \end{equation}   
since $0<\theta_{2R}<\frac{1}{2}$, we have 
\begin{align}
    F(\theta_{2R},X_2)\rightarrow \infty \ \text{as} \ X_2^2\rightarrow \infty.
\end{align}
By Lemma \ref{lemma-varphixi3}, we obtain 
\begin{align}
2\pi F(\theta_{2R},X_2)-C < \text{Re} V^+(p,q;c,\theta_{2})<2\pi F(\theta_{2R},X_2)+C,
\end{align}
which implies the Lemma.
\end{proof}

\begin{proposition} \label{prop-D0'c}
        There exists a constant $C$ which is independent of $c$, such that 
    \begin{align}
        |\int_{D''_0(c)}e^{(N+\frac{1}{2})V^\pm(p,q;c,\theta_2)}d\theta_2|<Ce^{(N+\frac{1}{2})(\zeta_{\mathbb{R}}(p,q)-\epsilon)}. 
    \end{align}
    \end{proposition}
\begin{proof}
    
Let $\theta_2(c)=\theta_{2R}(c)+X_2(c)\sqrt{-1}$ be the unique solution given by formula (\ref{formula-theta2c}) to the one-dimensional critical point  equation $\frac{dV^+}{d\theta_2}=0$, we set
\begin{align}
    \Pi^c=\Pi^c_{top}\cup \Pi^c_{side}
\end{align}
where 
\begin{align}
    \Pi^c_{top}=\{(c,\theta_{2R}+X_2(c)\sqrt{-1})\in \mathbb{C}^2|(c,\theta_{2R})\in D''_{0}(c)\}
\end{align}
  and 
  \begin{align}
      \Pi_{side}^c=\{(c,\theta_{2R}(c)+X_2\sqrt{-1})\in \mathbb{C}^2|(c,\theta_{2R})\in \partial D''_{0}(c), X_2\in [0,X_2(c)] \ ( \text{or $[X_2(c),0]$ if $X_2(c)<0$ }) \}. 
  \end{align}
By Proposition \ref{Prop-Hessian}  and $\text{Re}V^+$ is a harmonic function, we have that 
\begin{align}
 \text{Re} V^+(c,\theta_{2R}+X_2(c)\sqrt{-1})<\text{Re} V^+(c,\theta_2(c))<\zeta_\mathbb{R}(p,q)-\epsilon
\end{align}
 on $\Pi_{top}\setminus \{(c,\theta_2(c))\}$, where the last "$<$" is by formula (\ref{formula-oneVleq}).  On the other hand,
since $\text{Re} V^+(p,q;\theta_1,\theta_2)<\zeta_{\mathbb{R}}(p,q)-\epsilon$ for $(\theta_1,\theta_2)\in \partial D_0$,  by Lemma \ref{lemma-onehess}, we obtain $\text{Re} V^+(p,q;c,\theta_2)<\zeta_{\mathbb{R}}(p,q)-\epsilon$ on $\Pi_{side}$. 

Therefore, by one dimensional saddle point formula, the integral 
\begin{align}  
        \int_{D''_0(c)}e^{(N+\frac{1}{2})V^+(p,q;c,\theta_2)}d\theta_2=\int_{\Pi^c}e^{(N+\frac{1}{2})V^+(p,q;c,\theta_2)}d\theta_2\leq C e^{(N+\frac{1}{2})(\zeta_\mathbb{R}(p,q)-\epsilon)} 
    \end{align}
where $C$ is a constant independent of $c$. Actually, $C$ can be chosen to be $\max_{c\in [0,\frac{1}{4}]}\{\text{length}(\Pi^c)\}$. 
\end{proof}

\subsubsection{The case $(s,m_1)\neq (s^\pm,m_1^\pm)$}
Note that the above computations are also works for the case $(s,m_1)\neq (s^\pm,m_1^\pm)$, so the Proposition \ref{prop-D0'c} also works for the case  $(s,m_1)\neq (s^\pm,m_1^\pm)$.

\end{document}